\documentclass[11pt,reqno]{article}

\usepackage{mathrsfs}
\usepackage{amssymb}
\usepackage{amsthm}
\usepackage{amsmath,amsfonts,amssymb,esint}
\usepackage{graphics,color}
\usepackage{enumerate}

\usepackage{authblk}

\usepackage{marginnote}
\usepackage{xfrac}
\usepackage{mathtools}
\usepackage{hyperref}

\usepackage{comment}

\newtheorem{theorem}{Theorem}[section]
\newtheorem{lemma}[theorem]{Lemma}
\newtheorem{conjecture}[theorem]{Conjecture}
\newtheorem{proposition}[theorem]{Proposition}

\newtheorem{definition}[theorem]{Definition\rm}
\newtheorem{remark}{Remark}

\newcommand{\T}{\ensuremath{\mathbb{T}}}
\newcommand*{\R}{\ensuremath{\mathbb{R}}}

\newcommand*{\N}{\ensuremath{\mathbb{N}}}
\newcommand*{\Z}{\ensuremath{\mathbb{Z}}}
\newcommand*{\C}{\ensuremath{\mathbb{C}}}
\newcommand*{\supp}{\ensuremath{\mathrm{supp\,}}}

\renewcommand*{\div}{\ensuremath{\mathrm{div\,}}}
\newcommand*{\tr}{\ensuremath{\mathrm{tr\,}}}

\newcommand*{\Id}{\ensuremath{\mathrm{Id}}}

\newcommand{\eps}{\varepsilon}

\newcommand*{\curl}{\ensuremath{\mathrm{curl\,}}}

\renewcommand*{\P}{\ensuremath{\mathcal{P}}}

\newcommand*{\RR}{\ensuremath{\mathcal{R}}}
\newcommand{\norm}[1]{\left\|#1\right\|}
\newcommand{\abs}[1]{\left|#1\right|}

\newcommand{\vertiii}[1]{{\left\vert\kern-0.25ex\left\vert\kern-0.25ex\left\vert #1 
    \right\vert\kern-0.25ex\right\vert\kern-0.25ex\right\vert}}

\begin{document}

\title{Dissipative Euler flows with Onsager-critical spatial regularity}

\author[1]{Tristan Buckmaster\thanks{ERC Grant Agreement No. 277993}}
\author[2]{Camillo De Lellis\thanks{Grant 200020-146349 of the Swiss National Foundation}}
\author[1]{L\'aszl\'o Sz\'ekelyhidi Jr.\thanks{ERC Grant Agreement No. 277993}}

\affil[1]{Universit\"at Leipzig}

\affil[2]{Universit\"at Z\"urich}


\maketitle

\begin{abstract}
For any $\varepsilon >0$ we show the existence of continuous periodic weak solutions $v$ of the Euler equations which do not conserve the kinetic energy and belong to the space $L^1_t (C_x^{\sfrac{1}{3}-\eps})$, namely $x\mapsto v (x,t)$ is $(\sfrac{1}{3}-\eps)$-H\"older continuous in space at a.e. time $t$ and the integral $\int [v(\cdot, t)]_{\sfrac{1}{3}-\eps}\, dt$ is finite. A well-known open conjecture of L. Onsager claims that such solutions exist even in the class $L^\infty_t (C_x^{\sfrac{1}{3}-\eps})$.
\end{abstract}

\section{Introduction}

In what follows, $\T^3$ denotes the $3$-dimensional flat torus, i.e. $\T^3 = \R^3/\Z^3$. We consider $L^2$ functions $v: \T^3 \times [0, 1] \to \mathbb R^3$ for which there is a (distributional) pressure field $p$ such that the Euler equations
\begin{equation}\label{e:Euler}
\left\{\begin{array}{l}
\partial_t v + \div (v\otimes v) + \nabla p =0\\ \\
\div v = 0
\end{array}\right.
\end{equation}
hold (in the sense of distributions). Such $v$ will be called a {\em weak solution} of \eqref{e:Euler}. In some other occasions the pressure field will be a specified function, and $(v,p)$ will again solve \eqref{e:Euler} distributionally: such pairs will also be called {\em weak solutions} of \eqref{e:Euler}.

Given a weak solution $v$, we define its kinetic energy $E:[0,1]\rightarrow \R$ by the formula
\begin{equation}\label{e:energy}
E (t) := \frac{1}{2} \int_{\T^3} |v (x,t)|^2\, dx 
\end{equation}
In the case of smooth solutions to \eqref{e:Euler}, a simple calculation yields the conservation of kinetic energy. This formal calculation however does not necessary hold for weak solutions. This was first demonstrated by V.~Scheffer's construction of a nontrivial weak solution with compact support in space and time in $\R^2$ \cite{Scheffer93}. Subsequently, a different construction (still in the 2-dimensional case) was provided by A.~Shnirelman \cite{Shnirelmandecrease} and then for general dimension $n\geq 2$ by the second and third author in \cite{DS1,DS2}. We note that these constructions lead to bounded but in general discontinuous weak solutions $v$. Furthermore, in \cite{DS1,DS2} the failure of energy conservation (as well as the non-uniqueness) were identified as a weak $h$-principle - we refer to the survey \cite{DSsurvey} for more details. The first construction of a {\it continuous }weak solution where $E(t)$ is not constant, was presented in \cite{DS3} -- moreover, the method of \cite{DS3} was further adapted to show also the non-uniqueness of continuous and H\"older-continuous {\it admissible } weak solutions in \cite{Daneri} (see also \cite{Choffrut} for associated $h$-principle statements). 

The relation between non-uniqueness and non-conservation of the energy for various classes of weak solutions is not yet clear. In this paper we will concentrate on the latter, primarily because of its relevance to the Kolmogorov-Onsager theory of fully developed turbulence in 3D and the dissipation anomaly \cite{Kolmogorov,Onsager}. We refer to \cite{FrischBook} for an excellent exposition of the K41 theory and to \cite{EyinkSreenivasan} for a survey of Onsager's contributions. In 1949 L. Onsager famously made the following conjecture \cite{Onsager}:
\begin{conjecture}\label{c:Onsager}\ 
\begin{enumerate}
\item[(a)] If $v$ is a continuous solution of \eqref{e:Euler} and there exists an $\eps>0$ such that
\begin{equation}\label{e:Holder-Onsager}
\sup_{t\in [0,1]}\; [v (\cdot, t)]_{\sfrac13-\eps} := \sup_{t\in [0,1]}\;  \sup_{x\neq y} \frac{|v (x,t) - v (y,t)|}{|x-y|^{\sfrac13-\eps}} < \infty\, ,
\end{equation}
then the total kinetic energy is constant: $E(t)=E(0)$.
\item[(b)] For each $\eps>0$ there is a continuous solution $v$ of \eqref{e:Euler} such that
\begin{equation}\label{e:Holder-Onsager-2}
\sup_{t\in [0,1]}\; [v (\cdot, t)]_{\sfrac13-\eps} < \infty
\end{equation}
and the kinetic energy $E$ is not constant.
\end{enumerate}
\end{conjecture}

Statement (a) has been completely settled: a slightly weaker statement was first proved by Eyink in \cite{Eyink} and later the full statement was proven by Constantin, E and Titi in \cite{ConstantinETiti} (cf.\ \cite{RobertDuchon,CCFS2007}). In fact \cite{ConstantinETiti} contains the following stronger statement:
\begin{theorem}\label{t:L^3-conserves}
For any $\eps>0$ any solution $v\in L^3_t (C_x^{\sfrac{1}{3}+\eps})$ of \eqref{e:Euler} preserves the total kinetic energy.
\end{theorem}
We recall that the mixed space-time norms $L^q_t (C^\alpha_x)$ for $q\in [1, \infty]$ and $\alpha\in ]0,1[$ are defined as
\[
\|v\|_{L^q_t (C^\alpha_x)} := \begin{cases}\left(\int_0^1 \left(\|v (\cdot, t)\|_\alpha\right)^q\, dt\right)^{1/q}& q<\infty\\
\sup_{t\in[0,1]}\|v (\cdot, t)\|_\alpha& q=\infty,
\end{cases}
\]
where $\|f\|_\alpha=\|f\|_0+[f]_\alpha$ is the usual H\"older-norm for functions $f:\T^3\to\R^3$, and $v\in L^q_t (C^\alpha_x)$ provided $\|v\|_{L^q_t (C^\alpha_x)}<\infty$.

The focus of this paper is Statement (b), which is still open. Following the construction introduced in \cite{DS3}, H\"older-continuous weak solutions have been constructed in \cite{DS4,Isett,BDS,Buckmaster,Isett3}. There are two types of statements: 
\begin{enumerate}
\item[(i)] In \cite{DS4} and \cite{BDS} weak solutions are constructed with $E(t)=e(t)$ for all $t\in[0,1]$ {\it for any prescribed smooth function} $e=e(t)>0$;
\item[(ii)] In \cite{Isett,Buckmaster,Isett3} a (nontrivial) weak solution {\it with compact support in time } is constructed.
\end{enumerate}
Obviously both statements lead to a weak solution with non-constant energy, thus aiming towards Statement (b) in Onsager's conjecture. Concerning the actual regularity of the solutions, the statements are
\begin{enumerate}
\item[(iii)] In \cite{DS4} $v\in L^{\infty}_t (C^{\sfrac{1}{10}-\eps}_x)$;
\item[(iv)] In \cite{Isett,Isett3} and \cite{BDS} $v\in L^{\infty}_t (C^{\sfrac{1}{5}-\eps}_x)$;
\item[(v)] In \cite{Buckmaster} $v\in L^{\infty}_t (C^{\sfrac{1}{5}-\eps}_x)$ and in addition $v(\cdot, t)\in C^{\sfrac{1}{3}-\eps}_x$ for almost every time $t\in[0,1]$.
\end{enumerate} 
In fact the solutions in (iii)-(iv) have the same H\"older-regularity in time as in space -- it has been shown in \cite{Isett2} that this improved regularity in time is not an artifact of the method of construction, but rather a regularization property of the equations themselves. In view of Theorem \ref{t:L^3-conserves} one could however speculate that the ``threshold'' for the energy conservation should in fact be $L^3_t (C_x^{\sfrac{1}{3}})$. The goal of this note is to show the following

\begin{theorem}\label{t:main}
For any $\eps>0$ there exists a non-trivial continuous weak solution $v: \T^3 \times \mathbb R\to \mathbb R^3$ of \eqref{e:Euler}, with $v\in L^1_t (C_x^{\sfrac{1}{3}-\eps})$ with compact support in time.
\end{theorem}

In line with previous works \cite{DS3,DS4,Isett,BDS,Buckmaster}, the solution will be constructed as a limit of a rather complicated \emph{convex integration} scheme. A key observation of the first author, made in \cite{Buckmaster}, is that the very same approach of \cite{BDS} (with some more careful choice of the parameters) yields better estimates on the H\"older continuity at most times, allowing to reach $C^{\sfrac{1}{3}-\eps}$ almost everywhere.  Building upon this important remark of \cite{Buckmaster}, the principal challenge of the present work is to carefully modify the convex integration scheme presented in \cite{BDS} in order to obtain better time localized estimates (a goal which anyway will be achieved at the expense of sacrificing the \emph{global} H\"older estimate).  In order to prove Theorem \ref{t:L^3-conserves}, these modifications will be required to be far more subtle than those presented in \cite{Buckmaster}. In addition, a significantly more complicated bookkeeping system will become necessary. One important remark is that, although the vast majority of the scheme adheres to the one in \cite{BDS}, there is one relevant difference: for some time intervals we use one tool introduced in \cite{Isett} to smooth carefully the so-called Reynolds stress, respecting the key estimate for its advective derivative. This approach was entirely substituted in \cite{BDS} by another smoothing device. An interesting point is that it seems necessary to use {\em both} approaches in different time regions.

\section{Iteration scheme}\label{s:scheme}

The aim of this section is to introduce the main ingredients required by our iteration scheme.

\subsection{Euler-Reynolds system}\label{ss:euler_reynolds} At each step $q\in \N$ we construct a triple $(v_q, p_q, \mathring{R}_q)$ of smooth compactly supported functions which solve the Euler-Reynolds system 
(see \cite[Definition 2.1]{DS3}):
\begin{equation}\label{e:euler_reynolds}
\left\{\begin{array}{l}
\partial_t v_q + \div (v_q\otimes v_q) + \nabla p_q =\div\mathring{R}_q\\ \\
\div v_q = 0\, ,
\end{array}\right.
\end{equation} 
where $\mathring{R}_q$ is a $3\times 3$ symmetric traceless tensor. The pair $(v, p)$ of Theorem \ref{t:main} will be the uniform limits of $v_q$ and $p_q$, whereas $\mathring{R}_q$ converges uniformly to $0$ as $q\to\infty$. The difference $v_q-v_{q-1}$ will be denoted by $w_q$.

\subsection{Parameters of regularity}\label{ss:lambda-betas} The principal parameter for measuring the regularity of the pair $(v_q, p_q)$ is an integer valued \emph{frequency parameter} $\lambda_q$ which blows-up as a double exponential as $q\to\infty$. In particular, there is a $\lambda_0$ (sufficiently large) and a $b$ close to but slightly larger than $1$ (the size of $b-1$ is in fact constrained by the parameter $\eps>0$ from Theorem \ref{t:main}) such that
\begin{equation}\label{e:lambda}
\lambda_q \in \left[\lambda_0^{b^q}, 2 \lambda_0^{b^q}\right]\, .
\end{equation}
The exponent $\eps$ in Theorem \ref{t:main} is also related to two exponents $0 <\beta_0 < \beta_\infty< \frac{1}{3}$, which are the endpoints of a sequence of increasing positive exponents $\beta_j$ defined by the recursive relation
\begin{equation}\label{e:betas}
b (\beta_{j+1} - \beta_\infty) = \beta_j - \beta_\infty \, ,
\end{equation}
that is
\begin{equation}\label{e:betas_again}
\beta_j = \frac{\beta_0}{b^j} + \left(1 - \frac{1}{b^j}\right) \beta_\infty\, .
\end{equation}
For notational convenience we also introduce the exponent
\begin{equation}\label{e:beta-1}
\beta_{-1} = b \beta_0 + (1-b) \beta_\infty\, ,
\end{equation}
which we also assume to be positive.
The parameters $\beta_\infty$ and $\beta_0$ should be thought of, respectively, \emph{approximate} upper and lower bounds for the H\"older regularity exponent of the final velocity field $v$ at any time $t$.   

\subsection{Subdivision of the time interval}\label{ss:time_intervals}

We start with a division of the time interval $[0,1]$ into finitely many closed intervals $I^{(q)}_\alpha$, $\alpha \in \{0, 1, \ldots, N (q)+1\}$, where each pair of closed segments $I^{(q)}_\alpha$ and $I^{(q)}_{\alpha+1}$ will intersect at one single point: the right endpoint of $I^{(q)}_\alpha$, which is the left endpoint of $I^{(q)}_{\alpha+1}$. The number $N(q) +2$ denotes the total number of intervals. The precise value will not play a role.

To each $I^{(q)}_\alpha$ we will associate a natural number $j_q (\alpha)\in \{0,1,\dots,q\}$. We require lower bounds on the size of each interval $I^{(q)}_\alpha$ and upper bounds on the total measure of all intervals with  $j_q (\alpha)=j$ for fixed $j$. We begin with the lower bound. 
We will require that the intervals $I^{(q)}_\alpha$ be large enough in order that they may be subdivided into intervals of length $\approx \mu^{-1}_{q+1, j_q (\alpha)}$, where the parameter $\mu_{q+1, j_q (\alpha)}$ is defined by the formula
\begin{equation}\label{e:mu}
\mu_{q+1, j} =
\left\{\begin{array}{ll}
\lambda_{q+1}^{1-\beta_j}&\qquad\mbox{for $j\geq 2$}\\
\lambda_{q+1}^{(1-\beta_0)\frac{(b+1)}{2b} + \frac{(b-1)}{2}\beta_\infty}&\qquad\mbox{for $j\leq 1$.}
\end{array}\right.
\end{equation}
More precisely, we impose the constraint
\begin{equation}\label{e:interval_constraint}
|I^{(q)}_\alpha|\geq \frac{4}{\mu_{q+1, j_q (\alpha)}}\, .
\end{equation}

As for the upper bound, define the regions:
\begin{equation}\label{e:regions_V}
V^{(q)}_j = \bigcup_{\alpha\in\{1,\dots,N(q)\} : j_q (\alpha) =j} I^{(q)}_\alpha\, . 
\end{equation}
In these regions we will claim several (inductive) estimates on the triple $(v_q,p_q,\mathring{R}_q)$ (see Section \ref{ss:estimates} below). Notice that 
\begin{equation}\label{e:regions_increase}
V^{(q)}_j = \emptyset \qquad \mbox{for every $j>q$,}
\end{equation} 
and the intersection of two distinct sets in the collection $\{V^{(q)}_j:\,j=0\dots q\}$ consists of (at most) a finite number of points and it is a subset of the set of endpoints of the intervals $I^{(q)}_\alpha$.

Upon the (Lebesgue) measure of each region $V^{(q)}_j$ we require
\begin{equation}\label{e:regions_est}
\Bigl|\bigcup_{i=0}^jV^{(q)}_i\Bigr| \leq \lambda_0 \lambda_{q+1}^{\beta_j - \beta_\infty+\eps/4}\, .
\end{equation}

\subsection{Inductive estimates}\label{ss:estimates} 

In order to ensure the convergence of the sequence $(v_q, p_q)$ to a solution $(v,p)$ of the Euler equation satisfying the regularity condition $v\in L^1 (C^{\sfrac{1}{3}-\eps})$, we will require a series of inductive estimates on the triple $(v_q,p_q,\mathring{R}_q)$ along the iteration. There are two sets of estimates. One set will be local in time, i.e. depend on the specific time interval $I^{(q)}_\alpha$, and the second set of estimates will be global in time, i.e. hold uniformly for all $t\in [0,1]$.

\subsubsection{Local estimates}
Fix $j=0,1\dots,q$ and consider intervals $I^{(q)}_\alpha$ with $j_{q}(\alpha)=j$. We assume on the triple $(v_q,p_q,\mathring{R}_q)$ for all $t\in I^{(q)}_\alpha$:
\begin{align}
&\lambda_q^{-2} \|v_q (t)\|_2 + \lambda_q^{-1} \|v_q (t)\|_1  \leq M \lambda_q^{-\beta_{(j-1)_+}}\,, \label{e:velocity_est}\\
&\lambda_q^{-2}\|p_q (t)\|_2 + \lambda_q^{-1}\|p_q (t) \|_1  \leq M^2 \lambda_q^{-2\beta_{(j-1)_+}}\,,\label{e:pressure_est}\\
&\lambda_q^{-2} \|\mathring{R}_q (t)\|_2 + \lambda_q^{-1} \|\mathring{R}_q (t)\|_1 + \|\mathring{R}_q (t)\|_0 \leq \lambda_{q+1}^{-2\beta_j}\,,\label{e:R_est}\\
&\|(\partial_t +v_q\cdot \nabla) \mathring{R}_q (t)\|_0 \leq \lambda_q^{1-\beta_{j-1}} \lambda_{q+1}^{-2\beta_j}\, .\label{e:D_tR_est}
\end{align}
Here $(j-1)_+$ equals $j-1$ for all $j\in \mathbb N$ larger than $0$ and  $(0-1)_+ = 0$; $M$ is a geometric constant which is independent of $q$ and of all parameters introduced so far. Its value, however, will be specified much latter, in the proof of
Lemma \ref{l:ugly_lemma3}, cf.~\eqref{e:determines_M}.
Notice that, since $j\mapsto \beta_j$ is monotonic increasing, the estimates \eqref{e:velocity_est}-\eqref{e:D_tR_est} in fact hold for all $t\in \bigcup_{i\geq j}V^{(q)}_i$.

\smallskip

\subsubsection{Global estimates}\label{sss:global}
In order to ensure the convergence to a solution with compact temporal support, we impose the following constraint on the (temporal) support of the triple
\begin{equation}\label{e:spt_triple}
\supp (v_q,p_q,\mathring R_q)(x,\cdot)\subset \bigcup_{1\leq \alpha\leq N(q)} I^{(q)}_\alpha \subset [2^{-q-2},1-2^{-q-2}]
\end{equation}
for all $x\in\T^3$. In particular the role of the two intervals $I^{(q)}_0$ and $I^{(q)}_{N(q)+1}$ is only to identify 
the portion of the time interval $[0,1]$ where we know that the solution vanishes identically.

Consequently, since
$$
\bigcup_{j=0}^qV_j^{(q)}=\bigcup_{1\leq\alpha\leq N(q)}I^{(q)}_\alpha,
$$
the estimates \eqref{e:velocity_est}-\eqref{e:D_tR_est} for the case $j=0$ hold in fact for all $t\in[0,1]$. Therefore, for clarity of the presentation we repeat them here:
\begin{align}
&\lambda_q^{-2} \|v_q (t)\|_2 + \lambda_q^{-1} \|v_q (t)\|_1  \leq M \lambda_q^{-\beta_{0}}\,, \label{e:g-velocity_est}\\
&\lambda_q^{-2}\|p_q (t)\|_2 + \lambda_q^{-1}\|p_q (t) \|_1  \leq M^2 \lambda_q^{-2\beta_{0}}\,,\label{e:g-pressure_est}\\
& \lambda_q^{-2}\|\mathring{R}_q (t)\|_2 + \lambda_q^{-1} \|\mathring{R}_q (t)\|_1 + \|\mathring{R}_q (t)\|_0 \leq \lambda_{q+1}^{-2\beta_0}\,,\label{e:g-R_est}\\
&\|(\partial_t +v_q\cdot \nabla) \mathring{R}_q (t)\|_0 \leq \lambda_q^{1-\beta_{-1}} \lambda_{q+1}^{-2\beta_0}\, .\label{e:g-D_tR_est}
\end{align}

These estimates will be complemented with the uniform estimates
\begin{align}
&\|v_q\|_0 \leq M + M \sum_{i=0}^q \lambda_i^{-\beta_0}\label{e:C0_v_global}\\
&\|p_q\|_0 \leq M^2 + M^2 \sum_{i=0}^q \lambda_i^{-2\beta_0}\label{e:C_p_global}\, .
\end{align}

\subsection{The main Proposition of the iterative procedure} 

We are now in a position to state the main proposition which will enable us to perform the iteration step, from which we will conclude Theorem \ref{t:main}.

\begin{proposition}\label{p:inductive_step}  Let $\eps$ be as in Theorem \ref{t:main} and assume the positive parameters  $\lambda_0>1$, $b>1$ and $\beta_{-1}=b\beta_0+(1-b)\beta_{\infty}<\beta_{\infty}<\sfrac13$ satisfy the following constraints:
\begin{align}
&1 -3b(\beta_0+\beta_{\infty})>0\label{e:condition}\\
&5\beta_{\infty} > b(1+3\beta_0)\, ,\label{e:condition2}\\
&\mbox{$\lambda_0$ is sufficiently large, depending only upon $b, \beta_0, \beta_\infty$ and $\eps$.}
\end{align}
Let $(v_q, p_q, \mathring{R}_q)$ be a triple which solves the Euler-Reynolds system in $\T^3\times \mathbb R$ and $\{I^{(q)}_\alpha\}$ a subdivision of $[0,1]$ in closed time intervals which satisfy the assumptions of the Sections \ref{ss:time_intervals} and \ref{ss:estimates}. Then there is a second triple $(v_{q+1}, p_{q+1}, \mathring{R}_{q+1})$ which solves the Euler-Reynolds system in $\T^3\times \mathbb R$, together with a subdivision $\{I^{(q+1)}_\alpha\}$ satisfying the very same requirements. In addition we have
\begin{align}
&\|v_{q+1}(\cdot, t)-v_q (\cdot, t)\|_0 \leq M \lambda_{q+1}^{-\beta_{(j-1)_+}} \qquad \forall t\in V^{(q+1)}_j\, ,\label{e:w_est}\\
&\|p_{q+1} (\cdot, t) - p_q (\cdot, t)\|_0 \leq M^2 \lambda_{q+1}^{-\beta_{(j-1)_+}} \qquad \forall t\in V^{(q+1)}_j\, .\label{e:p_difference_est}
\end{align} 
\end{proposition}

The proof of Theorem \ref{t:main} will be based on Proposition \ref{p:inductive_step}, starting from a nontrivial solution $(v_0, p_0, \mathring{R}_0)$ of \eqref{e:euler_reynolds}. \eqref{e:w_est} will then ensure that the limiting pair $(v.p)$ reached by the sequence constructed wth the help of Proposition \ref{p:inductive_step} is nontrivial. \eqref{e:w_est} together with \eqref{e:regions_est} will provide the key bound in the space $L^1 (C^{\sfrac{1}{3}-\eps})$.

\section{The inductive construction}\label{s:new_triple}

In order to commence the proof of Proposition \ref{p:inductive_step}, in this section we will detail the inductive  construction of the tuple $(v_{q+1},p_{q+1},\mathring R_{q+1})$ from $(v_{q},p_{q},\mathring R_{q})$. Before starting specifying the definition of the new tuple, we need several preliminary lemmas.

\subsection{Preliminaries}\label{ss:old_preliminaries} In this paper we denote by $\R^{n\times n}$, as usual, the space of $n\times n$ matrices, whereas $\mathcal{S}^{n\times n}$
and $\mathcal{S}^{n\times n}_0$ denote, respectively, the corresponding subspaces of symmetric matrices 
and of trace-free symmetric matrices. The $3\times 3$ identity matrix will be denoted with $\Id$. 
For definitiveness we will use the matrix operator norm $|R|:=\max_{|v|=1}|Rv|$. Since we will
deal with symmetric matrices, we have the identity $|R|= \max_{|v|=1} |Rv \cdot v|$.

\begin{proposition}[Beltrami flows]\label{p:Beltrami}
Let $\bar\lambda\geq 1$ and let $A_k\in\R^3$ be such that 
$$
A_k\cdot k=0,\,|A_k|=\tfrac{1}{\sqrt{2}},\,A_{-k}=A_k
$$
for $k\in\Z^3$ with $|k|=\bar\lambda$.
Furthermore, let 
$$
B_k=A_k+i\frac{k}{|k|}\times A_k\in\C^3.
$$
For any choice of $a_k\in\C$ with $\overline{a_k} = a_{-k}$ the vector field
\begin{equation}\label{e:Beltrami}
W(\xi)=\sum_{|k|=\bar\lambda}a_kB_ke^{ik\cdot \xi}
\end{equation}
is real-valued, divergence-free and satisfies
\begin{equation}\label{e:Bequation}
\div (W\otimes W)=\nabla\frac{|W|^2}{2}.
\end{equation}
Furthermore
\begin{equation}\label{e:av_of_Bel}
\langle W\otimes W\rangle= \fint_{\T^3} W\otimes W\,d\xi = \frac{1}{2} \sum_{|k|=\bar\lambda} |a_k|^2 \left( \Id - \frac{k}{|k|}\otimes\frac{k}{|k|}\right)\, .  
\end{equation}
\end{proposition}

The proof of \eqref{e:Bequation}, which is quite elementary
(see also \cite{DS3}), is based on the following algebraic identity, which we state separately for future reference:
\begin{lemma}\label{l:BkBk'}
Let $k,k'\in \Z^3$ with $|k|=|k'|=\bar{\lambda}$ and let $B_k,B_{k'}\in \C^3$ be the associated vectors from Proposition \ref{p:Beltrami}. Then we have
\begin{equation}\label{e:BkBk'}
(B_k\otimes B_{k'}+B_{k'}\otimes B_k)(k+k')=(B_k\cdot B_{k'})(k+k').
\end{equation}
\end{lemma}

Another important ingredient is the following geometric lemma, also taken from \cite{DS3}.

\begin{lemma}[Geometric Lemma]\label{l:split}
For every $N\in\N$ we can choose $r_0>0$ and $\bar{\lambda} > 1$ with the following property.
There exist pairwise disjoint subsets 
$$
\Lambda_j\subset\{k\in \Z^3:\,|k|=\bar{\lambda}\} \qquad j\in \{1, \ldots, N\}
$$
and smooth positive functions 
\[
\gamma^{(j)}_k\in C^{\infty}\left(B_{r_0} (\Id)\right) \qquad j\in \{1,\dots, N\}, k\in\Lambda_j
\]
such that
\begin{itemize}
\item[(a)] $k\in \Lambda_j$ implies $-k\in \Lambda_j$ and $\gamma^{(j)}_k = \gamma^{(j)}_{-k}$;
\item[(b)] For each $R\in B_{r_0} (\Id)$ we have the identity
\begin{equation}\label{e:split}
R = \frac{1}{2} \sum_{k\in\Lambda_j} \left(\gamma^{(j)}_k(R)\right)^2 \left(\Id - \frac{k}{|k|}\otimes \frac{k}{|k|}\right) 
\qquad \forall R\in B_{r_0}(\Id)\, .
\end{equation}
\end{itemize}
\end{lemma}

Following \cite{DS3}, we introduce the following operator
in order to deal with the Reynolds stresses.

\begin{definition}\label{d:reyn_op}
Let $v\in C^\infty (\T^3, \R^3)$ be a smooth vector field. 
We then define $\RR v$ to be the matrix-valued periodic function
\begin{equation*}
\RR v:=\frac{1}{4}\left(\nabla\P u+(\nabla\P u)^T\right)+\frac{3}{4}\left(\nabla u+(\nabla u)^T\right)-\frac{1}{2}(\div u) \Id,
\end{equation*}
where $u\in C^{\infty}(\T^3,\R^3)$ is the solution of
\begin{equation*}
\Delta u=v-\fint_{\T^3}v\textrm{ in }\T^3
\end{equation*}
with $\fint_{\T^3} u=0$ and $\P$ is the Leray projection onto divergence-free fields with zero average.
\end{definition}

\begin{lemma}[$\RR=\textrm{div}^{-1}$]\label{l:reyn}
For any $v\in C^\infty (\T^3, \R^3)$ we have
\begin{itemize}
\item[(a)] $\RR v(x)$ is a symmetric trace-free matrix for each $x\in \T^3$;
\item[(b)] $\div \RR v = v-\fint_{\T^3}v$.
\end{itemize}
\end{lemma}

\bigskip

\subsection{New intervals}\label{ss:new_intervals} 

We now describe the inductive procedure in order to define the new time intervals $I^{(q+1)}_\alpha$ in terms of the old intervals $I^{(q)}_\alpha$.  In addition we will describe a partition of unity of time which will provide a crucial ingredient to the construction of the perturbation $w_{q+1}$.

We begin by subdividing each interval $I^{(q)}_\alpha$ for $\alpha\in \{1, \ldots , N (q)\}$ into further subintervals $J^{(q+1)}_{\alpha, \alpha'}$, based on the set of parameters $\mu_{q+1, j_q (\alpha)}$:
\begin{itemize}
\item We let $n (\alpha,q)$ be the largest integer smaller than $\mu_{q+1, j_q (\alpha)} |I^{(q)}_\alpha|/2$ and note that the estimate \eqref{e:interval_constraint} ensures that $n (\alpha, q)\geq 2$.
\item We subdivide $I^{(q)}_\alpha$ from left to right in $n (\alpha, q)$ closed intervals $J_{\alpha, \alpha'}$, satisfying the conditions:
\begin{enumerate}
\item The right endpoint of $J_{\alpha, \alpha'}$ coincides with the left endpoint of $J_{\alpha, \alpha'+1}$.
\item The first $n (\alpha, q)-1$ intervals have length exactly $2\mu_{q+1,j_q (\alpha)}^{-1}$.
\item The last interval has length 
\[
|I^{(q)}_\alpha| - 2 (n (\alpha, q) -1)\mu_{q+1, j_q (\alpha)}^{-1}\, ,
\] 
which in particular is bounded below by $2\mu_{q+1,j_q (\alpha)}^{-1}$ and  above by $4\mu_{q+1,j_q (\alpha)}^{-1}$.
\end{enumerate}
\end{itemize}

\bigskip
 
We next relabel the intervals $J_{\alpha, \alpha'}$ as $J_\varsigma$, $\varsigma \in \{1, \ldots N'\}$ and complete the collection of intervals by settings $J_0:=I^{(q)}_{0}$ and $J_{N'+1}:=I^{(q)}_{N(q)+1}$. Therefore  each $\varsigma$ is associated with an index $\alpha_q (\varsigma)$ such that $J_\varsigma \subset I^{(q)}_{\alpha_q (\varsigma)}$.  We then call $J_\varsigma$ an internal interval if it is contained in the interior of $I^{(q)}_{\alpha (\varsigma)}$, otherwise we call it a boundary interval (note that $J_0$ and $J_{N'+1}$ are boundary intervals). 

\smallskip

\begin{definition}\label{d:ov_regions}
For each pair of intervals $J_\varsigma$ and $J_{\varsigma+1}$ with 
$0\leq \varsigma \leq N'$ we define an ``overlapping region'' $K_\varsigma$, in the following way:
Let 
\begin{equation}\label{e:eta}
\eta_{q+1,j} := \lambda_{q+1}^{b\beta_0-(b-1)\beta_\infty - \beta_j}\, 
\end{equation}
for all $j=0,1,2\dots$, and let 
$$
j=j_q (\alpha_q (\varsigma)),\quad j'=j_q (\alpha_q (\varsigma+1)).
$$
\begin{itemize}
\item[(A1)] If $j \leq  j'$, then the region $K_\varsigma$ is a closed interval with left endpoint coinciding with the right endpoint of $J_\varsigma$ and length 
$$
|K_{\varsigma}|=\frac{\eta_{q+1, j}}{\mu_{q+1, j} }.
$$
\item[(A2)] If $j> j'$, then the region $K_\varsigma$ is a closed interval  with right endpoint coinciding with the left endpoint of $J_{\varsigma+1}$ and length 
$$
|K_{\varsigma}|=\frac{\eta_{q+1, j'}}{\mu_{q+1, j'} }.
$$
\end{itemize}
Next, we define the {\em non-overlapping regions} to be the closed segments
\begin{equation}\label{e:nonov_regions}
H_\varsigma := \overline{J_\varsigma\setminus (K_\varsigma\cup K_{\varsigma-1})}\, .
\end{equation}
\end{definition}

We next claim that the overlapping region $K_{\varsigma}$ constructed in Definition \ref{d:ov_regions} is indeed contained in the correct interval $J_{\varsigma}$ (or $J_{\varsigma+1}$ resp.). In fact we will prove a better estimate.

\begin{lemma}\label{l:overlapping-geometry}
Let $K_\varsigma$ be as in Definition \ref{d:ov_regions}. If (A1) holds, namely the left endpoint of $K_\varsigma$ is the right endpoint of $J_\varsigma$, set $J:= J_{\varsigma+1}$, otherwise set $J:= J_\varsigma$. In either case we have
\begin{equation}\label{e:overlapping-geometry}
|K_{\varsigma}|\leq \frac{1}{4}|J|\, 
\end{equation}
and in particular $K_\varsigma \subset J$.

Therefore it follows that each non-overlapping region is nonempty and satisfies
\begin{equation}\label{e:overlapping-geometry2}
|J_\varsigma| \geq |H_{\varsigma}|\geq \frac{1}{2} |J_\varsigma|\, .
\end{equation}
\end{lemma}
\begin{proof} It is obvious that \eqref{e:overlapping-geometry2} follows from \eqref{e:overlapping-geometry}. 
We start by assuming that the interval $J = J_{\varsigma'}$ (which will end up containing $K_\varsigma$) is neither $J_0$ nor $J_{N'+1}$. Observe therefore that $J_{\varsigma'} \subset I^{(q)}_{\alpha_q (\varsigma')}$ with $\alpha_q (\varsigma')\in \{1, \ldots , N_q\}$ and thus, if we define $j'=j_q(\alpha_q (\varsigma'))$ we have
\begin{equation}\label{e:size-of-J}
\frac{2}{\mu_{q+1,j'}}\leq |J|\leq \frac{4}{\mu_{q+1,j'}}\, .
\end{equation}
On the other hand, by the alternatives (A1) and (A2) in Definition \ref{d:ov_regions}, if $j = j_q (\alpha_q (\varsigma))$,
then $j'\geq j$ and
\begin{equation}\label{e:size-of-K}
|K|=\frac{\eta_{q+1,j}}{\mu_{q+1,j}}
\end{equation}
Thus,
$$
\frac{|K|}{|J|}\leq \frac{1}{2}\frac{\mu_{q+1,j'}}{\mu_{q+1,j}}\eta_{q+1,j}.
$$
Then \eqref{e:overlapping-geometry} will follow from 
\begin{equation}\label{e:mu-ineq}
\mu_{q+1,j}\geq \mu_{q+1,j'}\quad\textrm{ for all }j\leq j',
\end{equation}
and the inequality 
\begin{equation}\label{e:trivial_eta}
\eta_{q+1,j}\leq \eta_{q+1,0}=\lambda_{q+1}^{-(b-1)(\beta_\infty-\beta_0)}\leq \frac12.
\end{equation}
The latter follows easily from the fact that $j\mapsto \beta_j$ is increasing and from a sufficiently large choice of $\lambda_0$, whereas for \eqref{e:mu-ineq} we need to consider two cases: (a) $j'>j\geq 2$ and (b) $j'=2>j$. 

The case (a) is obvious from the definition of $\mu_{q+1,j}$ in \eqref{e:mu} since $j\mapsto \beta_j$ is increasing. For the case (b), it suffices to show it for $j=1$. By taking the logarithm we see that \eqref{e:mu-ineq} is equivalent to
\[
(1-\beta_0) \frac{b+1}{2b} + \frac{b-1}{2} \beta_\infty \geq 1 - \beta_2 = 1 - \frac{\beta_0}{b^2} - \left(1-\frac{1}{b^2}\right) \beta_\infty\, ,
\] 
which turns into
\[
 \left(\frac{b-1}{2} + \frac{b^2-1}{b^2}\right) \beta_\infty \geq \frac{b-1}{2b} +\frac{b^2 +b - 2}{2b^2} \beta_0\,.
\]
Factorizing $b-1$ from both sides we are left with the inequality
\[
(b^2+2 b+2) \beta_\infty \geq b + (b+2) \beta_0\, .
\]
Since $b^2 +2b+2 > 5$, the latter inequality is implied by \eqref{e:condition2}.

It remains to examine the case in which $J$ is either $J_0 = I^{(q)}_0 = [0, t_1]$ or $J_{N'+1} = I^{(q)}_{N(q)+1} = [t_2, 1]$. 
From our inductive hypothesis \eqref{e:spt_triple} we have $t_1\geq 2^{-q-2}$  and $t_2\leq 1-2^{-q-2}$.  
The only overlapping region which can be contained in $J_0$ is obviously $K_0$ and we must be in case (A2). Similarly, the only overlapping region which can be contained in $J_{N'+1}$ is $K_{N'}$.
Then \eqref{e:overlapping-geometry} follows from \eqref{e:mu-ineq} and \eqref{e:trivial_eta}:
\[\frac{\eta_{q+1,j}}{\mu_{q+1,j}}\leq \frac{1}{2}\lambda_{q+1}^{\beta_{\infty}-1}\leq 2^{-q-3},\]
where in the last inequality we assume $\lambda_0$ to be sufficiently large.
\end{proof}

The new collection of intervals $\{I^{(q+1)}_{\alpha}\}_{\alpha \in \{0, N (q+1) +1\}}$ is then given by the {\em overlapping regions} $K_\varsigma$ together with the {\em non-overlapping regions} $H_\varsigma$. The intervals $I^{(q+1)}_{\alpha}$ will be ordered in terms of the left endpoints, starting from $I^{(q+1)}_0= H_0 \subset I^{(q)}_0$ and ending with $I^{(q+1)}_{N (q+1)+1} = H_{N'+1} \subset I^{(q)}_{N (q)+1}$.
We next define the map $j_{q+1}$. 

\begin{definition}\label{e:define_j_q+1}
For $\alpha\in \{0,N(q+1)+1\}$ simply set $j_{q+1}(\alpha)=0$.  Now fix $\alpha\in \{1,\dots, N(q+1)\}$. If  $I^{(q+1)}_\alpha \subset I^{(q)}_{\alpha'}$ is a non-overlapping region, then set $j_{q+1} (\alpha) = j_{q} (\alpha')+1$. Whereas, if $I^{(q+1)}_\alpha$ is an overlapping region, then set $j_{q+1} (\alpha) =0$. 
\end{definition}

\begin{lemma}\label{l:interval_iter}
The new collection of intervals $\{I^{(q+1)}_{\alpha}\}$ satisfies the constraints of Section \ref{ss:time_intervals} and the left inclusion of \eqref{e:spt_triple}, namely
\begin{equation}\label{e:spt_triple2}
\bigcup_{1\leq \alpha \leq N(q+1)} I_\alpha^{(q+1)} \subset [2^{-q-3}, I -2^{-q-3}]\, .
\end{equation}
\end{lemma}

\begin{proof} Observe that $I^{(q+1)}_0 = H_0 = [0, t_0] \subset J_0 = I_0^{(q)} = [0, t_1]$ and
\[
t_0 \stackrel{\eqref{e:overlapping-geometry2}}{\geq} \frac{t_1}{2} \stackrel{\eqref{e:spt_triple}}{\geq} 2^{-q-3}\, .
\]
An entirely analogous argument gives $I^{(q+1)}_{N (q+1)+1} = [t_3, 1]$ with $t_3 \leq 1 - 2^{-q-2}$. On the other hand by the very definition of our intervals, we have
\[
\bigcup_{1\leq \alpha \leq N(q + 1)} I_\alpha^{(q+1)} \subset [t_0, t_3]
\]
and thus \eqref{e:spt_triple2} follows at once.

\smallskip

Next we turn to the constraints of Section \ref{ss:time_intervals}, more precisely \eqref{e:interval_constraint}, \eqref{e:regions_increase} and \eqref{e:regions_est} for step $q+1$. The property \eqref{e:regions_increase} follows obviously from the definition of $j_{q+1}(\alpha)$. 

The constraint \eqref{e:interval_constraint} requires for all $\varsigma$ the inequalitites
\begin{equation*}
|K_{\varsigma}|\geq \frac{4}{\mu_{q+2,0}},\qquad |H_\varsigma|\geq \frac{4}{\mu_{q+2,j+1}}\,.
\end{equation*} 
where $j=j_q(\alpha_q(\varsigma))$.
In light of \eqref{e:overlapping-geometry}-\eqref{e:size-of-K} this will follow from the inequalities
\begin{align}
&\frac{\eta_{q+1,j}\mu_{q+2,0}}{\mu_{q+1,j}}\geq 4 \qquad &\textrm{for all $j\in \mathbb N$}\label{e:(a)}\\
&\frac{\mu_{q+2,j+1}}{\mu_{q+1,j}}\geq 4\qquad &\textrm{for all $j\in \mathbb N$.}\label{e:(b)}
\end{align}

Consider condition \eqref{e:(a)} and observe that for $j\geq 2$
\begin{equation}\label{e:eta/mu-ordering_1}
\frac{\eta_{q+1, j}}{\mu_{q+1,j}}=\lambda_{q+1}^{\beta_{-1}-1}\stackrel{\eqref{e:beta-1}}{=}\lambda_{q+1}^{b\beta_0-1-(b-1)\beta_{\infty}}\stackrel{\eqref{e:condition}}{\geq} \lambda_{q+2}^{-(1-\beta_0)\frac{(b+1)}{2b}}= \mu_{q+2,0}^{-1}\lambda_{q+2}^{\frac{(b+1)}{2}\beta_{\infty}},
\end{equation}
which implies \eqref{e:(a)}. Since $\eta_{q+1,1}<\eta_{q+1,0}$ we obviously obtain
\begin{equation}\label{e:eta/mu-ordering_2}
\frac{\eta_{q+1,1}}{\mu_{q+1,1}}\leq \frac{\eta_{q+1, 0}}{\mu_{q+1, 0}}.
\end{equation}
Then keeping \eqref{e:eta/mu-ordering_1} in mind, in order to conclude \eqref{e:(a)} it suffices to prove
\begin{equation}\label{e:eta/mu-ordering_3}\frac{\eta_{q+1, 2}}{\mu_{q+1,2}}\leq \frac{\eta_{q+1, 1}}{\mu_{q+1, 1}}\, .
\end{equation}
Using \eqref{e:beta-1} and \eqref{e:betas_again}, we have $b\beta_0=\beta_{-1}+(b-1)\beta_{\infty}$ and $\beta_1=\frac{\beta_0+(b-1)\beta_{\infty}}{b}$. Thus it follows that
\[\eta_{q+1,1}=\lambda_{q+1}^{b_{-1}-\frac{\beta_0}{b}-\frac{b-1}{b}\beta_{\infty}}~.\]
Hence taking logarithms, the  inequality \eqref{e:eta/mu-ordering_2} corresponds to
\[
\beta_{-1}-1\leq\beta_{-1}+\frac{(1+b)+(1-b)(1+\beta_0)+(-b^2+b+2)\beta_{\infty}}{2b} \, .
\]
Multiplying by $2b$ and factorizing by $b-1$, the inequality becomes
\[
(b-1)(1+\beta_0-(b+2)\beta_{\infty})\geq 0,
\]
which is implied by \eqref{e:condition}.

Now consider \eqref{e:(b)}. From \eqref{e:eta/mu-ordering_1}-\eqref{e:eta/mu-ordering_3} we deduce
\[\frac{\mu_{q+1,j+1}}{\mu_{q+1,j}}\geq \frac{\eta_{q+1,j+1}}{\eta_{q+1,j}}=\lambda_{q+1}^{\beta_{j}-\beta_{j+1}}\stackrel{\eqref{e:betas}}{=}\lambda_{q+1}^{(b-1)(\beta_{j+1}-\beta_\infty)}\geq \lambda_{q+1}^{-(b-1)\beta_\infty}
\]
and from \eqref{e:mu-ineq}
\[\frac{\mu_{q+2,j+1}}{\mu_{q+1,j+1}}\geq \lambda_{q+2}^{(b-1)(1-\beta_{\infty})}\geq \lambda_{q+1}^{(b-1)(1-\beta_{\infty})}~.\]
Combining the two inequalites yields
\[\frac{\mu_{q+2,j+1}}{\mu_{q+1,j}}\geq \lambda_{q+1}^{(b-1)(1-2\beta_{\infty})}~.\]
Therefore \eqref{e:(b)} follows as a consequence of \eqref{e:condition}. This concludes the verification of \eqref{e:interval_constraint}.

\bigskip

Concerning \eqref{e:regions_est}, we first prove the case $j=0$, which amounts to the estimate
\begin{equation}\label{e:V_q+2^0}
|V_0^{(q+1)}| \leq \lambda_0 \lambda_{q+2}^{\beta_0 - \beta_\infty + \sfrac{\eps}{4}}\, .
\end{equation}
we have
\begin{equation*}
\begin{split}
\left|V^{(q+1)}_0\right| &= \sum_{\varsigma}|K_{\varsigma}| \leq  \sum_{\varsigma}2\eta_{q+1, j_{q}(\alpha_q(\varsigma))}|J_{\varsigma}|\\
&\leq \sum_{j=0}^q2\eta_{q+1, j} |V^{(q)}_j| \\
& \leq 2\lambda_0\sum_{j=0}^q \lambda_{q+1}^{b\beta_0-(b-1)\beta_\infty- \beta_j} \lambda_{q+1}^{\beta_j - \beta_\infty + \eps/4}\\ 
&= 2(q+1)\lambda_0 \lambda_{q+1}^{-b(\beta_\infty-\beta_0)+\sfrac{\eps}{4}}\, .
\end{split}
\end{equation*}
So the inequality \eqref{e:V_q+2^0} follows if 
\[
2 (q+1) \lambda_{q+1}^{-b(\beta_\infty-\beta_0)+\sfrac{\eps}{4}} \leq
\lambda_{q+1}^{- b(\beta_\infty - \beta_0) + b \sfrac{\eps}{4}} = 
\lambda_{q+2}^{\beta_0 - \beta_\infty + \sfrac{\eps}{4}}\, .
\]
The latter inequality is however a consequence of 
\begin{equation}\label{e:logarithmic_gain}
2(q+1)\lambda_{q+1}^{-\eps(b-1)/4}\leq 1.
\end{equation}
Thus it suffices to choose $\lambda_0$ sufficiently large (depending on $b-1>0$ and $\eps>0$, but not on $q$).

Observe next that $V^{(q+1)}_{j+1}\subset V^{(q)}_j$ for all $j\geq 0$. Thus we have
\begin{align}
\left| \bigcup_{i=0}^{j+1} V_i^{(q+1)}\right| \leq & 
\left|V_0^{(q+1)}\right|+ \left|\bigcup_{i=0}^j V_i^{(q)}\right|\nonumber\\
\leq & \lambda_0 \left( 2(q+1)  \lambda_{q+1}^{-b(\beta_\infty-\beta_0)+\sfrac{\eps}{4}} + 
\lambda_{q+1}^{\beta_j - \beta_\infty + \sfrac{\eps}{4}}\right)\, .\label{e:partial_sums}
\end{align}
We first observe that we can impose $2(q+1)\lambda_{q+1}^{-\eps(b-1)/4}\leq \frac{1}{2}$ by choosing $\lambda_0$ yet larger. On the other hand we can also impose  
\[
\lambda_{q+2}^{\beta_{j+1}-\beta_\infty + \eps/4} \geq 2 \lambda_{q+1}^{\beta_j - \beta_\infty + \eps/4}\, ,
\]
which (again by choosing $\lambda_0$ sufficiently larger) is implied by
\[
\left(b (\beta_{j+1} - \beta_\infty) + b \frac{\eps}{4}\right)  - \left(\beta_j - \beta_\infty +\frac{\eps}{4}\right)
= \frac{(b-1) \eps}{4} > 0\, .
\]
Combining these inequalities with \eqref{e:partial_sums} we then achieve
\[
\left| \bigcup_{i=0}^{j+1} V_i^{(q+1)}\right| \leq \frac{1}{2} \lambda_{q+2}^{-(\beta_\infty-\beta_0) + \sfrac{\eps}{4}}
+ \frac{1}{2} \lambda_{q+2}^{- (\beta_\infty - \beta_j) + \sfrac{\eps}{4}}
\leq  \lambda_{q+2}^{- (\beta_\infty - \beta_j) + \sfrac{\eps}{4}}\, .
\]
\end{proof}

\subsection{Partition of unity in time and bounds on flows}

We subsequently define the cut-off functions $\chi_\varsigma$ on $[0,1]$ with the properties that
\begin{itemize}
\item $\sum \chi_\varsigma^2 = 1$;
\item $\chi_\varsigma$ is identically $1$ on $H_\varsigma$, its support is an interval and it is contained, for $\varsigma \in \{1, \ldots , N'\}$, in the interior of $K_{\varsigma-1} \cup H_\varsigma \cup K_\varsigma $; for $\varsigma =0$ $\chi_0$ is defined on $H_0 \cup K_0 = [0, t_0]\cup K_0$, identically $1$ on $[0, t_0]$ and $0$ in a neighborhood of the right endpoint of $K_0$; $\chi_{N'+1}$ is defined in a similar way;
\item For $\varsigma \in \{1, \ldots , N'\}$ on $K = K_\varsigma$ and $K = K_{\varsigma-1}$ we have the estimate 
\[
\|\partial_t^N \chi_\varsigma\|_{C^0(K)} \leq C(N) |K|^{-N}
\]
where $C(N)$ is a geometric constant. Similar estimates are valid, respectively, on $K_0$ and $K_{N'}$ for the
functions $\chi_0$ and $\chi_{N'+1}$.
\end{itemize}

\bigskip

We conclude this section with a proposition, which will be used later extensively to obtain the inductive estimates in Section \ref{ss:estimates} for the new triple $(v_{q+1},p_{q+1},\mathring{R}_{q+1})$, and in a sense serves to justify our choices of the parameters $\mu$ and $\eta$ above and the alternatives (A1) and (A2) in Definition \ref{d:ov_regions}.

\begin{proposition}\label{p:CFL}
Let $\chi_{\varsigma}$ be a cut-off function in our partition of unity on $[0,1]$ and let $j=j_q(\alpha_q(\varsigma))$. Then
the local estimates \eqref{e:velocity_est}-\eqref{e:D_tR_est} hold for all $t\in \supp\chi_\varsigma$. More precisely, with
$$
\delta_{q,(j-1)_+}=\lambda_q^{-2\beta_{(j-1)_+}},\qquad \delta_{q+1,j}=\lambda_{q+1}^{-2\beta_{j}}
$$
we have for all $t\in \supp\chi_\varsigma$
\begin{equation}\label{e:local-estimates}
\begin{split}
\lambda_q^{-2} \|v_q (t)\|_2 + \lambda_q^{-1} \|v_q (t)\|_1  &\leq M \delta_{q,(j-1)_+}^{\sfrac12}\,, \\
\lambda_q^{-2}\|p_q (t)\|_2 + \lambda_q^{-1}\|p_q (t) \|_1  &\leq M^2 \delta_{q,(j-1)_+}\,,\\
\lambda_q^{-2} \|\mathring{R}_q (t)\|_2 + \lambda_q^{-1} \|\mathring{R}_q (t)\|_1 + \|\mathring{R}_q (t)\|_0 &\leq \delta_{q+1,j}\,,\\
\|(\partial_t +v_q\cdot \nabla) \mathring{R}_q (t)\|_0 &\leq \delta_{q+1,j}\delta_{q,j-1}^{\sfrac12}\lambda_q\, .
\end{split}
\end{equation}
Furthermore, if $\varsigma \in \{1, \ldots , N'\}$ then for all $t\in \supp (\chi_\varsigma)$ we have
\begin{equation}\label{e:CFL}
\left(\sup_{t\in\supp\chi_\varsigma}\|v_{q}(t)\|_1\right)|\supp\chi_\varsigma|\leq 10M\lambda_{q+1}^{-\omega}\leq 1,
\end{equation}
and
$$
\omega=\frac{(b-1)(1-\beta_0+b\beta_{\infty})}{2b} 
$$
When $\zeta\in \{0, N'+1\}$ the estimate \eqref{e:CFL} holds with $|K_0|$, respectively $|K_{N'+1}|$ in place of $|\supp (\chi_\varsigma)|$.
\end{proposition}

\begin{remark}
In the statement above we have introduced the parameters $\delta_{r,l}$ in order to aid better the comparison with \cite{BDS}. These parameters will be used to estimate the new triple $(v_{q+1},p_{q+1},\mathring{R}_{q+1})$, in an analogous way to what has been done in \cite{BDS}. However, it is important to note that here $\delta_{q,j-1}$ and $\delta_{q+1,j}$ are not uniform in time, but depend on the particular time interval $\supp\chi_\varsigma$. 
\end{remark}

\begin{proof}
Let $j=j_{q}(\alpha_q(\varsigma))$. Recall that, if $\varsigma \in \{1, \ldots N'\}$, then
\[
\supp\chi_\varsigma\subset K_{\varsigma-1} \cup H_\varsigma\cup K_{\varsigma}\quad\textrm{ and }\quad \supp\chi_\varsigma'\subset K_{\varsigma-1}\cup K_{\varsigma},
\]
where $H_\varsigma \subset J_\varsigma\subset V_j^{(q)}$. Set $j'=j_{q}(\alpha_q(\varsigma+1))$ and $j''=j_{q}(\alpha_q(\varsigma-1))$. By Definition \ref{d:ov_regions}, 
$$
K_\varsigma\subset \begin{cases} J_{\varsigma+1}\subset V_{j'}^{(q)}& \textrm{ if }j'\geq j\\ 
J_{\varsigma}\subset V_{j}^{(q)}& \textrm{ if }j'<j,\end{cases}
\quad\textrm{ and }\quad
K_{\varsigma-1}\subset \begin{cases} J_{\varsigma-1}\subset V_{j''}^{(q)}& \textrm{ if }j< j''\\ 
J_{\varsigma}\subset V_{j}^{(q)}& \textrm{ if }j\geq j'',\end{cases}
$$
Consequently, 
by the same remark as in Section \ref{sss:global} regarding the monotonicity of $j\mapsto \beta_j$, the local estimates
\eqref{e:velocity_est}-\eqref{e:D_tR_est} hold for all $t\in \supp\chi_\varsigma$. 

The cases $\varsigma = 0$ and $\varsigma = N'+1$ follow from the obvious adjustments.

\smallskip

Next we turn to the proof of \eqref{e:CFL}. Again we assume $\varsigma \in \{1, \ldots , N'\}$ and leave the obvious adjustments to the reader in the case of the ``endpoints'' $\varsigma \in \{0, N'+1\}$.

Recall that $|J_{\varsigma}|\leq 4\mu_{q+1,j}^{-1}$, and for $K\in \{K_{\varsigma}, K_{\varsigma-1}\}$ we have that 
either $K\subset J_{\varsigma}$ or $|K|=\eta_{q+1,j'}\mu_{q+1,j'}^{-1}$ for some $j'\geq j$. Consequently
$$
|\supp\chi_\varsigma|\leq \frac{4}{\mu_{q+1,j}}+2\max_{j'\geq j}\frac{\eta_{q+1,j'}}{\mu_{q+1,j'}}.
$$
Next recall \eqref{e:eta/mu-ordering_1}-\eqref{e:eta/mu-ordering_3}, which imply
\begin{equation}\label{e:etamu-monotone}
\frac{\eta_{q+1,j}}{\mu_{q+1,j}}\geq \frac{\eta_{q+1,j'}}{\mu_{q+1,j'}}\quad\textrm{ if }j'\geq j.
\end{equation}
We deduce then that
\begin{equation}\label{e:size_of_suppchi}
|\supp\chi_\varsigma|\leq \frac{4}{\mu_{q+1,j}}+2\frac{\eta_{q+1,j}}{\mu_{q+1,j}}\leq \frac{6}{\mu_{q+1,j}}.
\end{equation}
Therefore, using \eqref{e:velocity_est} we obtain
$$
\left(\sup_{t\in\supp\chi_\varsigma}\|v_{q}(t)\|_1\right)|\supp\chi_\varsigma|\leq \begin{cases}\ 10 M \lambda_{q}^{-(b-1)(1-\beta_{\infty})}& j\geq 2\\
10 M \lambda_{q}^{-(b-1)(1+b\beta_{\infty}-\beta_0)/2}& j=0,1.\end{cases}
$$
We observe that the second quantity on the right hand side bounds the first:
\begin{equation}\label{e:nash_osc}
\lambda_{q}^{-(b-1)(1-\beta_{\infty})}=\lambda_{q}^{-(b-1)(1+b\beta_{\infty}-\beta_0)/2}
\underbrace{\lambda_q^{-(b-1)(1-(2+b)\beta_{\infty}+\beta_0)/2}}_I.
\end{equation}
where $I\ll 1$ as a result of \eqref{e:condition}.
\end{proof}

\begin{remark}
The inequality \eqref{e:nash_osc}, implies the following inequality between the parameters, which will be used often
in the rest of the paper:
\begin{equation}\label{e:CFL2}
\frac{\delta_{q, i-1}^{\sfrac12} \lambda_q}{\delta_{q+1, i}^{\sfrac12}\lambda_{q+1}}\leq 2\frac{\delta_{q, (i-1)_+} \lambda_q}{\mu_{q+1,i}}\leq 4\lambda_{q+1}^{-\omega}\leq \lambda_{q+1}^{-(b-1)\beta_{\infty}-2\eps_0}\, .
\end{equation}
for any fixed $\eps_0>0$ satisfying
\begin{equation}\label{e:ell_constr}
\eps_0\leq\frac{(b-1)(1-3b(\beta_{\infty}+\beta_0))}{8b},
\end{equation}
where we assume  $\lambda_0$ to be sufficiently large depending on the choice of $\eps_0$. Observe that the right hand side of \eqref{e:ell_constr} is positive due to \eqref{e:condition}.
As a consequence of \eqref{e:CFL2} and \eqref{e:betas} we have the useful identity
\begin{equation}\label{e:index_inc}
\delta_{q+1,j}\lambda_{q+1}^{-\omega}\leq \delta_{q+2,j+1}\lambda_{q+1}^{-2\eps_0}.
\end{equation}
\end{remark}

\subsection{Smoothing the velocity and estimates on the regularized flow}

We fix a symmetric non-negative convolution kernel $\psi\in C^\infty_c (\R^3)$ and a small parameter $\ell$ given by
\begin{align}
 \ell &= \lambda_{q+1}^{\eps_0-1}\stackrel{\eqref{e:CFL2}}{\leq}(1+M)^{-1}\lambda_{q+1}^{\omega-1}\, . \label{e:ell_choice}
\end{align}
 for fixed $\eps_0>0$ satisfying \eqref{e:ell_constr} -- assuming $\lambda_0$ larger if need be. Then define, considering $v_q(\cdot,t)$ as a $2\pi$-periodic function on $\R^3$, $v_\ell:=v*\psi_\ell$ i.e.
$$
v_\ell(x,t):=\int_{\R^3}v_q(x-y)\psi_\ell(y)\,dy.
$$
We have the standard mollification estimates for all $t\in [0,1]$:
\begin{align}
\|v_\ell(t)-v_q(t)\|_0&\leq C\ell\|v_q(t)\|_1\label{e:v_ell_est_0}\\
\|v_\ell(t)\|_N&\leq C_N\ell^{1-N}\|v_q(t)\|_1\qquad N\geq 1,\label{e:v_ell_est_N}
\end{align}
where $C,C_N$ are universal constants.

\smallskip

Given $s\in [0,1]$ we define the flow $X_s(x,t)$ and inverse flow $\Phi_s(x,t)$ of the vectorfield $v_\ell$ starting at time $t=s$ in the usual way, so that
\begin{equation}\label{e:X-transport}
\left\{\begin{array}{l}
\partial_t X_s=v_\ell\left(X_s,t\right),\\ \\
X_s(x,s)=x,
\end{array}\right.
\end{equation}
and 
\begin{equation}\label{e:Phi-transport}
\left\{\begin{array}{l}
\partial_t \Phi_s + v_{\ell}\cdot  \nabla \Phi_s =0\\ \\
\Phi_s (x, s)=x\, .
\end{array}\right.
\end{equation}
Observe that, if $y\in (2\pi \Z)^3$, then $X_s(x, t) - X_s(x+y, t) \in (2\pi \Z)^3$, hence
$X_s(\cdot, t)$, and similarly $\Phi_s(\cdot,t)$, can be thought of as 
volume-preserving diffeomorphisms of $\T^3$ onto itself. 

We have the following standard lemma:
\begin{lemma}\label{l:X_est}
Let $s\in[0,1]$. For any $t\in [0,1]$ with 
$|t-s|\leq M^{-1} \lambda_q^{-1+\beta_0}$ we have 
\begin{equation}\label{e:X_est}
\|D^NX_s(t)\|_{0}\leq C_N  \ell^{1-N} \quad N\geq 1,
\end{equation}
where the constants $C_N$ depends only on $N$.
\end{lemma}
\begin{proof}
Recall from \eqref{e:g-velocity_est} the uniform bound
$$
\sup_{t\in [0,1]}\|v_q(t)\|_1\leq M\lambda_q^{1-\beta_0}.
$$
Hence the restriction on $t$ in the statement of the lemma corresponds to the standard (CFL-)condition for the flow $X_s$:
$$
|t-s|\sup_t\|Dv_q(t)\|_0\leq  1.
$$
The estimate \eqref{e:X_est} for $N=1$ then follows from a standard application of Gronwall's inequality.
For the case $N\geq 2$ observe that, by the chain rule \eqref{e:chain1} in Proposition \ref{p:chain},
$$
\frac{\partial}{\partial t}\|D^NX_s(t)\|_0\leq C\Bigl[\|Dv_{\ell}\|_0\|D^NX_s(t)\|_0+\|D^Nv_{\ell}\|_0\|DX_s(t)\|_0^N\Bigr]
$$
Hence \eqref{e:X_est} follows from the case $N=1$ and from \eqref{e:v_ell_est_N}.
\end{proof}

Obviously the analogous estimates for the inverse flow $\Phi_s$ also hold. However, for the inverse flow we need 
more precise (local) estimates restricted to times in the support of each cutoff function $\chi_{\varsigma}$. More precisely, let 
$\varsigma \in \{0, N'+1\}$ and let $t_\varsigma$ be the center of the interval $\supp \chi_\varsigma$. We will consider the 
inverse flow
$$
\Phi_{\varsigma}(x,t):=\Phi_{t_{\varsigma}}(x,t)
$$
for times $t\in \supp\chi_\varsigma$.

We will frequently deal with the transport derivative with respect to the regularized flow $v_\ell$ of various expressions, and will henceforth use the notation
\[
D_t:=\partial_t+v_{\ell}\cdot\nabla\, .
\]

\begin{lemma}\label{l:Phi_est}
For every $t\in \supp\chi_{\varsigma}$ we have
\begin{align}
\|D\Phi_{\varsigma}(t)\|_0&\leq C\label{e:Phi1}\\
\|D\Phi_{\varsigma}(t)-\Id\|_0&
\leq C  M\lambda_{q+1}^{-\omega}\label{e:Phi2}\\
\|D^N\Phi_{\varsigma}(t)\|_0&
\leq C_N  M  \ell^{1-N}\lambda_{q+1}^{-\omega}\quad \forall N\geq 2\label{e:Phi3}\\
\|D_t D\Phi_\varsigma (t)\|_N &\leq C_N  M^2 \delta_{q, (j-1)_+}^{\sfrac{1}{2}} \lambda_q \lambda_{q+1}^{-\omega}\ell^{-N} \, .\label{e:D_tDPhi_est}
\end{align}
where the constant $C$ is universal and the constant $C_N$ depends only on $N$.
\end{lemma}

\begin{proof} We first treat the main case $\varsigma \in \{1, \ldots , N'\}$.
The estimates \eqref{e:Phi1}, \eqref{e:Phi2} and \eqref{e:Phi3} follow analogously to those in Lemma \ref{l:X_est} using Proposition \ref{p:transport_derivatives} and 
the local (CFL-)condition \eqref{e:CFL} in Proposition \ref{p:CFL}. 

Next we observe that
\begin{equation}\label{e:D_tDPhi}
D_t D\Phi_\varsigma (x, t) = D_t (D\Phi_\varsigma (x,t) - \Id) = Dv_\ell (x, t) (D\Phi_\varsigma (x,t)- \Id)
\end{equation}
and thus, using \eqref{e:Phi2} we obtain
\[
\|D_t D\Phi_\varsigma (t)\|_N \leq C M \|v_q(t)\|_1\lambda_{q+1}^{-\omega}\ell^{-N} \, ,
\]
from which \eqref{e:D_tDPhi_est} easily follows (using Proposition \ref{p:CFL}).

We now come to $\varsigma \in \{0, N'+1\}$. Fix for instance $\varsigma =0$. In this case the vector field $v_\ell$ vanish identically on $H_0$. It thus suffices to apply the estimate \eqref{e:CFL} with $|K_0|$ replacing $|\supp (\chi_\varsigma)|$. 
\end{proof}

\subsection{Smoothing of the Reynolds stress} 

We now wish to define a tensor $R_{\varsigma}$ which will be obtained from a suitable approximation of $\mathring{R}_q$ and the addition of a tensor $\rho_{\varsigma}\Id$ for some carefully chosen $\rho_\varsigma$. As before, let
$j=j_q(\alpha_q(\varsigma))$. 

\begin{definition}\label{d:regularized_Reynolds}
First of all, we define the mollification of $\mathring{R}_q(\cdot,t)$ as we did with $v_q$, i.e.
$$
\mathring{R}_\ell(x,t):=\int_{\R^3}\mathring{R}_q(x-y,t)\psi_\ell(y)\,dy,
$$
where we treat $\mathring{R}_q(\cdot,t)$ as a $2\pi$-periodic tensor on $\R^3$. Then $\mathring{R}_\ell(\cdot,t)$ is also $2\pi$-periodic, so that we can think of it as a tensor on $\T^3$.  

Next we distinguish three cases:
\begin{itemize}
\item[(a)] $\varsigma \in \{1, \ldots , N'\}$ and $j=0$;
\item[(b)] $\varsigma \in \{1, \ldots , N'\}$ and $j\geq 1$;
\item[(c)] $\varsigma \in \{0, N'+1\}$.
\end{itemize}
In the first case we will use an approximation procedure borrowed from the paper \cite{Isett} to define $R_{\varsigma}$, whereas for the second case we will employ the approximation procedure from \cite{BDS}.

\begin{itemize}
\item[(a)] We extend first of all $\mathring{R}_{q}(\cdot,t)$ (and $\mathring{R}_{\ell}(\cdot,t)$) by zero to all $t\in \R$ (recalling \eqref{e:spt_triple}) and then define
\begin{equation}\label{e:mollify_along_the_flow}
\mathring{R}_\varsigma (x,t) = \int_{-\tau}^\tau \mathring{R}_\ell (X_t (x, t+s), t+s)\, \varrho_\tau (s)\, ds\, ,
\end{equation}
where $\varrho\in C^{\infty}_c(\R)$ is a symmetric nonnegative convolution kernel supported in the interval $]-1,1[$ and 
\begin{equation}\label{e:tau_choice}
\tau=\lambda_{q+1}^{-1+\beta_0}.
\end{equation}
\item[(b)] We set 
\begin{equation}\label{e:R-transport}
\mathring{R}_\varsigma (x, t)= \mathring{R}_\ell (\Phi_{\varsigma}(x,t), t);
\end{equation}
\item[(c)] We simply define $\mathring{R}_\varsigma \equiv 0$.
\end{itemize}
\end{definition}

We conclude this section by listing a number of estimates related to the approximation of $\mathring R_{q}$ in \eqref{e:R-transport} and \eqref{e:mollify_along_the_flow}. Similar estimates can be found in \cite{BDS,Isett}.

\begin{lemma}\label{l:regularized_R_q}
Assume $t\in \supp \chi_\varsigma$ and set $j:= j_q (\alpha_q (\varsigma))$.
Then the following estimates are satisfied, provided $\lambda_0$ is sufficiently large (depending only on $M$):
\begin{align}
\|\mathring{R}_\varsigma (t)\|_0&\leq \delta_{q+1,j}\label{e:R_mol_0}\\
\|\mathring R_\varsigma (t)\|_N&\leq  \bar C \delta_{q+1,j} \lambda_q \ell^{1-N} \qquad\qquad \mbox{for $N\geq 1$} \label{e:R_mol}\\
\|D_t\mathring R_\varsigma (t)\|_N&\leq  C \delta_{q+1,j}\delta_{q,j-1}^{\sfrac12}\lambda_q\ell^{-N}\qquad\qquad \mbox{for $N\geq 0$} \label{e:R_Dt}\\
\|D_t^2\mathring R_\varsigma (t)\|_N&\leq    C \delta_{q+1,j}^{\sfrac32}\delta_{q,j-1}^{\sfrac12}\lambda_q\lambda_{q+1}\ell^{-N}
\qquad\qquad \mbox{for $N\geq 0$} \label{e:R_Dt2}\\ 
\|(\mathring R_{q}-\mathring R_\varsigma) (t)\|_0& \leq  C \delta_{q+2,j} \lambda_{q+1}^{-\eps_0}\, , \label{e:R-R_moll}
\end{align}
where the constant $\bar C$ depends only on $N$ and $C$ depends only on $N$ and $M$.
\end{lemma}
\begin{proof} 
We have three cases to examine:
\begin{itemize}
\item[(a)] $\varsigma \in \{0 , N'+1\}$;
\item[(b)] $j>0$ and $\varsigma \in \{1, \ldots , N'\}$;
\item[(c)] $j=0$ and $\varsigma \in \{1, \ldots , N'\}$.
\end{itemize}

\medskip
 {\bf Case (a)} Obviously all estimates are trivial except \eqref{e:R-R_moll}. To fix ideas let us now consider $\varsigma = 0$.
It follows that $\mathring{R}_\varsigma (t) = \mathring{R}_q (t) =0$ as long as $t\in I^{(q)}_0$. Now observe that, by construction, $\supp (\chi_0) \subset  I^{(q)}_0 \cup K_0$, that $|K_0| \leq \eta_{q+1,0} \mu_{q+1, 0}^{-1}$ and that the left endpoint of $K_0$ lies necessarily in $I^{(q)}_0$. Thus, for $t\in K_\varsigma$ we can use the global estimate \eqref{e:g-D_tR_est} to conclude
\begin{align}
\|\mathring{R}_\varsigma (t) - \mathring{R}_q (t)\|_0 
\leq & |K_0| \sup_t \|(\partial_t + v_q \cdot \nabla) \mathring{R}_q (t)\|_0 \leq |K_0| \delta_{q+1,0} \delta_{q,-1}^{\sfrac{1}{2}} \lambda_q\nonumber\\
\leq & \frac{\eta_{q+1,0}}{\mu_{q+1,0}} \delta_{q+1,0} \delta_{q,-1}^{\sfrac{1}{2}} \lambda_q \, .  \label{e:trivial_case}
\end{align}
Recall $b\beta_0=\beta_{-1}+(b-1)\beta_{\infty}$, then calculating we have
\[\eta_{q+1,0}\delta_{q,-1}^{\sfrac12}\leq 2\lambda_{q,0}^{-(b^2-2b+1)\beta_{\infty}+(b^2-2b)\beta_0}
=\lambda_{q,0}^{-(b-1)^2(\beta_{\infty}-\beta_0)-\beta_0}\leq\lambda_{q,0}^{-\beta_0}
\]
Hence \eqref{e:R-R_moll} follows as a consequence of \eqref{e:CFL2} and \eqref{e:index_inc}.

An analogous argument proves the same estimate for $\varsigma = N'+1$.

\medskip

{\bf Case (b)} Observe that according to our definitions $D_t \mathring{R}_\varsigma = 0$ and thus the estimates \eqref{e:R_Dt} and \eqref{e:R_Dt2} are obvious. 

Let us consider \eqref{e:R-R_moll}.  Since $D_t\mathring{R}_\varsigma=0$, we have
\begin{align*}
\|D_t(\mathring{R}_q(t)-\mathring{R}_{\varsigma})(t)\|_{0}&  = \|D_t\mathring{R}_q(t)\|_{0}\\
& \leq \|(\partial_t+v_q\cdot \nabla)\mathring{R}_q(t)\|_0+\|v_q(t)-v_{\ell}(t)\|_0\|\mathring{R}_q(t)\|_1\\
&\leq C\delta_{q+1,j}\delta_{q,j-1}^{\sfrac12}\lambda_q+C\delta_{q+1,j}\delta_{q,j-1}^{\sfrac12}\lambda_q^2\ell\\
&\leq C\delta_{q+1,j}\delta_{q,j-1}^{\sfrac12}\lambda_q\,.
\end{align*}
On the other hand, we recall that $\mathring{R}_\varsigma (x, t_\varsigma) = \mathring{R}_{\ell} (\Phi_{\varsigma}(x,t_{\varsigma}),t_{\varsigma})=\mathring{R}_{\ell} (x,t_{\varsigma})$ and thus applying Proposition \ref{p:transport_derivatives} yields\begin{equation*}
\|\mathring{R}_\varsigma (t) -  \mathring{R}_q (t)\|_0 \leq \|\mathring{R}_q (t_\varsigma)-\mathring{R}_\ell(t_\varsigma)\|_0 +C |t-t_\varsigma|\delta_{q+1,j}\delta_{q,j-1}^{\sfrac12}\lambda_q\,.
\end{equation*}
Using \eqref{e:size_of_suppchi} and \eqref{e:CFL2} (recall that $j >0$ and thus $j-1 = (j-1)_+$), we conclude
\begin{align*}
\|\mathring{R}_\varsigma (t) -  \mathring{R}_q (t)\|_0&\leq  C \delta_{q+1,j} \lambda_q \ell+C \mu_{q+1, j}^{-1} \delta_{q+1,j}\delta_{q,j-1}^{\sfrac12}\lambda_q\\
&\stackrel{\eqref{e:CFL2}}{\leq} C\delta_{q+1,j} \lambda_{q+1}^{-\omega+\eps_0} \stackrel{\eqref{e:index_inc}}{\leq}
C\delta_{q+2,j+1} \lambda_{q+1}^{-\eps_0}\, .
\end{align*}
Next recall that $\|\mathring{R}_\ell (t)\|_0 \leq \|\mathring{R}_q (t)\|_0$ and $\|\mathring{R}_\ell (t)\|_N \leq \|\psi_\ell\|_{N-1} \|\mathring{R}_q\|_1$. Obviously, since $\mathring{R}_\varsigma$ solves a transport equation, we have
$\|\mathring{R}_\varsigma (t)\|_0 \leq \|\mathring{R}_\ell (t_\varsigma)\|_0 \leq \delta_{q+1,j}$, from which \eqref{e:R_mol_0} easily follows. 

Applying Proposition \ref{p:transport_derivatives} we get
\[
\|\mathring{R}_\varsigma (t)\|_1 \leq \|\mathring{R}_\ell (t_\varsigma)\|_1 e^{S |t-t_\varsigma|}\, 
\]
where 
\[
S = \sup_{t\in \supp (\chi_\varsigma)} \|v_\ell (t)\|_1\, .
\]
Once again applying \eqref{e:CFL} we conclude $|t-t_\varsigma| S \leq C \lambda_{q+1}^{-\omega}$, where the latter constant $C$ depends only on $M$. Choosing $\lambda_0$ sufficiently large we then can assume 
$e^{|t-t_\varsigma |S}\leq 2$ and
conclude \eqref{e:R_mol} with $N=1$. For larger $N$ we apply again Proposition \ref{p:transport_derivatives} and the argument above to conclude 
\begin{align}
\|\mathring{R}_\varsigma (t)\|_N &\leq 2 \|\mathring{R}_\ell (t_\varsigma)\|_N + 2 C_N |t-t_\varsigma| \sup_{t\in \supp (\chi_\varsigma)} \|v_\ell (t)\|_N \|\mathring{R}_\ell (t_\varsigma)\|_1\nonumber\\
&\leq 2 C_N \delta_{q+1, j} \lambda_q \ell^{1-N} + 2 C_N |t-t_\varsigma| \ell^{1-N} \sup_{t\in \supp (\chi_\varsigma)} \|v_q (t)\|_1 \delta_{q+1, j} \lambda_q\, ,
\end{align}
where $C_N$ is a constant which depends only on $N$. Again applying \eqref{e:CFL}, the estimate in \eqref{e:R_mol} follows.

\medskip

 {\bf Case (c)}. Set $j=0$. In this case we will use the {\em global} estimates (cf.~Section \ref{sss:global}): for all $t\in[0,1]$ we have
\begin{align*}
&\|\mathring{R}_\ell(t)\|_0 \leq \delta_{q+1, 0}\\
& \|\mathring{R}_\ell(t)\|_N \leq C_N \delta_{q+1,0} \lambda_q \ell^{1-N} \qquad \forall N \geq 1\, ,
\end{align*}
where the constant $C_N$ depends only on $N$.

Hence \eqref{e:R_mol_0} is obvious. For $N\geq 1$ we use \eqref{e:chain1}:
\begin{align}
\|D^N ( \mathring{R}_\ell (X_t (\cdot, t+s), t+s))\|_0 \leq &C_N \|D^N X_t (\cdot, t+s)\|_0 \delta_{q+1,0} \lambda_q\nonumber\\
& + C_N \|D X_t (\cdot, t+s)\|^N_0 \delta_{q+1,0} \lambda_q \ell^{1-N}\, ,\label{e:compos_flow}
\end{align}
where again the constant $C_N$ depends only upon $N$.

Since $|s|\leq \tau=\lambda_{q+1}^{-1+\beta_0}\leq \lambda_q^{-1+\beta_0}$, from Lemma \ref{l:X_est} we deduce
\[
\|D^N X_t (\cdot, t+s)\|_0 \leq C_N\ell^{1-N}\qquad\textrm{ for all $N\geq 1$}\, .
\]
Inserting in \eqref{e:compos_flow} we conclude
\[
\|D^N ( \mathring{R}_\ell (X_t (\cdot, t+s), t+s))\|_0 \leq C \delta_{q+1,0} \lambda_q \ell^{1-N} \qquad \forall N\geq 1\, .
\]
Hence differentiating \eqref{e:mollify_along_the_flow} we achieve \eqref{e:R_mol}.

\medskip

In order to prove \eqref{e:R_Dt}-\eqref{e:R_Dt2} we use Lemma 18.2 of \cite{Isett} to deduce
\begin{align}
D_t \mathring{R}_\varsigma(x,t) & = \int (D_t \mathring{R}_\ell) (X_t (x, t+s), t+s)\, \varrho_\tau (s)\, ds \label{e:Phil_1}\\
D^2_t \mathring{R}_\varsigma(x,t) &= \int (D_t^2 \mathring{R}_\ell) (X_t (x, t+s), t+s)\, \varrho_\tau (s)\, ds\nonumber\\
&= \int \frac{d}{ds} [(D_t \mathring{R}_\ell) (X_t (x, t+s), t+s)]\, \varrho_\tau (s)\, ds\nonumber\\
&= -\tau^{-1} \int (D_t \mathring{R}_\ell) (X_t (x, t+s), t+s)\, (\varrho')_\tau (s)\, ds\, . \label{e:Phil_2}
\end{align}
We therefore conclude, arguing as in \eqref{e:compos_flow}
\begin{align}
\|D_t \mathring{R}_\varsigma(t) \|_N &\leq C \sup_s
\left(\|D_t \mathring{R}_\ell(s)\|_N + C \|D_t \mathring{R}_\ell (s)\|_0 \ell^{1-N}\right)
\label{e:Phil_3}\\
\|D^2_t \mathring{R}_\varsigma (t)\|_N &\leq C \tau^{-1} \sup_s
\left(\|D_t \mathring{R}_\ell(s)\|_N + C \|D_t \mathring{R}_\ell (s)\|_0 \ell^{1-N}\right)\, .\label{e:Phil_4}
\end{align}
Observe that, for any $s$, we have
\begin{align*}
D_t \mathring{R}_\ell (s) = &(\partial_t + v_q* \psi_\ell (s) \cdot \nabla) (\mathring{R}_q (s) * \psi_\ell)\nonumber\\
= &((\partial_t + v_q (s) \cdot\nabla )\mathring{R}_q (s)) * \psi_\ell\nonumber\\
 &+ [(v_q *\psi_\ell (s) \cdot \nabla) (\mathring{R}_q (s) *\psi_\ell)  - ((v_q (s)\cdot \nabla) \mathring{R}_q (s))*\psi_\ell.
\end{align*}
Since $v_q$ is divergence free, the components of the vector function in the last line can be written as
\[
{\rm div} \big(\underbrace{(v_q (s)*\psi_\ell) \otimes (\mathring{R}_q (s)* \psi_\ell) - (v_q\otimes \mathring{R}_q) (s)*\psi_\ell}_{T (s)}\big)
\]
Thus, using Proposition \ref{p:CET} and \eqref{e:g-D_tR_est} we reach
\begin{align}
\|D_t \mathring{R}_\ell (s)\|_N &\leq C \ell^{-N} \|(\partial_t + v_q (s)\cdot \nabla) \mathring{R}_q (s)\|_0 + \|T (s)\|_{N+1}\nonumber\\
&\leq C \ell^{-N} \|(\partial_t + v_q (s)\cdot \nabla) \mathring{R}_q (s)\|_0  + C \ell^{1-N} \|\mathring{R}_q (s)\|_1 \|v_q (s)\|_1\nonumber\\
&\leq C \delta_{q+1,0} \delta_{q, -1}^{\sfrac{1}{2}} \lambda_q \ell^{-N} + C \delta_{q+1, 0} \delta_{q, -1}^{\sfrac{1}{2}} \lambda_q^2 \ell^{1-N}\nonumber\\
&\leq C \delta_{q+1,0} \delta_{q, -1}^{\sfrac{1}{2}} \lambda_q \ell^{-N}\label{e:Phil_5}
\end{align}
Plugging \eqref{e:Phil_5} into \eqref{e:Phil_3} we immediately conclude \eqref{e:R_Dt}. Plugging \eqref{e:Phil_5} into \eqref{e:Phil_4}
we instead reach
\[
\|D^2_t \mathring{R}_\varsigma (t)\|_N \leq  C \tau^{-1} \delta_{q+1,0} \delta_{q, -1}^{\sfrac{1}{2}} \lambda_q \ell^{-N}
\]
and since $\tau^{-1} = \lambda_{q+1}^{1-\beta_0} = \lambda_{q+1} \delta_{q+1,0}^{\sfrac{1}{2}}$ we get \eqref{e:R_Dt2}.
Finally, we estimate
\[
|\mathring{R}_\varsigma (x,t) - \mathring{R}_\ell (x,t)| \leq \sup_{|s|\leq \tau} \max_x |\mathring{R}_\ell (X_t (x, t+s), t+s) - \mathring{R}_\ell (x,t)|\, .
\]
Using that $X_t (x,t)=x$ and differentiating in $s$ the map $\mathring{R}_\ell (X_t (x, t+s), t+s)$ we obtain

\begin{equation}\label{e:using_tau}
\|\mathring{R}_\varsigma (\cdot ,t) - \mathring{R}_\ell (\cdot,t)\|_0 \leq \tau \|D_t \mathring{R}_\ell\|_0
\leq C \lambda_{q+1}^{\beta_0-1} \delta_{q+1,0} \delta_{q,-1}^{\sfrac{1}{2}} \lambda_q\, .
\end{equation}
 Plugging \eqref{e:CFL2} and \eqref{e:index_inc}, we reach \eqref{e:R-R_moll}.
\end{proof}

\subsection{Definition of $v_{q+1}$} 
 
In this section we define the new velocity by prescribing the perturbation
$$
w_{q+1} =v_{q+1}-v_q.
$$
We start by defining, for each $\varsigma\in \{1, \ldots , N'\}$,
$$
R_\varsigma(x,t) = \rho_\varsigma \Id - \mathring{R}_\varsigma(x,t) ,
$$
where 
\begin{equation}\label{e:rho}
\rho_\varsigma = 4 r_0^{-1}  \lambda_{q+1}^{-2\beta_j}\, 
\end{equation}
and
\begin{itemize}
\item $r_0$ is the constant of the geometric Lemma \ref{l:split};
\item $j = j_q (\alpha_q (\varsigma))$.
\end{itemize}
Then we apply Lemma \ref{l:split} with $N=2$, denoting by $\Lambda^e$ and $\Lambda^o$ the corresponding families of frequencies in $\Z^3$, and we set $\Lambda := \Lambda^o$ + $\Lambda^e$. For each $k\in \Lambda$ and each $\varsigma$ we then set
\begin{align}
a_{k\varsigma}(x,t)&:=\sqrt{\rho_\varsigma}\gamma_k \left(\frac{R_\varsigma(x,t)}{\rho_\varsigma}\right),\label{e:a_kl}\\
w_{k\varsigma}(x,t)& := a_{k\varsigma}(x,t)\,B_ke^{i\lambda_{q+1}k\cdot \Phi_\varsigma (x,t)}.\label{e:w_kl}
\end{align}
We observe that the $a_{k\varsigma}$ are well defined. Indeed, thanks to \eqref{e:R_mol_0}
$$
\left|\frac{R_\varsigma(x,t)}{\rho_\varsigma}-\Id\right|=\frac{|\mathring{R}_\varsigma(x,t)|}{\rho_\varsigma}\leq \frac{r_0}{4}.
$$
The ``principal part'' of the perturbation $w$ consists of the map
\begin{align}\label{e:def_wo}
w_o (x,t) := \sum_{\textrm{$\varsigma$ odd}, k\in \Lambda^o} \chi_\varsigma (t)w_{k\varsigma} (x,t) +
\sum_{\textrm{$\varsigma$ even}, k\in \Lambda^e} \chi_\varsigma (t)w_{k\varsigma} (x,t)\, ,
\end{align}
where the sums are taken over the indices $\varsigma \in \{1, \ldots , N'\}$. We therefore agree upon the convention that $w_{k\varsigma} \equiv 0$, $\rho_\varsigma =0$, $R_\varsigma \equiv 0$ and $a_{k\varsigma} \equiv 0$ when $\varsigma \in \{0, N'+1\}$.

From now on, in order to make our notation simpler, we agree that the pairs of indices 
$(k,\varsigma)$ which enter in our summations satisfy always 
the following condition: $k\in \Lambda^e$ when $\varsigma$ is even and $k\in \Lambda^o$ when $\varsigma$ is odd.

It will be useful to introduce the ``phase" 
\begin{equation}\label{e:phi_kl}
\phi_{k\varsigma}(x,t)=e^{i\lambda_{q+1}k\cdot[\Phi_\varsigma (x,t)-x]},
\end{equation}
with which we obviously have
\[
\phi_{k\varsigma}\cdot e^{i\lambda_{q+1}k\cdot x}=e^{i\lambda_{q+1}k\cdot\Phi_\varsigma}.
\]
The corrector $w_c$ is then defined in such a way that $w_{q+1} := w_o+w_c$ is divergence free:
\begin{align}
w_c&:= \sum_{(k,\varsigma)} \frac{\chi_\varsigma}{\lambda_{q+1}}\curl\left(ia_{k\varsigma}\phi_{k\varsigma}\frac{k\times B_k}{|k|^2}\right) e^{i\lambda_{q+1}k \cdot x}\nonumber\\
&=\sum_{(k,\varsigma)}\chi_\varsigma\Bigl(\frac{i}{\lambda_{q+1}}\nabla a_{k\varsigma}-a_{k\varsigma}(D\Phi_{\varsigma}-\Id)k\Bigr)\times\frac{k\times B_k}{|k|^2}e^{i\lambda_{q+1}k\cdot\Phi_\varsigma}\label{e:corrector}
\end{align}

\begin{remark} To see that $w_{q+1} =w_o+w_c$ is divergence-free, just note that, since $k\cdot B_k=0$, we have
$k\times (k\times B_k)=-|k|^2B_k$ and hence $w_{q+1}$ can be written as
\begin{equation}\label{e:w_compact_form}
w = \frac{1}{\lambda_{q+1}} \sum_{(k,\varsigma)} \chi_\varsigma \,\curl \left( i a_{k\varsigma}\,\phi_{k\varsigma}\,\frac{k\times B_k}{|k|^2} e^{i\lambda_{q+1}k \cdot x}\right)\, .
\end{equation}
\end{remark}

For future reference it is useful to introduce the notation
\begin{equation}\label{e:L_kl}
L_{k\varsigma}:=a_{k\varsigma}B_k+\Bigl(\frac{i}{\lambda_{q+1}}\nabla a_{k\varsigma}-a_{k\varsigma}(D\Phi_{\varsigma}-\Id)k\Bigr)\times\frac{k\times B_k}{|k|^2},
\end{equation}
so that the perturbation $w_{q+1}$ can be written as
\begin{equation}\label{e:w_Lform}
w_{q+1}=\sum_{(k,\varsigma)}\chi_\varsigma\,L_{k\varsigma}\,e^{i\lambda_{q+1}k\cdot\Phi_\varsigma}\,.
\end{equation}

\bigskip

\subsection{The pressure $p_{q+1}$ and the Reynolds stress $\mathring{R}_{q+1}$}\label{ss:R_decomposed}
We set 
\begin{equation}\label{e:R_decomposition}
\mathring{R}_{q+1} = R^0+R^1+R^2+R^3+R^4\, ,
\end{equation}
where
\begin{align}
R^0 &= \mathcal R \left(\partial_t w_{q+1}v_\ell\cdot \nabla  w_{q+1}+w_{q+1}\cdot\nabla v_\ell\right)\label{e:R^0_def}\\
R^1 &=\mathcal R \div \Big(w_o \otimes w_o- \sum_\varsigma \chi_\varsigma^2 R_\varsigma
-\textstyle{\frac{|w_o|^2}{2}}\Id\Big)\label{e:R^1_def}\\
R^2 &=w_o\otimes w_c+w_c\otimes w_o+w_c\otimes w_c - \textstyle{\frac{|w_c|^2 + 2\langle w_o, w_c\rangle}{3}}  \Id\label{e:R^2_def}\\
R^3 &=  w_{q+1}\otimes ( v_q - v_\ell) + ( v_q-v_\ell)\otimes  w_{q+1}
 - \textstyle{\frac{2 \langle ( v_q-v_{\ell}),  w_{q+1}\rangle}{3}} \Id\label{e:R^3_def}\\
R^4&=\mathring{R}_q- \sum_\varsigma \chi_\varsigma^2 \mathring{R}_\varsigma \label{e:R^4_def}
\end{align}
Observe that $\mathring{R}_{q+1}$ is indeed a traceless symmetric tensor. The corresponding form of the new pressure will then be
\begin{equation}\label{e:def_p_1}
p_{q+1} =p_q-\frac{|w_o|^2}{2} - \frac{1}{3} |w_c|^2 - \frac{2}{3} \langle w_o, w_c\rangle - \frac{2}{3} \langle v_q -v_\ell, w\rangle \,.
\end{equation}

Observe that $\sum_\varsigma \chi_\varsigma^2 \tr R_\varsigma$ is a function of time only. Since also 
$\sum_\varsigma \chi_\varsigma^2 = 1$, it is then straightforward to check that 
\begin{align*}
& \div\mathring{R}_{q+1} - \nabla p_{q+1}\\ 
=& \partial_t w + \div (v_q\otimes w + w \otimes v_q + w\otimes w) + \div \mathring{R}_q - \nabla p_q\nonumber\\
=&\partial_t w + \div (v_q\otimes w + w\otimes v_q + w \otimes w) + \partial_t v_q + \div (v_q\otimes v_q)\nonumber\\
= &\partial_t v_{q+1} + \div v_{q+1}\otimes v_{q+1}\, .
\end{align*}

The following lemma, same as in \cite{BDS}, will play a key role:
\begin{lemma}\label{l:doublesum}
The following identity holds:
\begin{equation}\label{e:doublesum}
w_o\otimes w_o = \sum_\varsigma \chi_\varsigma^2 R_\varsigma + \sum_{(k,\varsigma), (k',\varsigma'), k\neq - k'} \chi_\varsigma \chi_{\varsigma'} w_{k\varsigma} \otimes w_{k'\varsigma'}\, .
\end{equation}
\end{lemma}

\begin{remark}\label{r:triviality}
Observe also that, by our treatment of the endpoints $\varsigma \in \{0, N'+1\}$ we have that 
$\mathring{R}_\varsigma = R_\varsigma = 0$ and $w_{k\varsigma}=0$ for such values of $\varsigma$. In addition we already know that $(v_q, p_q, \mathring{R}_q) \equiv 0$ on $I^{(q)}_0 \cup I^{(q)}_{N(q)+1} \supset I^{(q+1)}_0 \cup I^{(q+1)}_{N(q+1)}$. We therefore also conclude
easily the following important lemma. 
\end{remark}

\begin{lemma}\label{l:compact_support}
The new triple $(v_{q+1}. \mathring{R}_{q+1}, p_{q+1})$ is supported in
\[
\bigcup_{\alpha =1}^{N (q+1)} I^{(q+1)}_\alpha
\]
which in turn is supported, by \eqref{e:spt_triple2} in $[2^{-q-3}, 1- 2^{-q-3}]$. Namely, \eqref{e:spt_triple}
holds with $q+1$ in place of $q$.
\end{lemma}


\section{Orderings among the parameters}\label{s:parameters}

Our choice of the parameters respect certain natural orderings which will be extremely useful in the sequel. In fact most
of them have already been proved and used in the previous sections: we collect them in the next lemma for the
reader's convenience.

\begin{lemma}\label{l:parameters}
According to our choice of the parameters $\beta_j$ and $b$ we have the following orderings:
\begin{itemize}
\item[(a)] The parameters $\delta_{q,i}$ decrease both in the index $i$ and in the index $q$. namely
\begin{equation}\label{e:asc_delta}
\delta_{q,i}\geq \delta_{q+1,i}\geq \delta_{q+1, i+1}\, ;
\end{equation}
\item[(b)] The parameters $\mu_{q+1,i}$ decrease in the index $i$, namely
\begin{equation}\label{e:asc_mu}
\mu_{q+1,i}\geq \mu_{q+1,i+1}\, ;
\end{equation}
\item[(c)] The ratios $\frac{\mu_{q+1,i}}{\eta_{q+1,i}}$ are increasing in $i$ and indeed we have
\begin{equation}\label{e:more-scaling2}
\frac{\mu_{q+1,0}}{\eta_{q+1,0}} \leq \frac{\mu_{q+1,1}}{\eta_{q+1,1}}
\leq \frac{\mu_{q+1,2}}{\eta_{q+1,2}} = \frac{\mu_{q+1,i+1}}{\eta_{q+1,i+1}}
= \delta_{q+1,-1}^{\sfrac12}\lambda_{q+1}
\end{equation}
for any $i\geq 1$; 
\item[(d)] Finally, the parameter $\delta_{q,i}^{\sfrac{1}{2}} \lambda_q$ satisfies the inequality
\begin{equation}\label{e:asc_|v|_1}
\delta_{q,i}^{\sfrac{1}{2}} \lambda_q \leq \delta_{q+1,i+1}^{\sfrac{1}{2}} \lambda_{q+1}
\end{equation}
\end{itemize}
\end{lemma}
\begin{proof}
 {\bf (a) } The inequality \eqref{e:asc_delta} is obvious because $j\mapsto \beta_j$ and $q\mapsto \lambda_q$ are both increasing. 

\medskip

{\bf (b) } The inequality \eqref{e:asc_mu} has been already proved in Lemma \ref{l:overlapping-geometry}, cf.\ \eqref{e:mu-ineq}.

\medskip

{\bf (c) } The inequality \eqref{e:more-scaling2} was proved in Lemma \eqref{l:interval_iter}, cf.\ \eqref{e:eta/mu-ordering_1}-\eqref{e:eta/mu-ordering_3}.

\medskip

{\bf (d) } The inequality \eqref{e:asc_|v|_1} is equivalent to $1-\beta_i \leq b - b \beta_{i+1}$ which, given the definition of the betas, is equivalent to $1 \geq \beta_\infty$.
\end{proof}

Next, some further inequalities will help simplifying several estimates
\begin{lemma}\label{l:parameters2}
For $\eps_0$ satisfying \eqref{e:ell_constr}, we  have 
\begin{equation}\label{e:CFL3}
\frac{\delta_{q, i-1}^{\sfrac12} \lambda_q}{\delta_{q+1, i}^{\sfrac12}\lambda_{q+1}}\leq2\frac{\delta_{q, (i-1)_+}^{\sfrac12} \lambda_q}{\mu_{q+1,i}}\leq 4\lambda_{q+1}^{-\omega}\leq \frac{\delta_{q+2,i+1}}{\delta_{q+1,i}\lambda_{q+1}^{2\eps_0}}\, .
\end{equation}
and
\begin{equation}\label{e:delta_lambda/mu}
\delta_{q,i-1}^{\sfrac{1}{2}} \lambda_{q} \leq \mu_{q+1,i}\, .
\end{equation}

\end{lemma}
\begin{proof}
Note that \eqref{e:CFL3} is just a restating of \eqref{e:CFL2} and \eqref{e:index_inc}.
The proof of \eqref{e:delta_lambda/mu} follows from \eqref{e:CFL3} and the inequality $\delta^{\sfrac12}_{q,-1}\delta^{-\sfrac12}_{q,0}\leq \delta^{\sfrac12}_{q,-1}\delta^{-\sfrac12}_{q+1,0}$ implied by \eqref{e:asc_delta}.
\end{proof}


\section{Perturbation estimates}\label{s:perturbation}

\subsection{Preliminaries}
Using the same arguments as in Lemma 3.1 of \cite{BDS}, we easily obtain the following estimates on the components of the perturbation $w$.
\begin{lemma}\label{l:ugly_lemma}
Assume $t\in \supp \chi_\varsigma$ and set $j:= j_q (\alpha_q (\varsigma))$.
Then the following estimates are satisfied:
\begin{align}
&\|a_{k\varsigma} (t)\|_0 + \|L_{k\varsigma} (t)\|_0 \leq \bar{C} \delta^{\sfrac{1}{2}}_{q+1,j}\label{e:a+L_C0}\\
&\|a_{k\varsigma} (t)\|_N\leq \bar{C} \delta_{q+1,j}^{\sfrac{1}{2}} \lambda_q \ell^{1-N} \qquad \mbox{for $N\geq 1$}\label{e:a_N}\\
&\|L_{k\varsigma} (t)\|_N \leq \bar{C} \delta_{q+1,j}^{\sfrac{1}{2}}\ell^{-N}\qquad \mbox{for $N\geq 1$}\label{e:L_N}\\
&\|\phi_{k\varsigma} (t)\|_N \leq \bar{C} M \lambda_{q+1}^{1-\omega} \ell^{1-N}\leq  \bar{C}\ell^{-N} \qquad \mbox{for $N\geq 1$},\label{e:phi_N}
\end{align}
where the constants $\bar{C}$ depend only on $N$.
\end{lemma}
It should be remarked that, compared to Lemma 3.1 in  \cite{BDS} the inequality 
\[
\ell\leq  (1+M)^{-1}\lambda_{q+1}^{\omega-1},
\]
was used in order to simplify the statement of \eqref{e:phi_N}.

\begin{proof} First observe that, since $\phi_{k\varsigma}(x,t) = e^{i \lambda_{q+1}(\Phi_{k\varsigma}(x,t) - x) \cdot k}$, for $N\geq 1$ 
we can use \eqref{e:chain1} to estimate
\[
\|\phi_{k\varsigma} (t)\|_N \leq C \left(\lambda_{q+1}\|D\Phi_{k\varsigma}(t) - \Id\|_{N-1} +\lambda_{q+1}^{N} \|D\Phi_{k\varsigma} (t) - \Id\|_0^N\right)\, .
\]
Thus the estimates \eqref{e:phi_N} are a direct consequence of Lemma \ref{l:Phi_est}. 

Next observe that from \eqref{e:chain0} and \eqref{e:R_mol} we obtain
\begin{align}
\|a_{k\varsigma} (t)\|_N &\leq C \delta_{q+1, j}^{-\sfrac{1}{2}} \|R_\varsigma (t)\|_N
\leq C \delta_{q+1, j}^{\sfrac{1}{2}} \lambda_q\ell^{1-N}\, ,
\end{align}
where the constant $C$ depends only on $N$. This proves \eqref{e:a_N}

Finally, differentiating \eqref{e:L_kl} and using \eqref{e:Holderproduct}, 
\begin{align}
\|L_{k\varsigma}\|_N \leq &\bar{C}  \|a_{k\varsigma} (t)\|_N+C\lambda_{q+1}^{-1}\|a_{k\varsigma} (t)\|_{N+1}+\nonumber\\
& + \bar{C}  \left(\|a_{k\varsigma} (t)\|_N \|D\Phi_\varsigma (t) -\Id\|_0 + \|a_{kl}\|_0 \|D\Phi_\varsigma (t) - \Id\|_N\right)\, \nonumber\\
\leq& \bar{C} \delta_{q+1, j}^{\sfrac{1}{2}} \lambda_q \ell^{1-N} + \bar{C} \lambda_{q+1}^{-1}  \delta_{q+1, j}^{\sfrac{1}{2}} \lambda_q \ell^{-N} + \bar{C} \delta_{q+1, j}^{\sfrac{1}{2}} M \lambda_{q+1}^{-\omega} \ell^{-N}\, .
\end{align}
To achieve estimate \eqref{e:L_N} is then enough to assume $\lambda_0 \geq M^{\sfrac{1}{\omega}}$ and to
apply $\lambda_{q+1} \geq \ell^{-1} \geq \lambda_q$.
\end{proof}

The next technical lemma deals with a number of helpful material derivative estimates.  The proof follows similar arguments to that of Lemma 3.1 from \cite{BDS}, taking advantage of sharper second order estimates on $(v_q,p_q,\mathring R_q)$ (cf.\ \cite{Buckmaster,Isett}).
\begin{lemma}\label{l:ugly_lemma2}
Assume $t\in \supp (\chi_\varsigma)$ and set $j:= j_q (\alpha_q (\varsigma))$.
Then the following estimates are satisfied:
\begin{align}
&\|D_t v_\ell (t)\|_0\leq C \delta_{q,(j-1)_+}\lambda_q\label{e:D_tv_ell_0}\\
& \|D_t v_\ell (t)\|_N\leq C \delta_{q,(j-1)_+}\lambda_q^2\ell^{1-N}\label{e:D_tv_ell_N}\qquad N\geq 1\\
& \|D_t a_{k\varsigma} (t)\|_N \leq C \delta_{q+1,j}^{\sfrac{1}{2}} \delta_{q,j-1}^{\sfrac{1}{2}} \lambda_q \ell^{-N}\label{e:D_ta}\\
& \|D_t^2 a_{k\varsigma} (t)\|_N \leq C \delta_{q+1,j} \delta^{\sfrac{1}{2}}_{q,j-1} \lambda_q\lambda_{q+1} \ell^{-N}\label{e:D^2_ta}\\
& \|D_t L_{k\varsigma} (t)\|_N \leq C \delta_{q+1,j}^{\sfrac{1}{2}} \delta_{q,j-1}^{\sfrac{1}{2}} \lambda_q \ell^{-N}\label{e:D_tL}\\
&\|D_t^2 L_{k\varsigma} (t)\|_N \leq C \delta_{q+1,j} \delta^{\sfrac{1}{2}}_{q,j-1} \lambda_q \lambda_{q+1} \ell^{-N} \, ,\label{D_t^2L}
\end{align}
where the constants only depend on $N$ and $M$.
\end{lemma}
\begin{proof} Recall that
$$
(\partial_t+v_q\cdot\nabla)v_q+\nabla p_q=\div \mathring{R}_q.
$$
We write
\[
D_t v_\ell = {\rm div}\, \mathring{R}_q * \psi_\ell - \nabla p_q * \psi_\ell + {\rm div} (v_q*\psi_\ell \otimes v_q * \psi_\ell - (v_q\otimes v_q) * \psi_\ell)\, 
\]
and apply Proposition \ref{p:CET} to estimate
\begin{equation*}
\|{\rm div} (v_q(t)*\psi_\ell \otimes v_q(t) * \psi_\ell - (v_q\otimes v_q)(t) * \psi_\ell)\|_N\leq \norm{v_q(t)}_1^2 \ell^{1-N}\, .
\end{equation*}
Recall also that, by standard convolution estimates (and Proposition \ref{p:CFL})
\begin{align*}
\| \nabla p_q (t) * \psi_\ell\|_N &\leq C M^2 \delta_{q, (j-1)_+} \lambda_q ^2 \ell^{1-N} \qquad \qquad &\forall N\geq 1\\
\|{\rm div} \mathring{R}_q (t) * \psi_\ell\|_N &\leq C \delta_{q+1,j} \lambda_q^2 \ell^{1 - N} &\forall N\geq 1\, .
\end{align*}
We therefore conclude
\begin{equation}
\|D_t v_\ell (t)\|_N\leq C \delta_{q,(j-1)_+}\lambda_q^2\ell^{1-N}\quad \mbox{when $N\geq 1$},
\end{equation}
whereas
\[
\|D_t v_\ell (t)\|_0 \leq C \delta_{q,(j-1)_+} \lambda_q\, .
\]
Next, note that
\begin{equation}\label{e:compute_Dt_a}
D_t a_{k\varsigma} = \rho_{\varsigma}^{-\sfrac{1}{2}} D \gamma_k \left(\frac{\mathring{R}_\varsigma}{\rho_\varsigma}\right) D_t \mathring{R}_\varsigma\, .
\end{equation}
By Lemma \ref{l:regularized_R_q}, \eqref{e:Holderproduct} and \eqref{e:chain0}, we have
\begin{align*}
\|D_t a_{k\varsigma} (t)\|_N \leq & C \delta_{q+1,j}^{-\sfrac{1}{2}} \|D_t \mathring{R}_\varsigma(t)\|_N 
+ C \delta_{q+1,j}^{-\sfrac{3}{2}} \|\mathring{R}_\varsigma\|_N \|D_t \mathring{R}_\varsigma (t)\|_0\\
\leq &\delta_{q+1,j}^{\sfrac{1}{2}} \delta_{q,j-1}^{\sfrac{1}{2}} \lambda_q \ell^{-N}\, ,
\end{align*}
from which \eqref{e:D_ta} follows.  Taking a further material derivative of both sides in \eqref{e:compute_Dt_a} we achieve
\[
D^2_t a_{k\varsigma} = \rho_{k\varsigma}^{-\sfrac{3}{2}} D^2 \gamma_k \left(\frac{\mathring{R}_\varsigma}{\rho_\varsigma}\right)
D_t \mathring{R}_\varsigma D_t \mathring{R}_\varsigma +  \rho_{\varsigma}^{-\sfrac{1}{2}} D \gamma_k \left(\frac{\mathring{R}_\varsigma}{\rho_\varsigma}\right) D^2_t \mathring{R}_\varsigma\, .
\]
Applying again Lemma \ref{l:regularized_R_q}, \eqref{e:chain0} and \eqref{e:Holderproduct} we then get
\begin{align*}
\|D^2_t a_{k\varsigma}(t)\|_N \leq & C \delta_{q+1,j}^{-\sfrac{3}{2}} \|D_t \mathring{R}_\varsigma (t)\|_0\|D_t \mathring{R}_\varsigma (t)\|_N+ C \delta_{q+1}^{-\sfrac{1}{2}} \|D^2_t \mathring{R}_\varsigma (t)\|_N\\
& + C \delta_{q+1,j}^{-\sfrac{5}{2}} \|\mathring{R}_\varsigma (t)\|_N \|D_t \mathring{R}_\varsigma (t)\|_0^2\\
& + C \delta_{q+1,i}^{-\sfrac{3}{2}}  \|\mathring{R}_\varsigma (t)\|_N \|D^2_t \mathring{R}_\varsigma (t)\|_0\\
\leq & \delta_{q+1,j} \delta_{q,j-1}^{\sfrac12} \lambda_q \lambda_{q+1} \ell^{-N}\, ,
\end{align*}
where we have used \eqref{e:asc_|v|_1}.

We now proceed to the estimates involving $L_{k\varsigma}$. First we have
\begin{equation}\label{e:D_tDa}
D_t \nabla a_{k\varsigma}  =  - D v_\ell \nabla a_{k\varsigma}+\nabla D_t a_{k\varsigma}
\end{equation}
and thus, using Lemma \ref{l:ugly_lemma}, \eqref{e:D_ta} and \eqref{e:Holderproduct}, 
\begin{equation}\label{e:D_tDa_est}
\|D_t \nabla a_{k\varsigma} (t)\|_N \leq C \delta_{q+1,j}^{\sfrac{1}{2}}\delta_{q,j-1}^{\sfrac12} \lambda_q \ell^{-N-1} \, .
\end{equation}
Differentiating the formula \eqref{e:L_kl} defining $L_{k\varsigma}$ we conclude
\begin{align*}
\|D_t L_{k\varsigma} (t)\|_N \leq & C \|D_t a_{k\varsigma} (t)\|_N + C \lambda_{q+1}^{-1} \|D_t \nabla a_{k\varsigma} (t)\|_{N}\\
& + \|D_t a_{k\varsigma} (t)\|_N \|D\Phi_\varsigma (t) - \Id\|_0\\
&+ C \|D_t a_{k\varsigma} (t)\|_0 \|D\Phi_\varsigma (t) - \Id\|_N\\
& + C \|a_{k\varsigma} (t)\|_N \|D_t D \Phi_\varsigma (t)\|_0 + C \|a_{k\varsigma} (t)\|_0 \|D_t D \Phi_\varsigma (t)\|_N\\
\leq & C \delta_{q+1,j}^{\sfrac{1}{2}} \delta_{q,j-1}^{\sfrac{1}{2}} \lambda_q \ell^{-N}\, ,
\end{align*}
where we have used Lemma \ref{l:Phi_est} repeatedly.

Differentiating further \eqref{e:D_tDPhi} we get
\begin{align*}
D^2_t D\Phi_\varsigma &= D_t Dv_\ell (D\Phi_\varsigma- \Id) + D v_\ell D_t D\Phi\nonumber\\
 &= (D D_t v_\ell - Dv_\ell Dv_\ell) (D\Phi_\varsigma- \Id) + Dv_\ell (D D_t \Phi - Dv_\ell (D\Phi- \Id))\, .
\end{align*}
Hence using \eqref{e:D_tv_ell_N} and Lemma \ref{l:Phi_est} we also get
\begin{align*}
\|D_t^2 D\Phi_\varsigma (t)\|_N \leq & C\|D_t v_\ell (t)\|_{N+1} \|D\Phi_\varsigma (t) - \Id\|_0\\
&+ C \|D_t v_\ell (t)\|_1 \|D\Phi (t) - \Id\|_N\\
&+C \|Dv_\ell (t)\|_N \|D v_\ell (t)\|_0 \|D\Phi_\varsigma (t) - \Id\|_0\\
& + C \|Dv_\ell (t)\|_0^2 \|D\Phi_\varsigma (t)- \Id\|_N\\
& + C \|Dv_\ell (t)\|_0 \|D_t D\Phi_\varsigma (t)\|_N\\
& + C \|Dv_\ell (t)\|_N \|D_t D\Phi_\varsigma (t)\|_0\\
\leq& C\delta_{q,(j-1)_+}\lambda_q^2 \lambda_{q+1}^{-\omega}\ell^{-N} \, .
\end{align*}
Next, differentiating further \eqref{e:D_tDa} we get
\begin{align*}
D^2_t \nabla a_{k\varsigma} = & - Dv_\ell D_t \nabla a_{k\varsigma} - D_t D v_\ell \nabla a_{k\varsigma} + D_t \nabla D_t a_{k\varsigma}\\
= & - Dv_\ell \nabla D_t a_{k\varsigma} + Dv_\ell Dv_\ell \nabla a_{k\varsigma} - D D_t v_\ell \nabla a_{k\varsigma} + Dv_\ell Dv_\ell \nabla a_{k\varsigma}\\
 &+ \nabla D^2_t a_{k\varsigma} - Dv_\ell \nabla D_t a_{k\varsigma}\, .
\end{align*}
Thus, using \eqref{e:Holderproduct},
\begin{align*}
\|D_t^2 \nabla a_{k\varsigma} (t)\|_N \leq & C \|Dv_\ell (t)\|_0 \|D_t a_{k\varsigma} (t)\|_{N+1} + C \|Dv_\ell (t)\|_N \|D_t a_{k\varsigma} (t)\|_1\\
&+C \|D v_\ell (t)\|_{N}\|D v_\ell (t)\|_{0}\|a_{k\varsigma}(t)\|_{0}
+C \|D v_\ell (t)\|_{0}^2\norm{a_{k\varsigma}(t)}_{N}\\
&+C \|D_t v_\ell (t)\|_{N+1}\|a_{k\varsigma}(t)\|_1+C \|D_t v_\ell (t)\|_1 \|a_{k\varsigma}(t)\|_{N+1}\\
&+ C \norm{D_t^2 a_{k\varsigma}(t)}_{N+1}
\\
\leq& C \delta_{q+1,j} \delta_{q,j-1}^{\sfrac12} \lambda_q \lambda_{q+1} \ell^{-N-1} \, .
\end{align*}
We finally take two material derivatives of \eqref{e:L_kl} to conclude the estimate
\begin{align*}
\|D^2_t L_{k\varsigma} (t)\|_N \leq & C \|D^2_t a_{k\varsigma} (t)\|_N + C \lambda_{q+1}^{-1} \|D^2_t \nabla a_{k\varsigma} (t)\|_{N}\\
& + \|D^2_t a_{k\varsigma} (t)\|_N \|D\Phi_\varsigma (t) - \Id\|_0\\
&+ C \|D^2_t a_{k\varsigma} (t)\|_0 \|D\Phi_\varsigma (t) - \Id\|_N\\
& + C \|a_{k\varsigma} (t)\|_N \|D^2_t D\Phi_\varsigma (t)\|_0 + C \|a_{k\varsigma} (t)\|_0 \|D^2_t D\Phi_\varsigma (t)\|_N\\
& + C\|D_t a_{k\varsigma} (t)\|_N \|D_t D\Phi_\varsigma (t)\|_0\\
& + C \|D_t a_{k\varsigma} (t)\|_0 \|D_t D\Phi_\varsigma (t)\|_N\\
\leq & C \delta_{q+1, j} \delta_{q, j-1}^{\sfrac12} \lambda_q \lambda_{q+1} \ell^{-N} + C \delta_{q+1, j}^{\sfrac{1}{2}} \delta_{q, j-1} \lambda_q^2 \lambda_{q+1}^{-\omega} \ell^{-N} \\
\leq & C \delta_{q+1, j} \delta_{q, j-1}^{\sfrac12} \lambda_q \lambda_{q+1} \ell^{-N}\, .
\end{align*}
\end{proof}

\subsection{Estimates on $w_o$ and $w_c$}
In the estimates above whether a time was in a non-overlapping or overlapping region played no role.  In the two lemmata below this distinction will play an important role, in particular we obtain better estimates on the non-overlapping regions than on the overlapping ones.  

\begin{lemma}\label{l:ugly_lemma3}
Assume $t\in \supp (\chi_\varsigma)$ and set 
\[
j = \left\{
\begin{array}{ll}
j_q (\alpha_q (\varsigma)) & \mbox{if $t\in H_\varsigma$}\\ \\
\min \{j_q (\alpha_q (\varsigma)), j_q (\alpha_q (\varsigma +1))\}\qquad &
\mbox{if $t\in K_\varsigma$}\\ \\
\min \{j_q (\alpha_q (\varsigma)), j_q (\alpha_q (\varsigma -1))\}\qquad &
\mbox{if $t\in K_{\varsigma -1}$\, .}
\end{array}\right.
\]
Then we have
\begin{align}
&\|w_c (t)\|_N \leq  \bar{C} \delta_{q+2, j+1}\delta_{q+1, j}^{-\sfrac{1}{2}} \lambda_{q+1}^{N-\eps_0}\label{e:w_c_est}\\
&\|w_o (t)\|_N \leq \bar{C} \delta_{q+1, j}^{\sfrac{1}{2}} \lambda_{q+1}^N \label{e:w_o_est}\\
&\lambda_{q+2}^{-1}\|v_{q+1} (t)\|_2+\lambda_{q+1}^{-1}\|v_{q+1} (t)\|_1 + \|w_{q+1} (t)\|_0 \leq M \delta_{q+1, j}^{\sfrac{1}{2}} \label{e:final_v_est}\\
&\lambda_{q+1}^{-2} \|p_{q+1} (t)\|_2 + \lambda_{q+1}^{-1}\|p_{q+1} (t)\|_1 + \|(p_{q+1} - p_q) (t)\|_0 \leq M^2 \delta_{q+1, j}\, ,\label{e:final_p_est}
\end{align}
where $\bar{C}$ is a constant which depends only on $N$.
\end{lemma}
\begin{proof} Let $t\in \supp (\chi_\varsigma$). We then have
\begin{align}
w_o (x, t) =& \sum_k \underbrace{\chi_{\varsigma} (s) a_{k\varsigma} (x, t) \phi_{k\varsigma} (x, t) e^{i\lambda_{q+1} k\cdot x}}_{\Sigma_{k\varsigma}}\nonumber\\
& + \sum_k \underbrace{\chi_{\varsigma'} (t) a_{k\varsigma'} (x, s) \phi_{k\varsigma'} (x, t) e^{i\lambda_{q+1} k\cdot x}}_{\Sigma_{k\varsigma'}}\label{e:few_summands}
\end{align}
where 
\begin{itemize}
\item $\varsigma' = \varsigma+1$ in case $t\in K_\varsigma$;
\item $\varsigma' = \varsigma-1$ in case $t\in K_{\varsigma-1}$;
\item the second sum is in fact absent in case $t\in H_\varsigma$.
\end{itemize}
In particular $j\leq j_q (\alpha (\varsigma')), j_q (\alpha (\varsigma))$ and, by \eqref{e:asc_delta}, we easily conclude from
Lemma \ref{l:ugly_lemma} that
\begin{align}
&\|a_{k\varsigma'} (t)\|_0 + \|L_{k\varsigma'} (t)\|_0 + \|a_{k\varsigma} (t)\|_0 + \|L_{k\varsigma} (t)\|_0 \leq \bar{C} \delta^{\sfrac{1}{2}}_{q+1,j}\\
&\|a_{k\varsigma'} (t)\|_N+\|a_{k\varsigma} (t)\|_N\leq \bar{C} \delta_{q+1,j}^{\sfrac{1}{2}} \lambda_q \ell^{1-N} \qquad \mbox{for $N\geq 1$}\\
&\|L_{k\varsigma'} (t)\|_N+\|L_{k\varsigma} (t)\|_N \leq \bar{C} \delta_{q+1,j}^{\sfrac{1}{2}}\ell^{-N}\qquad \mbox{for $N\geq 1$}\\
&\|\phi_{k\varsigma'} (t)\|_N+\|\phi_{k\varsigma} (t)\|_N \leq \bar{C}  \ell^{-N} \qquad \mbox{for $N\geq 1$}\, ,
\end{align}
where $\bar{C}$ is a constant which depends only on $N$.

Therefore, for each summand in \eqref{e:few_summands} we have
\begin{align*}
\|\Sigma(t)\|_N \leq \bar{C} \delta_{q+1, j}^{\sfrac{1}{2}} \lambda_{q+1}^N + \bar{C} \delta_{q+1, j}^{\sfrac{1}{2}} \lambda_q \ell^{1-N}
+  \bar{C} \delta_{q+1, j}^{\sfrac{1}{2}} \ell^{N} \leq\bar{C} \delta_{q+1, j}^{\sfrac{1}{2}} \lambda_{q+1}^N\, ,
\end{align*}
where we used  $\lambda_q\leq\ell^{-1}\leq \lambda_{q+1}$ and the constant $\bar{C}$ depends only on $N$.

Observe that the number of summands in \eqref{e:few_summands} is at most $|\Lambda_e|+ |\Lambda_o|$, a number depending on Lemma \ref{l:split}, which is applied with $N=2$. Therefore, imposing 
$M\geq 4 (|\Lambda_e| +|\Lambda_o|) \bar{C}$, we achieve
\begin{equation}\label{e:determines_M}
\lambda_{q+1}^{-2} \|w_o (t)\|_2 + \lambda_{q+1}^{-1} \|w_o (t)\|_1 + \|w_o (t)\|_0 \leq \frac{M}{4} \delta_{q+1, j}\, .
\end{equation}

The estimate \eqref{e:w_c_est} follows from entirely analogous arguments. Indeed
\begin{align}
&w_c (x, s)\nonumber\\ 
=& \sum_k \underbrace{\chi_{\varsigma} (s) \Bigl(\frac{i}{\lambda_{q+1}}\nabla a_{k\varsigma}-a_{k\varsigma}(D\Phi_{\varsigma}-\Id)k\Bigr)\times\frac{k\times B_k}{|k|^2} \phi_{k\varsigma} (x, s) e^{i\lambda_{q+1} k\cdot x}}_{\Sigma'_{k\varsigma}}\nonumber\\
& + \sum_k \underbrace{\chi_{\varsigma'} (s) \Bigl(\frac{i}{\lambda_{q+1}}\nabla a_{k\varsigma'}-a_{k\varsigma'}(D\Phi_{\varsigma'}-\Id)k\Bigr)\times\frac{k\times B_k}{|k|^2} \phi_{k\varsigma'} (x, s) e^{i\lambda_{q+1} k\cdot x}}_{\Sigma'_{k\varsigma'}}\nonumber\, 
\end{align}
and arguing as above we conclude (using also Lemma \ref{l:Phi_est})
\begin{align}
\|\Sigma' (t)\|_N \leq &\bar{C} \delta_{q+1,j}^{\sfrac{1}{2}}\lambda_q \lambda_{q+1}^{N-1} +
C \|a (t)\|_N \|D\Phi (t)-\Id\|_0\nonumber\\
& + C \|a(t)\|_0 \|D\Phi (t)- \Id\|_N + C\|a(t)\|_0 \|D\Phi - \Id\|_0 \|\phi\|_N\nonumber\\
\leq &  \bar{C} \delta_{q+1,j}^{\sfrac{1}{2}}\lambda_q \lambda_{q+1}^{N-1} + \bar{C} M \delta_{q+1,j}^{\sfrac{1}{2}}  \lambda_{q+1}^{-\omega} \left(\lambda_q \ell^{1-N} + \ell^{-N} + \lambda_{q+1}^N\right)\nonumber\\
\leq& \bar{C}\delta_{q+1,j}^{\sfrac{1}{2}} \lambda_{q+1}^N \left(\frac{\lambda_q}{\lambda_{q+1}} + M  \lambda_{q+1}^{-\omega}\right)\nonumber\\ 
\stackrel{\eqref{e:CFL3}}{\leq} &\bar{C}  \delta_{q+2,j+1}\delta_{q+1,j}^{-\sfrac{1}{2}} \lambda_{q+1}^{N-\eps_0} \, ,
\end{align}
where in the last inequality we have assumed $\lambda_0^{\eps_0}>M$. This proves \eqref{e:w_c_est}.

Thus, assuming $\lambda_0$ is chosen sufficiently large, we obviously have
\[
\lambda_{q+1}^{-2} \|w_c (t)\|_2 + \lambda_{q+1}^{-1} \|w_c (t)\|_1 + \|w_c (t)\|_0 \leq \frac{M}{4} \delta_{q+1,j}^{\sfrac{1}{2}}\, .
\]
Next observe that $v_{q+1} - v_q = w_{q+1} = w_c + w_o$ and thus using \eqref{e:local-estimates}
\begin{align*}
&\lambda_{q+1}^{-2} \|v_{q+1} (t)\|_2+\lambda_{q+1}^{-1} \|v_{q+1} (t)\|_1 + \|v_{q+1}- v_q (t)\|_0\\
\leq &M \frac{\lambda_q}{\lambda_{q+1}} \delta^{\sfrac{1}{2}}_{q, (j-1)_+} + \frac{M}{2} \delta_{q+1, j}^{\sfrac{1}{2}}\, .
\end{align*}
On the other hand, by \eqref{e:CFL3},
\[
\frac{\lambda_q}{\lambda_{q+1}}\delta_{q, (j-1)_+}^{\sfrac{1}{2}} \leq \delta_{q+1, j}^{\sfrac{1}{2}} \lambda_{q+1}^{-\omega}
\]
and so, having $\lambda_0$ sufficiently large ensures \eqref{e:final_v_est}. 

Next, using \eqref{e:def_p_1}, we easily achieve
\begin{align}
\|p_{q+1}(t) - p_q (t)\|_0\leq & \frac{1}{2} \|w_o (t)\|_0^2 + \|w_c (t)\| \left(\frac{1}{3}\|w_c (t)\|
+ \frac{2}{3}\|w_o (t)\|_0\right)\nonumber\\
& + \frac{2}{3} \|w_{q+1} (t)\|_0 \|v-v_\ell (t)\|\nonumber\\
\leq &\frac{M^2 \delta_{q+1,j}}{32} + \frac{M^2 \delta_{q+1,j}}{16} + \frac{2 M \delta_{q+1,j}^{\sfrac{1}{2}}}{3} \bar{C} M \delta_{q,(j-1)_+}^{\sfrac{1}{2}} \lambda_q \ell \nonumber\\
\leq &\frac{3M^2}{32} \delta_{q+1,j} + \bar{C} M^2 \delta_{q+1, j}^{\sfrac{1}{2}}\delta_{q, (j-1)_+}^{\sfrac{1}{2}} \lambda_{q+1}^{-\omega}\, ,\label{e:p_est}
\end{align}
where the constant $\bar{C}$ is universal. Thus, again assuming $\lambda_0$ is large enough compared to $M$, we conclude
\[
\|p_{q+1} (t) - p_q(t)\|_0 \leq \frac{M^2}{8} \delta_{q+1,j}\, .
\]
The analogous estimates for $\lambda_{q+1}^{-1} \|p_{q+1} (t)\|_1 + \lambda_{q+1}^{-2} \|p_{q+1} (t)\|_2$ are left to the reader.
\end{proof}

\subsection{Estimates on $D_t w_o$ and $D_t w_c$}
Finally we list material derivative estimates of the principal perturbation $w_o$ and the corrector $w_c$.

\begin{lemma}\label{l:ugly_lemma4}
Assume $t\in \supp (\chi_\varsigma)$ and set 
\[
j = \left\{
\begin{array}{ll}
j_q (\alpha_q (\varsigma)) & \mbox{if $t\in H_\varsigma$}\\ \\
\min \{j_q (\alpha_q (\varsigma)), j_q (\alpha_q (\varsigma +1))\}\qquad &
\mbox{if $t\in K_\varsigma$}\\ \\
\min \{j_q (\alpha_q (\varsigma)), j_q (\alpha_q (\varsigma -1))\}\qquad &
\mbox{if $t\in K_{\varsigma -1}$\, .}
\end{array}\right.
\]
If $t\in H_\varsigma$ then
\begin{equation}\label{e:D_t_w's_good}
\|D_t w_o(t)\|_N + \|D_t w_c(t)\|_N \leq C \delta_{q+1, j}^{\sfrac{1}{2}} \delta_{q, j-1}^{\sfrac{1}{2}} \lambda_q \lambda_{q+1}^N\, .
\end{equation}
If instead $t$ belongs to $K_{\varsigma-1} \cup K_{\varsigma}$, then
\begin{equation}\label{e:D_t_w's_bad}
\|D_t w_o(t)\|_N + \|D_t w_c(t)\|_N \leq C \frac{\mu_{q+1, j}}{\eta_{q+1,j}}\, \delta_{q+1, i}^{\sfrac{1}{2}} \lambda_{q+1}^N\, .
\end{equation}
In both cases the constant $C$ depends only on $N$ and $M$.
\end{lemma}
\begin{proof} First assume $t$ belongs to the non-overlapping region.  
 We observe that, under such assumption, we have $\chi_\varsigma (t) =1$ and $\chi_{\varsigma'} (t) =0$ for any $\varsigma' \neq \varsigma$. Therefore
\[
w_o (x, t) = \sum_k w_{k\varsigma} (x,t) = \sum_k a_{k\varsigma} (x,t) e^{i\lambda_{q+1} k\cdot \Phi_\varsigma (x,t)}\, 
\]
and we have
\[
D_t w_o (x,t) = \sum_k D_t a_{k\varsigma} (x,t) e^{i\lambda_{q+1} k\cdot \Phi_\varsigma (x,t)} = \sum_k D_t a_{k\varsigma} (x,t) \phi_{k\varsigma} e^{i\lambda_{q+1} k\cdot x}\, .
\]
We thus can estimate
\begin{align*}
\|D_t w_o (t)\|_N \leq &\sum_k \|a_{k\varsigma} (t)\|_0 \lambda_{q+1}^N\\
& +C\sum_k \left(\|D_t a_{k\varsigma} (t)\|_N  + \|D_t a_{k\varsigma} (t)\|_0 \left( \lambda_{q+1}^N + \|\phi_{k\varsigma}(t)\|_N\right)\right)\\
\leq &C \sum_k \delta_{q+1, j}^{\sfrac{1}{2}} \delta_{q, j-1}^{\sfrac{1}{2}}\lambda_q \left(\lambda_{q+1}^N + \ell^{-N} + M \lambda_{q+1}^{-\omega} \ell^{1-N}\right)\\ 
\leq & C\delta_{q+1, j}^{\sfrac{1}{2}} \delta_{q, j-1}^{\sfrac{1}{2}} \lambda_q \lambda_{q+1}^N\, ,
\end{align*}
where we have used the estimates \eqref{e:phi_N} and \eqref{e:D_ta}.
Next, consider that
\[
w_{q+1} (x,t) = \sum_k L_{k\varsigma} (x,t) \phi_{k\varsigma} (x,t) e^{i\lambda_{q+1} k\cdot x}\, ,
\]
and thus 
\[
D_t w_{q+1} = \sum_k D_t L_{k\varsigma} (x,t) \phi_{k\varsigma} (x,t) e^{i\lambda_{q+1} k\cdot x}\, .
\]
We can argue as above and use this time \eqref{e:D_tL} for $D_t L_{k\varsigma}$ (which amounts to the same estimate used for $D_t a_{k\varsigma}$) to conclude
\[
\|w_{q+1} (t)\|_N \leq C \delta_{q+1,j}^{\sfrac{1}{2}} \delta_{q,j-1}^{\sfrac{1}{2}} \lambda_q \lambda_{q+1}^N\, .
\] 
Since $w_c = w_{q+1} - w_o$, \eqref{e:D_t_w's_good} follows.

\medskip

Now assume $t$ belongs to the overlapping region. In this case there are two functions in the partition of unity which are not vanishing,
namely the functions $\chi_\varsigma$ itself and another one, $\chi_{\varsigma'}$ where either $\varsigma' = \varsigma-1$ or $\varsigma' = \varsigma+1$. More precisely
\begin{align*}
w_o (x, t) =& \sum_k \chi_{\varsigma} (t) a_{k\varsigma} (x, t) \phi_{k\varsigma} (x, t) e^{i\lambda_{q+1} k\cdot x}\nonumber\\
& + \sum_k \chi_{\varsigma'} (t) a_{k\varsigma'} (x, s) \phi_{k\varsigma'} (x, t) e^{i\lambda_{q+1} k\cdot x}\, .
\end{align*}
Moreover, $\chi'_\varsigma$ and $\chi'_{\varsigma'}$ do not vanish. So, this time we have
\begin{align*}
D_t w_o (x,t) =& \sum_k \chi_{\varsigma}' (t) a_{k\varsigma} (x, t) \phi_{k\varsigma} (x, t) e^{i\lambda_{q+1} k\cdot x}\nonumber\\
& + \sum_k \chi_{\varsigma'}' (t) a_{k\varsigma'} (x, s) \phi_{k\varsigma'} (x, t) e^{i\lambda_{q+1} k\cdot x}\nonumber\\
&+\sum_k \chi_{\varsigma} (t) D_t a_{k\varsigma} (x, t) \phi_{k\varsigma} (x, t) e^{i\lambda_{q+1} k\cdot x}\nonumber\\
& + \sum_k \chi_{\varsigma'} (t) D_t a_{k\varsigma'} (x, s) \phi_{k\varsigma'} (x, t) e^{i\lambda_{q+1} k\cdot x}\, .
\end{align*}
Now, arguing as in Lemma \ref{l:ugly_lemma4} we know that $j= \min \{j_q (\alpha_q (\varsigma)), j_q (\alpha_q (\varsigma'))\}$. We can therefore conclude
\begin{align*}
\|a_{k\varsigma} (t)\|_0 + \|a_{k\varsigma'} (t)\|_0 &\leq C \delta_{q+1,j}^{\sfrac{1}{2}}\\
\|a_{k\varsigma} (t)\|_N + \|a_{k\varsigma'} (t)\|_N &\leq C \delta_{q+1,j}^{\sfrac{1}{2}} \lambda_q \ell^{1-N}\qquad N\geq 1\\
\|D_t a_{k\varsigma} (t)\|_N + \|D_t a_{k\varsigma'} (t)\|_N &\leq C \delta_{q+1,j}^{\sfrac{1}{2}} \delta_{q,j-1}^{\sfrac{1}{2}} \lambda_q \ell^{-N}\, .
\end{align*}
Thus, applying the same arguments as in the case of $t\in H_\varsigma$, we conclude
\[
\|D_t w_o (t)\|_N \leq C \delta_{q+1,j}^{\sfrac{1}{2}} \lambda_{q+1}^{N} \left( \delta_{q, j-1}^{\sfrac{1}{2}} \lambda_q + |\chi'_\varsigma (t)| + |\chi'_{\varsigma'} (t)|\right)\, .
\]
Next observe that
\begin{itemize}
\item $|\chi'_\varsigma (t)| = |\chi'_{\varsigma'} (t)| \leq C |K_\varsigma|^{-1}$ when $t\in K_\varsigma$;
\item $|\chi'_\varsigma (t)| = |\chi'_{\varsigma'} (t)| \leq C |K_{\varsigma-1}|^{-1}$ when $t\in K_{\varsigma'}$.
\end{itemize}
However, according to our choice,
\[
|K_\varsigma| = \frac{\eta_{q+1,i}}{\mu_{q+1,i}} 
\]
where $i = \min \{j_q (\alpha_q (\varsigma)), j_q (\alpha_q (\varsigma+1))\} = j$ and, similarly,
\[
|K_{\varsigma-1}| = \frac{\eta_{q+1,i}}{\mu_{q+1,i}} 
\]
where $i = \min \{j_q (\alpha_q (\varsigma-1)), j_q (\alpha_q (\varsigma))\} = j$. 

Thus,
\[
\|D_t w_o (t)\|_N \leq C \delta_{q+1,j}^{\sfrac{1}{2}} \lambda_{q+1}^{N} \left( \delta_{q, j-1}^{\sfrac{1}{2}} \lambda_q
+ \frac{\mu_{q+1,j}}{\eta_{q+1,j}}\right)\, .
\]
Observe however that by \eqref{e:CFL3} and \eqref{e:delta_lambda/mu} we have $\delta_{q,j-1}^{\sfrac{1}{2}} \lambda_q  \leq \frac{\mu_{q+1,j}}{\eta_{q+1,j}}$.

We thus conclude
\[
\|D_t w_o (t)\|_N \leq C \delta_{q+1,j}^{\sfrac{1}{2}}\frac{\mu_{q+1,j}}{\eta_{q+1,j}} \lambda_{q+1}^{-N}\, .
\]

We can use the same argument on
\begin{align*}
D_t w_{q+1} (x,t) =& \sum_k \chi_{\varsigma}' (t) L_{k\varsigma} (x, t) \phi_{k\varsigma} (x, t) e^{i\lambda_{q+1} k\cdot x}\nonumber\\
& + \sum_k \chi_{\varsigma'}' (t) L_{k\varsigma'} (x, s) \phi_{k\varsigma'} (x, t) e^{i\lambda_{q+1} k\cdot x}\nonumber\\
&+\sum_k \chi_{\varsigma} (t) D_t L_{k\varsigma} (x, t) \phi_{k\varsigma} (x, t) e^{i\lambda_{q+1} k\cdot x}\nonumber\\
& + \sum_k \chi_{\varsigma'} (t) D_t L_{k\varsigma'} (x, s) \phi_{k\varsigma'} (x, t) e^{i\lambda_{q+1} k\cdot x}\, .
\end{align*}
Using the estimates 
\begin{align*}
\|L_{k\varsigma} (t)\|_N + \|L_{k\varsigma'} (t)\|_N &\leq C \delta_{q+1,j}^{\sfrac{1}{2}}\ell^{-N}\\
\|D_t L_{k\varsigma} (t)\|_N + \|D_t L_{k\varsigma'} (t)\|_N &\leq C \delta_{q+1,j}^{\sfrac{1}{2}} \delta_{q,j-1}^{\sfrac{1}{2}} \ell^{-N}\, 
\end{align*}
we achieve the very same estimate
\[
\|D_t w_{q+1} (t)\|_N \leq C \delta_{q+1,j}^{\sfrac{1}{2}}\frac{\mu_{q+1,j}}{\eta_{q+1,j}} \lambda_{q+1}^{N}\, .
\]
Since $w_c = w_{q+1} - w_o$, this concludes the proof of  \eqref{e:D_t_w's_bad}.
\end{proof}

\section{Proof of Proposition \ref{p:inductive_step}: Reynolds stress estimates}

In order to complete the proof of Proposition \ref{p:inductive_step} it remains to estimate the new Reynolds stress $\mathring R_{q+1}$.

\begin{proposition}\label{p:R}
Assume $t\in \supp (\chi_\varsigma)$ and set
\[
i = \left\{
\begin{array}{ll}
j_q (\alpha_q (\varsigma)) & \mbox{if $t\in H_\varsigma$}\\ \\
\min \{j_q (\alpha_q (\varsigma)), j_q (\alpha_q (\varsigma +1))\}\qquad &
\mbox{if $t\in K_\varsigma$}\\ \\
\min \{j_q (\alpha_q (\varsigma)), j_q (\alpha_q (\varsigma -1))\}\qquad &
\mbox{if $t\in K_{\varsigma -1}$\, .}
\end{array}\right.
\]
If $t\in K_{\varsigma -1} \cup K_\varsigma$, namely $t\in V^{(q+1)}_0$, then
\begin{align}
&\|\mathring{R}_{q+1}(t)\|_0+\frac{1}{\lambda_{q+1}}\|\mathring{R}_{q+1} (t)\|_1 + \frac{1}{\lambda_{q+1}^2}
\|\mathring{R}_{q+1} (t)\|_2
\leq  C \delta_{q+2,0}\lambda_{q+1}^{-\eps_0} \label{e:allR_bad}\\
&\|D_t \mathring{R}_{q+1}(t)\|_0\leq C  \delta_{q+2,0} \delta_{q+1, -1}^{\sfrac{1}{2}}\ell^{-1}\, ,\label{e:Dt_R_all_bad}
\end{align}
where the constant $C$ depends only on $M$. In particular, if $\lambda_0$ is sufficiently large depending only on $\eps_0$ and $M$, then
\begin{align}
&\lambda_{q+1}^{-2} \|\mathring{R}_{q+1} (t)\|_2 + \lambda_{q+1}^{-1} \|\mathring{R}_{q+1} (t)\|_1 + \|\mathring{R}_{q+1} (t)\|_0\leq \delta_{q+2,0}\,,
\label{e:g-R_est-2}\\
&\|(\partial_t +v_{q+1} \cdot \nabla) \mathring{R}_q (t)\|_0 \leq \delta_{q+2,0} \delta_{q+1, -1}^{\sfrac{1}{2}}
\lambda_{q+1}\label{e:g-D_tR_est-2}
\end{align}
(which, given our definition of $\delta_{r,l}$ correspond to \eqref{e:g-R_est}-\eqref{e:g-D_tR_est} at step $q+1$).

If $t\in H_{\varsigma}$, namely $t\in V^{q+1}_{i+1}$, then we have
\begin{align}
&\|\mathring{R}_{q+1}(t)\|_0+\frac{1}{\lambda_{q+1}}\|\mathring{R}_{q+1}(t)\|_1 + \frac{1}{\lambda_{q+1}^2} \|\mathring{R}_{q+1} (t)\|_2 \leq C \delta_{q+2,i+1} \lambda_{q+1}^{- \eps_0}\label{e:allR_good}\\
&\|D_t \mathring{R}_{q+1}(t)\|_0\leq  C \delta_{q+2,i+1} \delta_{q+1,i}^{\sfrac{1}{2}}\ell^{-1}\, ,\label{e:Dt_R_all_good}
\end{align}
where the constant $C$ depends only on $M$. In particular, if $\lambda_0$ is sufficiently large depending only on $\eps_0$ and $M$,  then
\begin{align}
&\lambda_{q+1}^{-2} \|\mathring{R}_q+1 (t)\|_2 \lambda_{q+1}^{-1} \|\mathring{R}_{q+1} (t)\|_1 + \|\mathring{R}_{q+1} (t)\|_0 \leq \delta_{q+2,i+1}\,,
\label{e:R_est-2}\\
&\|(\partial_t +v_{q+1}\cdot \nabla) \mathring{R}_{q+1} (t)\|_0 \leq \delta_{q+2,i+1} \delta_{q+1, i}^{\sfrac{1}{2}}
\lambda_{q+1}\label{e:D_tR_est-2}
\end{align}
(which, given our definition of  the map $j_{q+1}$, correspond to \eqref{e:R_est}-\eqref{e:D_tR_est} for the step $q+1$).
\end{proposition}

\subsection{Preliminaries}
The proof will follow closely the arguments given in \cite{BDS} to prove Proposition 5.1 therein. A first important remark is that the estimates of Proposition \ref{p:CFL}, Lemma \ref{l:X_est}, Lemma \ref{l:Phi_est}, Lemma \ref{l:regularized_R_q}, Lemma \ref{l:ugly_lemma} and Lemma \ref{l:ugly_lemma2} can all be used for the analogous various quantities appearing
in our computations below with $i$ in place of $j$. The reason is that $i\leq j$ and thus the corresponding right hand sides can only become larger when we replace $j$ with $i$. The estimates of Lemma \ref{l:ugly_lemma3} and Lemma \ref{l:ugly_lemma4} can also be applied, this time because the index $j$ appearing in them {\em equals} the index $i$ defined above.

\medskip

First we show the estimates \eqref{e:g-D_tR_est-2} and \eqref{e:D_tR_est-2} follow as a consequence of \eqref{e:Dt_R_all_bad} and \eqref{e:Dt_R_all_good}. Note the decomposition
\[
\partial_t + v_{q+1} \cdot \nabla = D_t + w_{q+1}\cdot \nabla + (v_q - v_q*\psi_\ell) \cdot\nabla\, .
\] 
We thus estimate
\begin{align}
&\|(\partial_t + v_{q+1} \cdot \nabla) \mathring{R}_{q+1} (t)\|_0\nonumber\\ 
\leq &\|D_t \mathring{R}_{q+1} (t)\|_0 + (\|w_{q+1} (t)\|_0 + C\|v_q (t)\|_1 \ell) \|\mathring{R}_{q+1} (t)\|_1\nonumber\\
\leq & \|D_t \mathring{R}_{q+1} (t)\|_0 + C  \delta_{q, i-1}^{\sfrac{1}{2}} \lambda_q \|\mathring{R}_{q+1} (t)\|_1\label{e:intermediate10}
\end{align}
 Since $ \delta_{q, i-1}^{\sfrac{1}{2}} \lambda_q\leq \delta_{q, i}^{\sfrac{1}{2}} \lambda_{q+1}\leq  \delta_{q, -1}^{\sfrac{1}{2}} \lambda_{q+1}$ and $\ell^{-1} = \lambda_{q+1}^{1-\eps_0}$ we conclude  \eqref{e:D_tR_est-2} and \eqref{e:g-D_tR_est-2} from \eqref{e:Dt_R_all_bad} and \eqref{e:Dt_R_all_good} respectively, provided $\lambda_0$ is chosen large enough.

In order to  derive  \eqref{e:allR_bad}, \eqref{e:Dt_R_all_bad},\eqref{e:allR_good} and \eqref{e:Dt_R_all_good}, we will make heavy use of the parameter orderings stated in Section \ref{s:parameters}. For the particular estimates  \eqref{e:allR_bad} and \eqref{e:Dt_R_all_bad}, we will require one additional ordering which is stated in the lemma below.
\begin{lemma}\label{l:parameters3}
For $\eps_0$ satisfying \eqref{e:ell_constr}, depending only upon $b, \beta_\infty$ and $\beta_0$, then
\begin{align}
&2\delta_{q+2, 0} \lambda_{q+1}^{-2\eps_0} \geq \frac{\delta_{q+1,i}^{\sfrac{1}{2}}}{\lambda_{q+1}} \frac{\mu_{q+1,i}}{\eta_{q+1,i}}\, .
\label{e:ov_final}
\end{align}
\end{lemma}
\begin{proof}
First observe that
\[
\log_{\lambda_{q+1}} \left(\frac{\delta_{q+1,i}^{\sfrac{1}{2}}}{\eta_{q+1,i}}\right) = (b-1) \beta_\infty - b \beta_0\, ,
\]
i.e. it is independent of $i$. By \eqref{e:asc_mu} it then suffices to prove \eqref{e:ov_final} when $i=0$. Noting by definition $2\delta_{q+2,0}\geq \lambda_{q+1}^{-2b\beta_0}$ and taking logarithms, it is sufficient to show:
\[
- 2 b \beta_0 -  2\eps_0 \geq (b-1) \beta_\infty - b \beta_0 + \frac{b+1}{2b} (1-\beta_0) + \frac{b-1}{2} \beta_\infty - 1\, .
\]
The latter inequality can be rewritten as
\[
\frac{b-1}{2b} \geq \frac{2b^2 - b -1}{2b} \beta_0 + 3 \frac{b-1}{2} \beta_\infty + 2\eps_0\, ,
\]
which is equivalent to
\[
(b-1) (1 - 3 b \beta_\infty - (2b+1) \beta_0) \geq  4b \eps_0\, ,
\]
which is implied by \eqref{e:ell_constr}.
\end{proof}

\subsection{Estimates on $R^0$.}  By direct calculation we have
\begin{align*}
&\partial_t w_{q+1} (t)+(v_\ell\cdot \nabla) w_{q+1} (t) +(w_{q+1}\cdot \nabla) v_\ell (t)\\
= &\sum_{(k,\varsigma)}
\left(\chi_{\varsigma}' (t) L_{k\varsigma} (t) +\chi_{\varsigma} (t) D_tL_{k\varsigma} (t) +\chi_{\varsigma} (L_{k\varsigma} \cdot \nabla) v_\ell (t) \right)e^{ik\cdot\Phi_\varsigma}\\
=&\sum_{(k,\varsigma)}\Omega_{k\varsigma} (t) e^{i\lambda_{q+1}k\cdot x}\, ,
\end{align*}
where we write 
\[
\Omega_{k\varsigma} (t) := \left(\chi_{\varsigma}' (t) L_{k\varsigma} (t) +\chi_{\varsigma} (t) D_tL_{k\varsigma} (t) + \chi_{\varsigma} (t) (L_{k\varsigma}\cdot \nabla) v_\ell (t) \right) \phi_{k\varsigma} (t)\, .
\]

We now must distinguish two cases:
\begin{itemize}
\item[(O)] In the overlapping case $t\in K_{\varsigma-1}\cup K_\varsigma$ we have
\begin{equation}\label{e:chi'_bound}
|\chi'_\varsigma (t)| \leq C \frac{\mu_{q+1,i}}{\eta_{q+1,i}}\, .
\end{equation}
\item[(NO)] In the non-overlapping case $t\in H_\varsigma$ we have $\chi'_\varsigma (t) = 0$.
\end{itemize}

{\bf Case (O).} Applying Lemmas \ref{l:ugly_lemma} and \ref{l:ugly_lemma2} we obtain 
\begin{align}
\|\Omega_{k\varsigma} (t)\|_N\leq &C \frac{\mu_{q+1, i}}{\eta_{q+1, i}} \|L_{k\varsigma} (t)\|_N\nonumber\\
&+ \|D_t L_{k\varsigma} (t)\|_N + \|L_{k\varsigma} (t)\|_N \|v_\ell (t)\|_1 + \|L_{k\varsigma} (t)\|_0 \|v_\ell (t)\|_{N+1}\nonumber\\
& + C \|\phi_{k\varsigma} (t)\|_N \left(\|L_{k\varsigma} (t)\|_0 \left(\frac{\mu_{q+1, i}}{\eta_{q+1, i}} +
 \|v_\ell (t)\|_1\right) + \|D_t L_{k\varsigma} (t)\|_0\right)\nonumber\\
&\leq C \delta_{q+1,i}^{\sfrac{1}{2}} \ell^{-N} \left(\frac{\mu_{q+1, i}}{\eta_{q+1, i}} + \delta_{q,i-1}^{\sfrac{1}{2}} \lambda_q\right)\nonumber\\
&\stackrel{\eqref{e:CFL3}\&\eqref{e:ov_final}}{\leq}  C \delta_{q+2,0} \lambda_{q+1}^{-\eps_0}\ell^{-N-1}\, .\label{e:omega_over}
\end{align}

\medskip

{\bf Case (NO).} Similar computations yield
\begin{align}
\|\Omega_{k\varsigma} (t)\|_N\leq & C \|D_t L_{k\varsigma} (t)\|_N + \|L_{k\varsigma} (t)\|_N \|v_\ell (t)\|_1 + \|L_{k\varsigma} (t)\|_0 \|v_\ell (t)\|_{N+1}\nonumber\\
& + C \|\phi_{k\varsigma} (t)\|_N \left(\|L_{k\varsigma} (t)\|_0  \|v_\ell (t)\|_1 + \|D_t L_{k\varsigma} (t)\|_0\right)\nonumber\\
\leq & C \delta_{q+1, i}^{\sfrac{1}{2}}\delta_{q,i-1}^{\sfrac{1}{2}} \lambda_q \ell^{-N}\stackrel{\eqref{e:CFL3}}{\leq}
 C \delta_{q+2,i+1} \lambda_{q+1}^{-\eps_0} \ell^{-N-1}\, .\label{e:omega_nonover}
\end{align}

\medskip

By \eqref{e:Holderinterpolation2} we get the estimates analogous to \eqref{e:omega_over} and \eqref{e:omega_nonover} where $N$ is replaced by any
positive real number. 

We can now estimate
\begin{align*}
\|R^0 (t)\|_0 &\leq \sum_{k,\varsigma} \left\|\mathcal{R} \left(\Omega_{k\varsigma} (t) e^{i\lambda_{q+1} k\cdot x}\right)\right\|_0\\
\|R^0 (t)\|_2 \leq &\lambda_{q+1}^2 \sum_{k,\varsigma} \left\|\mathcal{R} \left(\Omega_{k\varsigma} (t) e^{i\lambda_{q+1} k\cdot x}\right)\right\|_0\\
& + \lambda_{q+1} \sum_{k,\varsigma} \left\|\mathcal{R} \left(D \Omega_{k\varsigma} (t) e^{i\lambda_{q+1} k\cdot x}\right)\right\|_0\nonumber\\
& + \sum_{k,\varsigma} \left\|\mathcal{R} \left(D^2 \Omega_{k\varsigma} (t) e^{i\lambda_{q+1} k\cdot x}\right)\right\|_0\, .
\end{align*}
Recalling that $\ell^{-1} \leq \lambda_{q+1}$,
we can now apply Lemma \ref{p:stat_phase} with $\alpha =  \eps_0$ and some $m$ (to be chosen in a moment) to conclude
\begin{align*}
\lambda_{q+1}^{-2}\|R^0 (t)\|_2 + \|R^0 (t)\|_0 & \leq C \delta_{q+2,0} \lambda_{q+1}^{-\eps_0}\left(1 + \lambda_{q+1} \left(\frac{\ell^{-1}}{\lambda_{q+1}}\right)^m\right)\\
\lambda_{q+1}^{-2}\|R^0 (t)\|_2 + \|R^0 (t)\|_0 & \leq C \delta_{q+2,i+1} \lambda_{q+1}^{-\eps_0}
\left(1 + \lambda_{q+1} \left(\frac{\ell^{-1}}{\lambda_{q+1}}\right)^m\right)
\end{align*}
respectively in the overlapping and non-overlapping case. Recalling \eqref{e:ell_choice}
it suffices to choose $m \eps_0 \geq 1$ to conclude that $\lambda_{q+1}^{-2}\|R^0 (t)\|_2 + \|R^0 (t)\|_0$ can be bounded by
the right hand sides of \eqref{e:allR_bad} and \eqref{e:allR_good} in the corresponding regions of time.

\subsection{Estimates on $D_t R^0$.} Again, by a direct calculation we have
\begin{align*}
D_t \left(\partial_tw_{q+1}+v_\ell\cdot \nabla w_{q+1}+w_{q+1}\cdot \nabla v_\ell\right) = &\sum_{(k,\varsigma)} D_t (\phi_{k\varsigma}^{-1} \Omega_{k\varsigma}) e^{k\cdot\Phi_\varsigma}\\ 
=:&\sum_{(k,\varsigma)} \Omega_{kl}' e^{i\lambda_{q+1}k\cdot x}
\end{align*}
where
\begin{align}
\Omega_{k\varsigma}' (t) = & \sum_{(k,\varsigma)}  \Bigl(\chi_{\varsigma}'' (t) L_{k\varsigma}+\chi_{\varsigma}' (t) \Bigl(2 D_tL_{k\varsigma} (t)+(L_{k\varsigma}\cdot\nabla) v_\ell (t)\bigr)\nonumber\\  
& \qquad+ \chi_{\varsigma} (t) \Bigl(D_t^2L_{k\varsigma} (t) + (D_tL_{k\varsigma}\cdot\nabla) v_\ell (t) + 
(L_{k\varsigma} \cdot\nabla) D_tv_\ell (t)\nonumber\\
&\qquad\qquad\qquad - ( (L_{k\varsigma}\cdot\nabla) v_\ell) \cdot \nabla v_\ell (t)\Bigr)\Bigr) \phi_{k\varsigma} (t)\, .\label{e:split_Dt_R0}
\end{align}
As above we distinguish the two cases (O) and (NO).

\medskip

{\bf Case (O).} In addition to \eqref{e:chi'_bound} we need 
\begin{equation}\label{e:chi''_bound}
|\chi''_\varsigma (t)|\leq C \left(\frac{\mu_{q+1,i}}{\eta_{q+1,i}}\right)^2\stackrel{\eqref{e:more-scaling2}}{\leq}
C \frac{\delta_{q+1,-1}^{\sfrac12}\lambda_{q+1}\mu_{q+1,i}}{\eta_{q+1,i}}\, .
\end{equation}

Applying the product rule \eqref{e:Holderproduct} we obtain
\begin{align}
& \|\Omega'_{k\varsigma}(t)\|_N\nonumber\\
\leq &C\left(\frac{\mu_{q+1,i}}{\eta_{q+1,i}}\right)^2 \left(\|L_{k\varsigma}\|_N + \|L_{k\varsigma}\|_0 \|\phi_{k\varsigma}\|_N\right)\nonumber\\ 
&+ C\, \frac{\mu_{q+1,i}}{\eta_{q+1,i}} \Bigl(\|D_t L_{k\varsigma} (t)\|_N + \|D_t L_{k\varsigma} (t)\|_0 \|\phi_{k\varsigma}\|_N + \|L_{k\varsigma} (t)\|_N \|v_\ell (t)\|_1\nonumber\\
&\qquad\qquad\qquad + \|L_{k\varsigma}\|_0 \|v_\ell\|_{N+1} +  \|L_{k\varsigma} (t)\|_0 \|v_\ell (t)\|_1 \|\phi_{k\varsigma} (t)\|_N \Bigr)\nonumber\\
& +\underbrace{C \|D^2_t L_{k\varsigma} (t)\|_N + \|D^2_t L_{k\varsigma} (t)\|_0 \|\phi_{k\varsigma}\|_N}_{I_1}\nonumber\\
& +\underbrace{\bigl(\|D_t L_{k\varsigma} (t)\|_N \|v_\ell (t)\|_1 + \|D_t L_{k\varsigma} (t)\|_0 \|v_\ell (t)\|_{N+1}}_{I_2}\nonumber\\
&\qquad\qquad + \underbrace{\|D_t L_{k\varsigma} (t)\|_0 \|v_\ell (t)\|_{1} \|\phi_{k\varsigma} (t)\|_N}_{I_3}\bigr)\nonumber\\
& +\underbrace{C \|L_{k\varsigma} (t)\|_N \|v_\ell (t)\|_1^2 + \|L_{k\varsigma}\|_0 \|v_\ell\|_1 \|v_\ell\|_{N+1} + \|L_{k\varsigma}\|_0 \|v_\ell\|_1^2 \|\phi_{k\varsigma}\|_N}_{I_4}\, . \label{e:omega3}
\end{align}
Using Lemmas \ref{l:ugly_lemma} and \ref{l:ugly_lemma2} we then get
\begin{align}
I_1+ I_2+I_3+I_4 \leq& C \delta_{q+1, i} \delta_{q, i-1}^{\sfrac{1}{2}} \lambda_q \lambda_{q+1} \ell^{-N} + C \delta_{q+2, i}^{\sfrac{1}{2}} \delta_{q, i-1} \lambda_q^2 \ell^{-N}\nonumber\\
 \stackrel{\eqref{e:CFL3}}{\leq} & C \delta_{q+1, i+1}\delta_{q+1, i}^{\sfrac{1}{2}}  \ell^{-N-2}\, \label{e:omega4}
\end{align}
and thus
\begin{align}
\|\Omega'_{k\varsigma} (t)\|_N \leq & C\frac{\delta_{q+1,-1}^{\sfrac12}\delta_{q+1, i}^{\sfrac{1}{2}}\lambda_{q+1}\mu_{q+1,i}}{\eta_{q+1,i}}  \ell^{-N}
+  C \delta_{q+1, i+1}\delta_{q+1, i}^{\sfrac{1}{2}}  \ell^{-N-2} \nonumber\\
\stackrel{\eqref{e:CFL3}\&\eqref{e:ov_final}}{\leq} & C\delta_{q+1,-1}^{\sfrac12}\delta_{q+2,0} \delta_{q+1, i}^{\sfrac{1}{2}} \ell^{-N-2}
 \qquad
\forall t\in K_{\varsigma-1}\cup K_\varsigma\,  .\label{e:omega5}
\end{align}

\medskip

{\bf Case (NO).} When $t\in H_\varsigma$ then the terms multiplied by $\chi'$ and $\chi''$ vanish identically in the expression of $\Omega'$. We therefore can use the same estimates of \eqref{e:omega3} and \eqref{e:omega4} to conclude
\begin{equation}\label{e:omega7}
\|\Omega'_{k\varsigma}(t)\|_N\leq \sum_{i=1}^4 I_i \leq  C \delta_{q+2, i+1}\delta_{q+1, i}^{\sfrac{1}{2}}  \ell^{-N-2}
\qquad \forall t\in H_\varsigma\, .
\end{equation}

\medskip

Next, observe that we can write
\begin{align*}
&D_tR^0=\Bigl([D_t,\mathcal{R}]+\mathcal{R}D_t\Bigr)(\partial_tw+v_\ell\cdot \nabla w+w\cdot \nabla v_\ell)\\
=&\Bigl([v_\ell,\mathcal{R}]\nabla+\mathcal{R}D_t\Bigr)(\partial_tw+v_\ell\cdot \nabla w+w\cdot \nabla v_\ell)\\
=&\sum_{(k,\varsigma)} \Bigl(\underbrace{[v_\ell,\mathcal{R}](\nabla\Omega_{k\varsigma}e^{i\lambda_{q+1}k\cdot x})}_{A_1}+ \underbrace{i\lambda_{q+1}[v_\ell\cdot k,\mathcal{R}](\Omega_{k\varsigma}e^{i\lambda_{q+1}k\cdot x})}_{A_2}\\
&\qquad\qquad+ \underbrace{\mathcal{R}(\Omega'_{k\varsigma}e^{i\lambda_{q+1}k\cdot x})}_{A_3}\Bigr)\, .
\end{align*}
We can now apply Proposition \ref{p:commutator} to $A_1$ and $A_2$ with $\alpha = \eps_0$:
conclude
\begin{align*}
& \|A_1 (t)\|_0 \leq \bar{C} \lambda_{q+1}^{\eps_0} \|v_\ell (t)\|_1 \|\Omega (t)\|_1\nonumber\\
& + C \lambda_{q+1}^{\eps_0-m} 
C \left(\|v_\ell (t)\|_{1+\eps_0} \|\Omega (t)\|_{m+\eps_0} + \|v_\ell (t)\|_{m+\eps_0} \|\Omega (t)\|_{1+\eps_0}\right)\\
& \|A_2 (t)\|_0 \leq \bar{C} \lambda_{q+1}^{\eps_0-1} \|v_\ell (t)\|_1 \|\Omega (t)\|_0\nonumber\\ 
&+ C \lambda_{q+1}^{1+\eps_0-m}\left(\|v_\ell (t)\|_{1+\eps_0} \|\Omega (t)\|_{m-1+\eps_0}
+ \|v_\ell (t)\|_{m+\eps_0} \|\Omega (t)\|_{\eps_0}\right)\, .
\end{align*}
Using \eqref{e:omega_over} and choosing $m$ such that $(m+\eps_0) \eps_0 \geq 1 + \eps_0$ we then conclude
\begin{align}
&\|A_1 (t)\|_0 + \|A_2 (t)\|_0\nonumber\\ 
\leq & C  \delta_{q, (i-1)_+}^{\sfrac{1}{2}}\delta_{q+2,0} \lambda_q \lambda_{q+1}^{-\eps_0} \left(1 + \lambda_{q+1}^{\eps_0+1} \left(\frac{\ell^{-1}}{\lambda_{q+1}}\right)^{m+\eps_0}\right)\nonumber\\
\stackrel{\eqref{e:asc_|v|_1}}{\leq}& C    \delta_{q, -1}^{\sfrac{1}{2}}\delta_{q+2,0} \ell^{-1} 
\left(1 + \lambda_{q+1}^{\eps_0+1} \left(\frac{\ell^{-1}}{\lambda_{q+1}}\right)^{m+\eps_0}\right)\nonumber\\
\leq & C   \delta_{q, -1}^{\sfrac{1}{2}}\delta_{q+2,0} \ell^{-1}\qquad \forall t\in K_{\varsigma-1}\cup K_\varsigma
\label{e:A1+A2_over}
\end{align}
and
\begin{align}
&\|A_1 (t)\|_0 + \|A_2 (t)\|_0 \nonumber\\
\leq & C \delta_{q+1,i}^{\sfrac{1}{2}}   \delta_{q+2,i+1} \ell^{-1} \left(1+ \lambda_{q+1}^{\eps_0+1}\left(\frac{\ell^{-1}}{\lambda_{q+1}}\right)^{m+\eps_0}\right)\nonumber\\
\leq & C \delta_{q+1,i}^{\sfrac{1}{2}}  \delta_{q+2,i+1} \ell^{-1}
\qquad\qquad \forall t\in H_\varsigma\, .\label{e:A1+A2_nonover}
\end{align}

\medskip

Next we apply Proposition \ref{p:stat_phase} to $A_3$ with $\alpha = \frac{\eps_0}{2}$ and $m\eps_0\geq 1$. 
Using \eqref{e:omega5} we conclude
\begin{align}
\|A_3 (t)\|_0 \leq & C \delta_{q+1,-1}^{\sfrac12}\delta_{q+2,0} \delta_{q+1, i}^{\sfrac{1}{2}}\ell^{-1} \left(1 + \lambda_{q+1} \left(\frac{\ell^{-1}}{\lambda_{q+1}}\right)^m\right)\nonumber\\
\leq & C  \delta_{q+1,-1}^{\sfrac12}\delta_{q+2,0} \delta_{q+1, i}^{\sfrac{1}{2}}\ell^{-1}\qquad
\forall t\in K_{\varsigma-1}\cup K_\varsigma\label{e:A3_over}\, .
\end{align}
Using \eqref{e:omega7} 
\begin{align}
\|A_3 (t)\|_0 \leq &  C \delta_{q+1,i}^{\sfrac{1}{2}}  \delta_{q+2,i+1} \ell^{-1} \left(1 + \lambda_{q+1} \left(\frac{\ell^{-1}}{\lambda_{q+1}}\right)^m\right)\nonumber\\
\leq &  C \delta_{q+1,i}^{\sfrac{1}{2}}  \delta_{q+2,i+1} \ell^{-1}
\qquad \forall t\in H_\varsigma\label{e:A3_nonover}\, .
\end{align}
Using \eqref{e:split_Dt_R0}, \eqref{e:A1+A2_over} and \eqref{e:A3_over} we conclude that $\|D_t R^0 (t)\|_0$ is bounded by the right hand side of \eqref{e:Dt_R_all_bad} when $t\in K_{\varsigma-1}\cup K_\varsigma$. Using  \eqref{e:split_Dt_R0}, \eqref{e:A1+A2_nonover} and \eqref{e:A3_nonover} we conclude that $\|D_t R^0 (t)\|_0$ is bounded by the right hand side of \eqref{e:Dt_R_all_good} when $t\in H_\varsigma$.

\subsection{Estimates on $R^1$} 
We recall the argument from \cite{BDS}. Using Lemma \ref{l:doublesum} we have
\begin{align*}
\div&\left(w_o\otimes w_o - \sum_\varsigma \chi_\varsigma^2\mathring{R}_\varsigma-\frac{|w_o|^2}{2}\Id\right)=\\
=&\mathop{\sum_{(k,\varsigma),(k',\varsigma')}}_{k+k'\neq 0}\chi_\varsigma\chi_{\varsigma'}\div\left(w_{k\varsigma}\otimes w_{k'\varsigma'} - \frac{w_{k\varsigma}\cdot w_{k'\varsigma'}}{2} \Id\right) = I + II
\end{align*}
where, setting $f_{k\varsigma k'\varsigma'}:=\chi_\varsigma\chi_{\varsigma'}a_{k\varsigma}a_{k'\varsigma'}\phi_{k\varsigma}\phi_{k'\varsigma'}$,
\begin{align*}
I=&\mathop{\sum_{(k,\varsigma),(k',\varsigma')}}_{k+k'\neq 0}\left(B_k\otimes B_{k'}-\tfrac{1}{2}(B_k\cdot B_{k'})\Id\right)\nabla f_{k\varsigma k'\varsigma'}e^{i\lambda_{q+1}(k+k')\cdot x}\\
II=&i\lambda_{q+1}\mathop{\sum_{(k,\varsigma),(k',\varsigma')}}_{k+k'\neq 0}f_{k\varsigma k'\varsigma'}\left(B_k\otimes B_{k'}-\tfrac{1}{2}(B_k\cdot B_{k'})\Id\right)(k+k')e^{i\lambda_{q+1}(k+k')\cdot x}\, .
\end{align*}
Concerning $II$, recall that the summation is over all $\varsigma$ and all $k\in\Lambda^e$ if $\varsigma$ is even and all $k\in\Lambda^o$ if $\varsigma$ is odd. Furthermore, both $\Lambda^e,\Lambda^o\subset\bar{\lambda}S^2\cap\Z^3$ satisfy the conditions of Lemma \ref{l:split}. Therefore we may symmetrize the summand in $II$ in $k$ and $k'$. On the other hand, recall from Lemma \ref{l:BkBk'} that 
$$
(B_k\otimes B_{k'}+B_{k'}\otimes B_k)(k+k')=(B_k\cdot B_{k'})(k+k').
$$
From this we deduce that $II=0$. We thus have the decomposition
\begin{align*}
\div&\left(w_o\otimes w_o - \sum_l \chi_{\varsigma}^2\mathring{R}_{\varsigma}-\frac{|w_o|^2}{2}\Id\right)=\\
=&\mathop{\sum_{(k,\varsigma),(k',\varsigma')}}_{k+k'\neq 0}\left(B_k\otimes B_{k'}-\tfrac{1}{2}(B_k\cdot B_{k'})\Id\right)\nabla f_{k\varsigma k\varsigma '}e^{i\lambda_{q+1}(k+k')\cdot x}\, .
\end{align*}
From the product rule \eqref{e:Holderproduct}, Lemma \ref{l:ugly_lemma} and \eqref{e:CFL3} we have 
\begin{equation}\label{e:f_N}
\norm{f_{k\varsigma k'\varsigma'}}_N\leq  C \delta_{q+2,i+1} \lambda_{q+1}^{-\eps_0} \ell^{-N} \qquad \forall N\geq 1
\end{equation}
in both the overlapping and non-overlapping regions (note that $i$ changes in the two regions, though, so to bound the
``worst'' term between $a_{k\varsigma}$ and $a_{k'\varsigma'}$). Applying Proposition \ref{p:stat_phase}(ii) with $\alpha = \eps_0$ and $m \eps_0 \geq 1$ we achieve
\begin{align}
\lambda_{q+1}^{-2} \|R^1 (t)\|_2 + \|R^1 (t)\|_0 \leq &  C \delta_{q+2,i+1} \lambda_{q+1}^{-\eps_0} 
\left(1 + \lambda_{q+1} \left(\frac{\ell^{-1}}{\lambda_{q+1}}\right)^m\right)\nonumber\\
\leq &  C \delta_{q+2,i+1} \lambda_{q+1}^{-\eps_0} \, \nonumber
\end{align}
in both the overlapping and non-overlapping regions. 

\subsection{Estimates on $D_tR^1$}
As we did for the estimate for $D_tR^0$, we make use of the identity
$D_t\mathcal{R}=[v_{\ell},\mathcal{R}]\nabla+\mathcal{R}D_t$ in order to write
\begin{align}
 &D_tR^1=\mathop{\sum_{(k,l),(k',l')}}_{k+k'\neq 0}\Bigg(\underbrace{[v_\ell,\mathcal{R}]\left(\nabla U_{k\varsigma k'\varsigma'}e^{i\lambda_{q+1}(k+k')\cdot x}\right)}_{T_1}\nonumber\\
&\qquad\qquad\qquad\qquad+\underbrace{i\lambda_{q+1}[v_\ell\cdot (k+k'),\mathcal{R}]\left(U_{k\varsigma k'\varsigma'}e^{i\lambda_{q+1}(k+k')\cdot x}\right)}_{T_2}\nonumber\\
&\qquad\qquad\qquad\qquad+\underbrace{\mathcal{R}\left(U'_{k\varsigma k'\varsigma'}e^{i\lambda_{q+1}(k+k')\cdot x}\right)}_{T_3}\Bigg)\, , \label{e:another_ugly_expression}
\end{align}
where we have set 
$U_{k\varsigma k'\varsigma'} = \left(B_k\otimes B_{k'}-\tfrac{1}{2}(B_k\cdot B_{k'})\Id\right)\nabla f_{k\varsigma k'\varsigma'}$ and
\begin{align*}
D_t\div\bigl(w_o\otimes w_o - \sum_l \chi_{\varsigma}^2\mathring{R}_{\varsigma}- \frac{|w_o|^2}{2}\Id\bigr)&=\mathop{\sum_{(k,\varsigma),(k',\varsigma')}}_{k+k'\neq 0}U'_{k\varsigma k'\varsigma'}e^{i\lambda_{q+1}(k+k')\cdot x}.
\end{align*}

Clearly we can use \eqref{e:f_N} to estimate
\begin{equation}\label{e:U_N}
\|U_{k\varsigma k'\varsigma'} (t)\|_N \leq \|f_{k\varsigma k'\varsigma'}\|_{N+1} \leq C \delta_{q+2,i+1} \lambda_{q+1}^{-\eps_0} \ell^{-N-1}\, ,
\end{equation}
where such estimate is valid in both the overlapping and non-overlapping regions.

In order to estimate $U'_{k\varsigma k'\varsigma'}$, we note the identity
\begin{align*}
\nabla &f_{k\varsigma k'\varsigma'}\,e^{i\lambda_{q+1}(k+k')\cdot x}=\chi_{\varsigma}\chi_{\varsigma'}\left(a_{k\varsigma}\nabla a_{k'\varsigma'}+a_{k'\varsigma'}\nabla a_{k\varsigma}\right)e^{i\lambda_{q+1}(k\cdot\Phi_\varsigma+k'\cdot\Phi_{\varsigma'})}
+\\
&+i\lambda_{q+1}\chi_{\varsigma}\chi_{\varsigma'}a_{k\varsigma}a_{k'\varsigma'}\left((D\Phi_\varsigma-\Id)k+(D\Phi_{\varsigma'}-\Id)k'\right)e^{i\lambda_{q+1}(k\cdot\Phi_\varsigma+k'\cdot\Phi_{\varsigma'})}\\
&=U''_{k\varsigma k'\varsigma'}e^{i\lambda_{q+1}(k\cdot\Phi_\varsigma+k'\cdot\Phi_{\varsigma'})}
\end{align*}
Thus we conclude
\begin{align}\label{e:ugly_comp1}
U'_{k\varsigma k'\varsigma'}(t) = \phi_{k\varsigma}(t)\phi_{k'\varsigma'}(t)D_t U''_{k\varsigma k'\varsigma'}(t)
\end{align}
and
\begin{align}
D_tU''_{k\varsigma k'\varsigma'}(t) &=\underbrace{(\chi_\varsigma\chi_\varsigma)'\left(a_{k\varsigma}\nabla a_{k'\varsigma'}+a_{k'\varsigma'}\nabla a_{k\varsigma}\right)}_{\Sigma^1_{k\varsigma k'\varsigma'}}\nonumber\\
&+\underbrace{i\lambda_{q+1}(\chi_\varsigma\chi_{\varsigma'})'a_{k\varsigma}a_{k'\varsigma'}\left((D\Phi_\varsigma-\Id)k+(D\Phi_{\varsigma'}-\Id)k'\right)}_{\Sigma^2_{k\varsigma k'\varsigma'}}\nonumber\\
&+\underbrace{\chi_\varsigma\chi_{\varsigma'}D_t \left(a_{k\varsigma} \nabla a_{k'\varsigma'}+a_{k'\varsigma'} \nabla a_{k\varsigma}\right)}_{\Sigma^3_{k\varsigma k'\varsigma'}}\nonumber\\
&+\underbrace{i\lambda_{q+1}\chi_\varsigma\chi_{\varsigma'} D_t (a_{k\varsigma}a_{k'\varsigma'})\left((D\Phi_\varsigma-\Id)k+(D\Phi_{\varsigma'}-\Id)k'\right)}_{\Sigma^4_{k\varsigma k'\varsigma'}}\nonumber\\
&+\underbrace{i\lambda_{q+1}\chi_\varsigma\chi_{\varsigma'} a_{k\varsigma}a_{k'\varsigma'} \left(D_t D\Phi_\varsigma k + D_t D\Phi_{\varsigma'} k'\right)}_{\Sigma^5_{k\varsigma k'\varsigma'}}\, .\label{e:ugly_comp2}
\end{align}

\medskip

{\bf Case (NO).} When $t\in H_\varsigma$, $\Sigma^1 = \Sigma^2 = 0$. As for the remaining terms, we got the following estimates (and observe that
they hold for both the non-overlapping and the overlapping case!). First, recall that
\[
D_t \nabla a_{k\varsigma} = \nabla D_t a_{k\varsigma} - Dv_\ell^T \nabla a_{k\varsigma}\, .
\]
Therefore we can use Lemma \ref{l:ugly_lemma} and Lemma \ref{l:ugly_lemma2} to bound
\begin{align}
&\|\Sigma^3_{k\varsigma k'\varsigma'} (t)\|_N\nonumber\\ 
\leq & C\|D_t a_{k\varsigma} (t)\|_N \|a_{k'\varsigma'} (t)\|_1 + C \|D_t a_{k\varsigma} (t)\|_0 \|a_{k'\varsigma'} (t)\|_{N+1}\nonumber\\
& + C\|D_t a_{k'\varsigma'} (t)\|_N \|a_{k\varsigma} (t)\|_1 + C \|D_t a_{k'\varsigma'} (t)\|_0 \|a_{k\varsigma} (t)\|_{N+1}\nonumber\\
&+ C\|a_{k\varsigma} (t)\|_N \|D_t a_{k'\varsigma'} (t)\|_1 + C\|a_{k\varsigma} (t)\|_0 \|D_t a_{k'\varsigma'} (t)\|_{N+1}\nonumber\\
&+ C\|a_{k'\varsigma'} (t)\|_N \|D_t a_{k\varsigma} (t)\|_1 + C\|a_{k'\varsigma}'\|_0 \|D_t a_{k\varsigma} (t)\|_{N+1}\nonumber\\
& + C \|D v_\ell (t)\|_{N} \left(\|a_{k\varsigma} (t)\|_0 \|a_{k'\varsigma'} (t)\|_1 + \|a_{k\varsigma} (t)\|_1 \|a_{k'\varsigma'} (t)\|_0\right)\nonumber\\
&+C \|Dv_\ell (t)\|_0 \bigl(\|a_{k\varsigma} (t)\|_0 \|a_{k'\varsigma'} (t)\|_{N+1}+ \|a_{k\varsigma} (t)\|_N \|a_{k'\varsigma'} (t)\|_1\nonumber\\
&\qquad\qquad + \|a_{k'\varsigma'} (t)\|_0 \|a_{k\varsigma} (t)\|_{N+1}+ \|a_{k'\varsigma'} (t)\|_N \|a_{k\varsigma} (t)\|_1\bigr)\nonumber\\
\stackrel{\eqref{e:CFL3}}{\leq} & C \delta_{q+1, i}^{\sfrac{1}{2}}\delta_{q+2, i+1} \ell^{-N-2}\, .\label{e:Sigma3}
\end{align}
Similarly we can use Lemma \ref{l:Phi_est}, Lemma \ref{l:ugly_lemma} and Lemma \ref{l:ugly_lemma2} to bound
\begin{align}
&\|\Sigma^4_{k\varsigma k'\varsigma'} (t)\|_N\nonumber\\ 
\leq & C \lambda_{q+1}^{1-\omega} \Bigl(\|D_t (a_{k\varsigma} a_{k'\varsigma'}) (t)\|_0
+ \ell^{-N}  \|D_t (a_{k\varsigma} a_{k'\varsigma'}) (t)\|_N\Bigr)\nonumber\\
\leq & C \lambda_{q+1}^{1-\omega} \Bigl( \|D_t a_{k\varsigma} (t)\|_0 \|a_{k'\varsigma'} (t)\|_N + \|D_t a_{k\varsigma} (t)\|_N \|a_{k'\varsigma'} (t)\|_0
\nonumber\\
&\qquad\qquad + \|D_t a_{k'\varsigma'} (t)\|_0 \|a_{k\varsigma} (t)\|_N + \|D_t a_{k'\varsigma'} (t)\|_N \|a_{k\varsigma} (t)\|_0\Bigr)\nonumber\\
& + C \lambda_{q+1}^{1-\omega} \ell^{-N} \Bigl( \|D_t a_{k'\varsigma'} (t)\|_0 \|a_{k\varsigma} (t)\|_0 + \|D_t a_{k\varsigma} (t)\|_0 \|a_{k'\varsigma'} (t)\|_0\Bigr)\nonumber\\
\leq &C \lambda_{q+1}^{1-\omega} \ell^{-N} \delta_{q+1,i} \delta_{q, i-1}^{\sfrac{1}{2}} 
\lambda_q\stackrel{\eqref{e:CFL3}}{\leq} C \delta_{q+1, i}^{\sfrac{1}{2}}\delta_{q+2, i+1} \ell^{-N-2}\, .\label{e:Sigma4}
\end{align}
Next, we use Lemma \ref{l:Phi_est} and Lemma \ref{l:ugly_lemma} to bound
\begin{align}
\|\Sigma^5_{k\varsigma k'\varsigma'} (t)\|_N \leq & \lambda_{q+1} \delta_{q+1, i} \left(\|D_t D\Phi_\varsigma\|_N + \|D_t D\Phi_{\varsigma'}\|_N \right)\nonumber\\
& + \lambda_{q+1} \delta_{q+1, i} \ell^{-N} \left(\|D_t D\Phi_\varsigma\|_0 + \|D_t D\Phi_{\varsigma'}\|_0 \right)\nonumber\\
\leq & C \delta_{q+1,i} \delta_{q,i-1}^{\sfrac{1}{2}} \lambda_q \ell^{-N} \lambda_{q+1}^{1-\omega}
\stackrel{\eqref{e:CFL3}}{\leq} C \delta_{q+1, i}^{\sfrac{1}{2}}\delta_{q+2, i+1} 
 \ell^{-N-2}\, .\label{e:Sigma5}
\end{align}
So we finally reach
\begin{equation}\label{e:Sigma3-4-5}
\|\Sigma^3_{k\varsigma k'\varsigma'} (t)\|_N+\|\Sigma^4_{k\varsigma k'\varsigma'} (t)\|_N +\|\Sigma^5_{k\varsigma k'\varsigma'} (t)\|_N
\leq C \delta_{q+1, i}^{\sfrac{1}{2}}\delta_{q+2, i+1} \ell^{-N-2}\, ,
\end{equation}
which is valid for times in both the overlapping and non-overlapping regions.

Thus, for the non-overlapping region we conclude
\begin{align}
\|U'_{k\varsigma k'\varsigma'} (t)\|_N\leq & C \delta_{q+1, i}^{\sfrac{1}{2}}\delta_{q+2, i+1}  \ell^{-N-2}
 + \delta_{q+1, i}^{\sfrac{1}{2}}\delta_{q+2, i+1}  \ell^{-N} \|\phi_{k\varsigma} (t) \phi_{k'\varsigma'} (t)\|_N\nonumber\\
\stackrel{\eqref{e:CFL3}}{\leq} & C \delta_{q+1, i}^{\sfrac{1}{2}}\delta_{q+2, i+1}  \ell^{-N-2} \qquad\qquad\qquad\qquad \forall t\in H_\varsigma\, . \label{e:U'_nonov}
\end{align}

\medskip

{\bf Case (O).} We use \eqref{e:chi'_bound} and Lemma \ref{l:ugly_lemma} to bound
\begin{align}
\|\Sigma^1_{k\varsigma k'\varsigma'} (t)\|_N \leq & C \frac{\mu_{q+1,i}}{\eta_{q+1,i}} \Bigl(\|a_{k\varsigma} (t)\|_N \|a_{k'\varsigma'} (t)\|_1 + 
\|a_{k\varsigma} (t)\|_0 \|a_{k'\varsigma'} (t)\|_{N+1}\nonumber\\
&\qquad\qquad + \|a_{k\varsigma} (t)\|_{N+1} \|a_{k'\varsigma'} (t)\|_0 +  \|a_{k\varsigma} (t)\|_{1} \|a_{k'\varsigma'} (t)\|_N \Bigr)\nonumber\\
\leq & C\frac{\mu_{q+1,i}}{\eta_{q+1,i}} \delta_{q+1,i} \lambda_q \ell^{-N}\\
\stackrel{\eqref{e:ov_final}}{\leq}& C\delta_{q+1,i}^{\sfrac12}\delta_{q+2,0}\lambda_q\lambda_{q+1}^{-\eps_0} \ell^{-N-1}\\
\stackrel{\eqref{e:asc_|v|_1}}{\leq} & C\delta_{q+1,0}^{\sfrac12}\delta_{q+2,0} \ell^{-N-2} \, .\label{e:Sigma_1}
\end{align}
Similarly we use Lemma \ref{l:Phi_est}, Lemma \ref{l:ugly_lemma} and \eqref{e:chi'_bound} to get
\begin{align}
\|\Sigma^2_{k\varsigma k'\varsigma'} (t)\|_N \leq & C \lambda_{q+1}^{1-\omega} \frac{\mu_{q+1,i}}{\eta_{q+1,i}} \left(\|a_{k\varsigma} (t) a_{k'\varsigma'} (t)\|_N + \ell^{-N} \|a_{k\varsigma} (t) a_{k'\varsigma'} (t)\|_0\right)\nonumber\\
\leq & C \frac{\mu_{q+1,i}}{\eta_{q+1,i}} \delta_{q+1, i}\lambda_{q+1}^{1-\omega} \ell^{-N}\\
\stackrel{\eqref{e:CFL3}\&\eqref{e:more-scaling2}}{\leq} & C\delta_{q+1,0}\delta_{q+2,i+1}\ell^{-N-2}\\
\stackrel{\eqref{e:asc_delta}}{\leq} & C\delta_{q+1,0}^{\sfrac12}\delta_{q+2,0} \ell^{-N-2}\, .
\end{align}
Since \eqref{e:Sigma3-4-5} is valid for $t\in K_{\varsigma-1}\cup K_\varsigma$, we conclude
\begin{align}
\|D_t U''_{k\varsigma k'\varsigma'} (t)\|_N \leq & C\delta_{q+1,0}^{\sfrac12}\delta_{q+2,0} \ell^{-N-2} ,
\end{align}
where we used \eqref{e:asc_delta}.
Thus, we finally reach
\begin{align}
\|U'_{k\varsigma k'\varsigma'} (t)\|_N \leq & C\delta_{q+1,0}^{\sfrac12}\delta_{q+2,0} \ell^{-N-2}\qquad \forall t\in K_{\varsigma-1}\cup K_\varsigma\, .
\label{e:U'_ov}
\end{align}

\medskip

We can now apply Proposition \ref{p:commutator} to the terms $T_1$ and $T_2$ in \eqref{e:another_ugly_expression} and Proposition \ref{p:stat_phase} to the term in $T_3$. In both cases we apply them with $\alpha = \eps_0$ and $m$ large enough. For $t\in H_\varsigma$ we then get:
\begin{align}
\|D_t R^1 (t)\|_0 \leq &  C  \delta_{q,(i-1)_+}^{\sfrac12}\delta_{q+2,i+1} \lambda_q\lambda_{q+1}^{-2\eps_0}  \left(1 + \lambda_{q+1}^{\eps_0+1} \left(\frac{\ell^{-1}}{\lambda_{q+1}}\right)^m\right)\nonumber\\
&+ C \delta_{q,(i-1)_+}^{\sfrac12}\delta_{q+2,i+1} \lambda_q \lambda_{q+1}^{-\eps_0}  \left(1 + \lambda_{q+1}^{\eps_0+1} \left(\frac{\ell^{-1}}{\lambda_{q+1}}\right)^m\right)\nonumber\\
&+ C \delta_{q+1, i}^{\sfrac{1}{2}}\delta_{q+2, i+1}  \ell^{-1} \left(1 + \lambda_{q+1} \left(\frac{\ell^{-1}}{\lambda_{q+1}}\right)^m\right)\nonumber\\
 \stackrel{\eqref{e:CFL3}}{\leq} & C \delta_{q+1, i}^{\sfrac{1}{2}}\delta_{q+2, i+1}  \ell^{-1}
\left(1 + \lambda_{q+1} \left(\frac{\ell^{-1}}{\lambda_{q+1}}\right)^m\right)\nonumber\\
\leq & C \delta_{q+1, i}^{\sfrac{1}{2}}\delta_{q+2, i+1}  \ell^{-1}\, ,\label{e:D_tR^2_nonov}
\end{align}
under the assumption that $m\eps_0 \geq 1 + \eps_0$. Clearly the right hand side of \eqref{e:D_tR^2_nonov} is bounded by the right hand side of \eqref{e:Dt_R_all_good}.

As for $t\in K_{\varsigma-1}\cup K_\varsigma$, we instead get
\begin{align}
\|D_t R^1 (t)\|_0 \leq & C  \delta_{q,(i-1)_+}^{\sfrac12}\delta_{q+2,i+1} \lambda_q\lambda_{q+1}^{-2\eps_0}  \left(1 + \lambda_{q+1}^{\eps_0+1} \left(\frac{\ell^{-1}}{\lambda_{q+1}}\right)^m\right)\nonumber\\
&+ C \delta_{q,(i-1)_+}^{\sfrac12}\delta_{q+2,i+1} \lambda_q \lambda_{q+1}^{-\eps_0}  \left(1 + \lambda_{q+1}^{\eps_0+1} \left(\frac{\ell^{-1}}{\lambda_{q+1}}\right)^m\right)\nonumber\\&+ C\delta_{q+1,0}^{\sfrac12}\delta_{q+2,0} \ell^{-1} \left(1+\left(\frac{\ell^{-1}}{\lambda_{q+1}}\right)^m\right)\nonumber\\
\stackrel{\eqref{e:CFL3}\&\eqref{e:ov_final}}{\leq} & C\delta_{q+1,0}^{\sfrac12}\delta_{q+2,0} \ell^{-1}\, . \label{e:D_tR^2_ov}
\end{align}
The right hand side in \eqref{e:D_tR^2_ov} is obviously bounded by the right hand side of \eqref{e:Dt_R_all_bad}. 

\subsection{Estimates on $R^2$ and $D_t R^2$} 
Using Lemma \ref{l:ugly_lemma3}, \eqref{e:CFL3} and product estimates we have
\begin{align*}
\|R^2(t)\|_N &\leq C(\|w_c(t)\|_N\|w_c(t)\|_0 + \|w_o(t)\|_N \|w_c(t)\|_0+\\
&\qquad\|w_o(t)\|_0 \|w_c(t)\|_N)\\
&\leq C \delta_{q+2,i+1} \lambda_{q+1}^{N-\eps_0}\, .
\end{align*}
In fact this estimate holds in both the overlapping and non-overlapping region. By \eqref{e:asc_delta} the right hand side is bounded by both \eqref{e:allR_good} and \eqref{e:allR_bad}.

Next, assume that $t\in H_\varsigma$. Then with the Lemmas \ref{l:ugly_lemma3} and \ref{l:ugly_lemma4} we achieve
\begin{align*}
\norm{D_t R^2(t)}_0&\leq C\norm{D_t w_c(t)}_0(\norm{w_o(t)}_0+\norm{w_c(t)}_0) + C\|D_t w_o(t)\|_0 \|w_c(t)\|_0\\&
\leq C \delta_{q+1,i} \delta_{q,i-1}^{\sfrac{1}{2}}\lambda_q \stackrel{\eqref{e:CFL3}}{\leq}
 C \delta_{q+2,i+1} \lambda_{q+1}^{N-2\eps_0}\, .
\end{align*}
The latter is obviously bounded by the right hand side of \eqref{e:Dt_R_all_good}.

Assume now $t\in K_{\varsigma-1}\cup K_\varsigma$. Again with the Lemmas \ref{l:ugly_lemma3} and \ref{l:ugly_lemma4} we achieve
\begin{align*}
\norm{D_t R^2(t)}_0&\leq C\norm{D_t w_c(t)}_0(\norm{w_o(t)}_0+\norm{w_c(t)}_0) + C\|D_t w_o(t)\|_0 \|w_c(t)\|_0\\
&\leq C \frac{\mu_{q+1,i}}{\eta_{q+1,i}} \delta_{q+1,i}\stackrel{\eqref{e:ov_final}}{\leq} \delta_{q+1}^{\sfrac 12}\delta_{q+2,0}\lambda_{q+1}^{-\eps_0}\ell^{-1} \, ,
\end{align*}
which is obviously less than the right hand side of \eqref{e:Dt_R_all_bad}.

\subsection{Estimates on $R^3$ and $D_t R^3$} Using again Lemma \ref{l:ugly_lemma} and obvious estimates on the convolution we establish
\begin{equation}
\|R^3 (t)\|_0 \leq \|w (t)\|_0 \|v_q - v_\ell\|_0 \leq C \delta_{q+1,i}^{\sfrac{1}{2}} \delta^{\sfrac{1}{2}}_{q,(i-1)_+} \lambda_q \ell
\stackrel{\eqref{e:CFL3}}{\leq} C \delta_{q+2,i+1} \lambda_{q+1}^{-\eps_0}
\end{equation}
and
\begin{align}
\|R^3 (t)\|_2 \leq &\|w (t)\|_2 \|v_q - v_\ell\|_0 + \|w (t)\|_0 \|v_q - v_\ell\|_2\nonumber\\ 
\leq &C \delta_{q+1,i}^{\sfrac{1}{2}} \delta_{q,(i-1)_+} \lambda_q
\left( \lambda_{q+1}^2 \ell + \lambda_q\right)\nonumber\\
\leq & C \delta_{q+1,i}^{\sfrac{1}{2}} \delta_{q,(i-1)_+}^{\sfrac{1}{2}} \lambda_q \lambda_{q+1}^2 \ell
\stackrel{\eqref{e:CFL3}}{\leq} C \delta_{q+2,i+1} \lambda_{q+1}^{2-\eps_0}\, .
\end{align}
Thus we conclude
\begin{align}
\lambda_{q+1}^{-2} \|R^3 (t)\|_2 + \|R^3 (t)\|_0 \leq C  \delta_{q+2,i+1} \lambda_{q+1}^{-\eps_0} \, .
\end{align}
Again by \eqref{e:asc_delta} the latter is bounded by both the right hand sides of \eqref{e:allR_good} and \eqref{e:allR_bad}.

Next we estimate
\begin{align}
\|D_t R^3(t)\|_0 &\leq \underbrace{\|v_q-v_\ell\|_0 \|D_t w\|_0}_{S_1} + \underbrace{\|D_t v_q  - D_t v_\ell\|_0 \|w\|_0}_{S_2}\label{e:DtR_3}
\end{align}
For $t\in H_\varsigma$ we have
\begin{align}
S_1 \leq & C \delta_{q, i-1}^{\sfrac{1}{2}} \lambda_q \ell \delta_{q+1, i}^{\sfrac{1}{2}} \delta_{q, i-1}^{\sfrac{1}{2}} \lambda_q
\stackrel{\eqref{e:CFL3}}{\leq} C \delta_{q+1, i}^{\sfrac{1}{2}} \delta_{q+2,i+1} \ell_{q+1}^{-1}
\label{e:good_good}
\end{align}
For $t\in K_{\varsigma-1} \cup K_\varsigma$ we instead have
\begin{align}
S_1 \leq & C \delta_{q, i-1}^{\sfrac{1}{2}} \lambda_q \ell \delta_{q+1, i}^{\sfrac{1}{2}} \frac{\mu_{q+1,i}}{\eta_{q+1,i}}
  \stackrel{\eqref{e:ov_final}}{\leq} C \delta_{ q, i-1}^{\sfrac{1}{2}} \delta_{q+2,0} \lambda_q \ell \lambda_{q+1}\label{e:good_bad}.
\end{align}
Concerning $S_2$, we first write
\begin{align}
\|D_t v_q - D_t v_\ell\|_0 \leq & \underbrace{\|(v_\ell \cdot \nabla) v_q - (v_q \cdot \nabla v_q)\|_0}_{S_{21}}\nonumber\\
&+ \underbrace{\|{\rm div}\, ( v_q \otimes v_q) * \psi_\ell - {\rm div}\, (v_\ell \otimes v_\ell)\|_0}_{S_{22}}\nonumber\\
& + \underbrace{\|(\partial_t v_q + (v_q \cdot \nabla) v_q) - (\partial_t v_q + (v_q \cdot \nabla) v_q) * \psi_\ell\|_0}_{S_{23}}\, .\nonumber
\end{align}
We subsequently estimate
\begin{equation}
S_{21} \leq \|v_\ell - v_q\|_0 \|v_q\|_1 \leq C \delta_{q, i-1} \lambda_q^2 \ell
\end{equation}
and, using Proposition \ref{p:CET},
\begin{equation}
S_{22} \leq C \|v_q\|_1^2 \ell \leq C\delta_{q, i-1} \lambda_q^2 \ell\, .
\end{equation}
As for $S_{23}$ we first observe that
\begin{equation}
S_{23} \leq C \ell \|\partial_t v_q + {\rm div}\, (v_q\otimes v_q)\|_1\, .
\end{equation}
We then use the equation \eqref{e:euler_reynolds} to achieve
\begin{align}
\|D (\partial_t v_q + {\rm div}\, (v_q\otimes v_q))\|_1 \leq & C \|p_q\|_2 + C \|\mathring{R}_q\|_2
\leq  C \delta_{q, i-1} \lambda_q^2\, .
\end{align}
Summarizing we conclude
\begin{equation}
\|D_t v_q - D_t v_\ell\|_0 \leq C \delta_{q,i-1} \lambda_1^2 \ell\, .
\end{equation}
Thus, we finally conclude
\begin{align}
S_2 \leq & C \delta_{q,i-1} \lambda_q^2 \ell \delta_{q+1,i}^{\sfrac{1}{2}} \stackrel{\eqref{e:CFL3}}
\leq C \delta_{q+1, i}^{\sfrac{1}{2}} \delta_{q+2,i+1} \ell^{-1}\, .\label{e:bad}
\end{align}
We then use \eqref{e:good_good} and \eqref{e:bad} to achieve
\begin{equation}
\|D_t R^3 (t)\|_0 \leq C \delta_{q+1, i}^{\sfrac{1}{2}} \delta_{q+2,i+1} \ell^{-1} \qquad \qquad \forall t\in H_\varsigma\, ,
\end{equation}
whereas we use \eqref{e:good_bad} and \eqref{e:bad} to conclude
\begin{equation}\label{e:final_Dt_R3}
\|D_t R^3 (t)\|_0 \leq  C \delta_{q+1, i}^{\sfrac{1}{2}} \delta_{q+2,i+1} \ell^{-1} + C \delta_{ q, i-1}^{\sfrac{1}{2}} \delta_{q+2,0} \lambda_q\ell \lambda_{q+1} \quad \forall t\in K_{\varsigma-1} \cup K_\varsigma\, .
\end{equation}
By \eqref{e:asc_delta} the first summand in the right hand side of \eqref{e:final_Dt_R3} is bounded by the right
hand side of \eqref{e:Dt_R_all_bad}. As for the second summand, we use \eqref{e:CFL3} and \eqref{e:asc_delta} to bound
\begin{align*}
\delta^{\sfrac{1}{2}}_{q,i-1} \lambda_q \ell \lambda_{q+1} 
\stackrel{\eqref{e:CFL3}}{\leq}& \delta_{q+1, i}^{\sfrac{1}{2}} \lambda_{q+1}^{1-2\eps_0} \ell 
\lambda_{q+1}
= \delta_{q+1,i}^{\sfrac{1}{2}} \ell^{-1} \stackrel{\eqref{e:asc_delta}}{\leq} \delta_{q+1, -1}^{\sfrac{1}{2}} \ell^{-1}\, .
\end{align*}

\subsection{Estimates on $R^4$ and $D_t R^4$} 
By Lemma \ref{l:regularized_R_q} we have
\[
\|R^4\|\leq \|\mathring{R}_\varsigma (t) -\mathring{R}_q (t)\|_0 \leq C \delta_{q+1,j} \lambda_q \ell\stackrel{\eqref{e:CFL3}}{\leq}  C \delta_{q+2,i+1} \lambda_{q+1}^{-\eps_0}
\]
On the other hand, using also \eqref{e:R_est}, we can estimate
\begin{align}
\lambda_{q+1}^{-2} \|R^4 (t)\|_2 \leq & \lambda_{q+1}^{-2} \left(\|\mathring{R}_q\|_2 + \|\mathring{R}_\varsigma (t)\|_0\right)\nonumber\\
&\leq \delta_{q+1, i} \left(\frac{\lambda_q^2}{\lambda_{q+1}^2} + \frac{\lambda_q \ell^{-1}}{\lambda_{q+1}^2}\right)
\leq \delta_{q+1, i} \lambda_q \ell\nonumber\\
&\stackrel{\eqref{e:CFL3}}{\leq} C \delta_{q+2,i+1} \lambda_{q+1}^{-\eps_0}
\end{align}
Finally we can use \eqref{e:D_tR_est} and Lemma \ref{l:regularized_R_q} to conclude
\begin{align}
\|D_t R^4 (t)\|_0 \leq& C \delta_{q+1,i} \delta_{q,i-1}^{\sfrac{1}{2}} \lambda_q
\stackrel{\eqref{e:CFL3}}{\leq} C  \delta_{q+1,i}^{\sfrac{1}{2}}\delta_{q+2,i+1} \ell^{-1} \, .
\end{align}
As in the previous steps we use  \eqref{e:asc_delta} to conclude the proof.

\section{Proof of Theorem \ref{t:main}}

\subsection{Choice of global parameters}
Without loss of generality we assume $\eps <  \frac13$.
We start by choosing our parameters in order to satisfy the hypothesis of Proposition \ref{p:inductive_step}. First choose $\beta_\infty$ and $\beta_{-1}$ such that
\begin{equation}\label{e:choice_of_betainfty}
\frac13-\frac{\eps}{4}< \frac{1}{3} - \frac{\eps}{16} =: \beta_\infty<\frac{1}{3},
\end{equation}
and
\begin{equation}\label{e:choice_of_betaminus1}
\beta_{-1}:=\frac{\eps}{16}.
\end{equation}
By definition, we have $b\beta_0=\beta_{-1}+(b-1)\beta_{\infty}$ and hence we have 
\[
1 - 3b (\beta_\infty+\beta_0) = 1 - 3 (\beta_\infty+\beta_{-1}) - 4(b-1) \beta_\infty  > \frac{3\eps}{8} - \frac{4(b-1)}{3} 
\]
and
\[
5\beta_\infty-3b\beta_0-b=5\beta_\infty-3\beta_{-1}-3(b-1)\beta_{\infty}-b=\frac{8}{3}-\frac{\eps}{2}+\frac{3(b-1)\eps}{16}>
\frac{5}{2}-2b\]
Hence choosing $b>1$ (depending on $\eps$) close enough to $1$ we obtain \eqref{e:condition} and \eqref{e:condition2}.

\subsection{The initial triple  $(v_0,p_0,\mathring R_0)$}

We now define the initial triple which will begin the iteration scheme. Our proof follows closely the construction provided in \cite{Buckmaster}

First let $\nu_0: \mathbb R \to \mathbb R$ be a smooth non-negative function, compactly supported on the interval $[\sfrac38,\sfrac58]$ and identically equal to $1$ on $[-\sfrac7{16},\sfrac9{16}]$. The initial velocity $v_0$ is then defined to be the divergence-free vector field 
\begin{equation*}
v_0 (t,x) := \lambda_0^{-\beta_0}\nu_0(t)(\cos(\lambda_0 x_3),\sin(\lambda_0 x_3),0),
\end{equation*}
where here we use the notation $x=(x_1,x_2,x_3)$. The initial pressure $p_0$ will be set to be identically zero.  Then if we define
\[
\mathring R_0 =  \lambda_0^{-\beta_0-1} \nu'_0(t)
 \begin{pmatrix}
  0 & 0 & \sin(\lambda_0 x_3) \\
  0 & 0 & -\cos(\lambda_0 x_3) \\
  \sin(\lambda_0 x_3) & -\cos(\lambda_0 x_3)  & 0
 \end{pmatrix},
\]
we obtain from a direct calculation
\[
\partial_t v_0+\div (v_0\otimes v_0)+\nabla p_0= \div \mathring R_0.
\]
Hence the triple $(v_0,p_0,\mathring R_0)$ is a solution to the Euler-Reynolds system \eqref{e:euler_reynolds}. 
Next, we set $I^{(0)}_0=[0,\sfrac38]$, $I^{(0)}_1=[\sfrac38,\sfrac58]$, $I^{(0)}_2=[\sfrac58,1]$, $N (0)=1$ and 
$j_0 (0) = j_0 (1) = j_0 (2) = 0$. 
We check that,
upon choosing $\lambda_0$ sufficiently large, all the requirements in Section \ref{ss:time_intervals} and \ref{ss:estimates} hold.

First, \eqref{e:regions_increase} is obvious. Next, since the smallest of the three intervals, $I^{(0)}_1$, has length $\frac{1}{4}$, \eqref{e:interval_constraint} becomes
\[
\frac{1}{16} \geq \lambda_0^{- (1-\beta_0) (b+1)/2 - b \beta_\infty (b-1)/2}\, ,
\]
and since the exponent in the latter power is negative, it just suffices to choose $\lambda_0$ sufficiently large.
Observe next that $V^{0}_0 = [\frac{3}{8},\frac{5}{8}]$ and thus \eqref{e:regions_est} is equivalent to
\[
\frac{1}{4} \leq \lambda_0^{1- b (\beta_\infty - \beta_0 + \sfrac{\eps}{4})}\, .
\]
Since
\[
1 - b \left( \beta_\infty - \beta_0 + \frac{\eps}{4}\right) \geq 1 - \frac{5}{4} \left(\frac{1}{3} + \frac{1}{4}\right) > 0\, ,
\] 
again \eqref{e:regions_est} is satisfied provided $\lambda_0$ is sufficiently large.
This concludes the verification of the conditions in Section \ref{ss:time_intervals}.

We next come to the conditions required in Section \ref{ss:estimates}. First observe that, since $V^{0}_j = \emptyset$ for $j>0$, these requirements amounts to:
\begin{itemize}
\item[(a)] The estimates \eqref{e:velocity_est}, \eqref{e:pressure_est}, \eqref{e:R_est} and \eqref{e:D_tR_est} with $q=0$ and $j=0$, which thus are equivalent to \eqref{e:g-velocity_est}, \eqref{e:g-pressure_est}, \eqref{e:g-R_est} and \eqref{e:g-D_tR_est} (with $q=0$);
\item[(b)] The requirement \eqref{e:spt_triple} with $q=0$;
\item[(c)] The estimates \eqref{e:C0_v_global} and \eqref{e:C_p_global} with $q=0$.
\end{itemize}
(b) is obvious. Observe also that, since $p_0=0$, \eqref{e:pressure_est} and \eqref{e:C_p_global} are also trivially satisfied. Next observe that
\[
\|v\|_0 + \lambda_0^{-1} \|v\|_1 + \lambda_0^{-2} \|v\|_2 \leq 3 \lambda_0^{-\beta_0}\, .
\]
Assuming, without loss of generality, that $M\geq 3$, \eqref{e:g-velocity_est} and \eqref{e:C0_v_global} are then
fullfilled.

It remains only to check \eqref{e:g-R_est} and \eqref{e:g-D_tR_est}.  First observe that
\[\|\mathring R_0\|_0+\lambda_0^{-1}\|\mathring R_0\|_1 + \lambda_0^{-2} \|\mathring{R}_0\|_2
\leq C \lambda_0^{-\beta_0-1}\, ,\]
where the constant depends only on the function $\nu_0$. Thus \eqref{e:g-R_est} is satisfied as soon
\[
C \lambda_0^{-\beta_0-1} \leq \lambda_1^{-2\beta_0} = \lambda_0^{-2\beta_0 b}
\]
Since $1 + \beta_0 - 2 \beta_0 b > 0$, again this follows upon choosing $\lambda_0$ sufficiently large.
Finally, 
\[
\| (\partial_t + v_0 \cdot \nabla) \mathring{R}_0\|_0 \leq C \lambda_0^{-\beta_0}\, ,
\]
where again the constant depends only upon $\nu_0$. Thus, \eqref{e:g-D_tR_est} needs only
\[
C \leq \lambda_0^{1- \beta_{-1} + \beta_0 - 2 b \beta_0}\, .
\]
It is again easy to see that the exponent
\[
1 - \beta_{-1} + \beta_0 - 2b \beta_0
\]
is positive, thus concluding the proof that the triple $(v_0, p_0, \mathring{R}_0)$ satisfies all the requirements to
start the iteration the scheme.

\subsection{Convergence to a nontrivial solution to the Euler equations}

We now apply Proposition \ref{p:inductive_step} to obtain a sequence of triples $(v_q,p_q,\mathring R_q)$.  From  \eqref{e:C_p_global}-\eqref{e:velocity_est} and interpolation we see that the sequence converges to a pair $(v,p)$ of continuous functions, with compact temporal support and which solve the Euler equations (in fact we have $v\in C^{\beta_0}$ and $p\in C^{2\beta_0}$).  Moreover, using \eqref{e:regions_est} and \eqref{e:betas}
\begin{align*}
\int_0^1 [v(\cdot,t)]_{\sfrac13-\eps}~dt&\leq \sum_q \int_0^1 [w_q(\cdot,t)]_{\sfrac13-\eps}~dt\\
&\leq \sum_q \int_0^1 \norm{w_q(\cdot,t)}_{0}^{1-\sfrac13+\eps}\norm{w_q(\cdot,t)}_{1}^{\sfrac13-\eps}~dt\\
&\leq \sum_{q=0}^{\infty}\sum_{j=0}^q \abs{V_j}\lambda_q^{\sfrac13-\eps-\beta_{j-1}}\\
&\leq \lambda_0\sum_{q=0}^{\infty}\sum_{j=0}^q\lambda_{q+1}^{\beta_{j}-\beta_{\infty}+\eps/4} \lambda_q^{\sfrac13-\eps-\beta_{j-1}}\\
&= \lambda_0\sum_{q=0}^{\infty}(q+1)\lambda_q^{\frac13-\beta_{\infty}+b\eps/4-\eps}\\
&\leq \lambda_0\sum_{q=0}^{\infty}(q+1)\lambda_q^{-\eps/4} \leq \lambda_0 \sum_{q=0}^\infty (q+1) \lambda_0^{-\eps b^q/4}\,,\\
\end{align*}
where in the last line we have used \eqref{e:choice_of_betainfty} and $b\leq \frac{5}{4}$. The final sum is obviously finite, since $b>1$.
An analogous calculation yields $p\in L^1_tC_x^{\sfrac23-2\eps}$. 

We only need to show that the solution is nontrivial. Fix then any $x\in \mathbb T^3$. Due to our estimates we have
\begin{align}
|v (x, \sfrac{1}{2})| \geq & |v_0 (x, \sfrac{1}{2})| - \sum_{i=1}^\infty \|v_{i} - v_{i-1}\|_0
\geq \lambda_0^{-\beta_0} - M \sum_{i=1}^\infty \lambda_0^{- \beta_0 b^i}\, .
\end{align}
However the constant $M$ is only geometric, whereas $\beta_0>0$ and $b>1$ are some fixed parameters.
Thus choosing $\lambda_0$ sufficiently large we obviously get $|v (x, 1/2)|>0$.

\appendix

\section{H\"older spaces}\label{s:hoelder}

In the following $m=0,1,2,\dots$, $\alpha\in (0,1)$, and $\beta$ is a multi-index. We introduce the usual (spatial) 
H\"older norms as follows.
First of all, the supremum norm is denoted by $\|f\|_0:=\sup_{\T^3\times [0,1]}|f|$. We define the H\"older seminorms 
as
\begin{equation*}
\begin{split}
[f]_{m}&=\max_{|\beta|=m}\|D^{\beta}f\|_0\, ,\\
[f]_{m+\alpha} &= \max_{|\beta|=m}\sup_{x\neq y, t}\frac{|D^{\beta}f(x, t)-D^{\beta}f(y, t)|}{|x-y|^{\alpha}}\, ,
\end{split}
\end{equation*}
where $D^\beta$ are {\em space derivatives} only.
The H\"older norms are then given by
\begin{eqnarray*}
\|f\|_{m}&=&\sum_{j=0}^m[f]_j\\
\|f\|_{m+\alpha}&=&\|f\|_m+[f]_{m+\alpha}.
\end{eqnarray*}
Moreover, we will write $[f (t)]_\alpha$ and $\|f (t)\|_\alpha$ when the time $t$ is fixed and the
norms are computed for the restriction of $f$ to the $t$-time slice.

Recall the following elementary inequalities:
\begin{equation}\label{e:Holderinterpolation}
[f]_{s}\leq C\bigl(\varepsilon^{r-s}[f]_{r}+\varepsilon^{-s}\|f\|_0\bigr)
\end{equation}
for $r\geq s\geq 0$, $\eps>0$, and 
\begin{equation}\label{e:Holderproduct}
[fg]_{r}\leq C\bigl([f]_r\|g\|_0+\|f\|_0[g]_r\bigr)
\end{equation}
for any $r\geq 0$. From \eqref{e:Holderinterpolation} with $\eps=\|f\|_0^{\frac1r}[f]_r^{-\frac1r}$ we obtain the 
standard interpolation inequalities
\begin{equation}\label{e:Holderinterpolation2}
[f]_{s}\leq C\|f\|_0^{1-\frac{s}{r}}[f]_{r}^{\frac{s}{r}}.
\end{equation}

Next we collect two classical estimates on the H\"older norms of compositions. These are also standard, for instance
in applications of the Nash-Moser iteration technique.

\begin{proposition}\label{p:chain}
Let $\Psi: \Omega \to \mathbb R$ and $u: \R^n \to \Omega$ be two smooth functions, with $\Omega\subset \R^N$. 
Then, for every $m\in \mathbb N \setminus \{0\}$ there is a constant $C$ (depending only on $m$,
$N$ and $n$) such that
\begin{align}
\left[\Psi\circ u\right]_m &\leq C ([\Psi]_1 \|Du\|_{m-1}+\|D\Psi\|_{m-1} \|u\|_0^{m-1} \|u\|_m)\label{e:chain0}\\
\left[\Psi\circ u\right]_m &\leq C ([\Psi]_1 \|Du\|_{m-1}+\|D\Psi\|_{m-1} [u]_1^{m} )\, .
\label{e:chain1}
\end{align} 
\end{proposition}

\section{Estimates for transport equations}\label{s:transport_equation}

In this section we recall some well known results regarding smooth solutions of
the \emph{transport equation}:
\begin{equation}\label{e:transport}
\left\{\begin{array}{l}
\partial_t f + v\cdot  \nabla f =g,\\ 
f|_{t_0}=f_0,
\end{array}\right.
\end{equation}
where $v=v(t,x)$ is a given smooth vector field. 
We denote the advective derivative $\partial_t+v\cdot \nabla$ by $D_t$. We will consider solutions
on the entire space $\R^3$ and treat solutions on the torus simply as periodic solution in $\R^3$.

\begin{proposition}\label{p:transport_derivatives}
Assume $t>t_0$. Any solution $f$ of \eqref{e:transport} satisfies
\begin{align}
\|f (t)\|_0 &\leq \|f_0\|_0 + \int_{t_0}^t \|g (\tau)\|_0\, d\tau\,,\label{e:max_prin}\\
[f(t)]_1 &\leq [f_0]_1e^{(t-t_0)[v]_1} + \int_{t_0}^t e^{(t-\tau)[v]_1} [g (\tau)]_1\, d\tau\,,\label{e:trans_est_0}
\end{align}
and, more generally, for any $N\geq 2$ there exists a constant $C=C_N$ so that
\begin{align}
[f (t)]_N & \leq \Bigl([ f_0]_N + C(t-t_0)[v]_N[f_0]_1\Bigr)e^{C(t-t_0)[v]_1}+\nonumber\\
&\qquad +\int_{t_0}^t e^{C(t-\tau)[v]_1}\Bigl([g (\tau)]_N + (t-\tau) [v ]_N [g (\tau)]_{1}\Bigr)\,d\tau.
\label{e:trans_est_1}
\end{align}
Define $\Phi (t, \cdot)$ to be the inverse of the flux $X$ of $v$ starting at time $t_0$ as the identity
(i.e. $\frac{d}{dt} X = v (X,t)$ and $X (x, t_0 )=x$). Under the same assumptions as above:
\begin{align}
\norm{D\Phi (t) -\Id}_0&\leq e^{(t-t_0)[v]_1}-1\,,  \label{e:Dphi_near_id}\\
[\Phi (t)]_N &\leq C(t-t_0)[v]_Ne^{C(t-t_0)[v]_1} \qquad \forall N \geq 2\, .\label{e:Dphi_N}
\end{align}
\end{proposition}

\section{Constantin-E-Titi commutator estimate}\label{s:CET}

Finally, we recall the quadratic commutator estimate from \cite{ConstantinETiti} 
(cf. also with \cite[Lemma 1]{CDSz}):

\begin{proposition}\label{p:CET}
Let $f,g\in C^{\infty}(\T^3\times\T)$ and $\psi$ a standard radial smooth and compactly supported kernel. For any $r\geq 0$ we have the estimate
\[
\Bigl\|(f*\psi_\ell)( g*\psi_\ell)-(fg)*\psi_\ell\Bigr\|_r\leq C\ell^{2-r}\|f\|_1\|g\|_1\, ,
\]
where the constant $C$ depends only on $r$.
\end{proposition}

\section{Schauder Estimates}
\label{s:schauder}

We recall here the following consequences of the classical Schauder estimates  (cf.\ 
\cite[Proposition 5.1]{DS3}). 

\begin{proposition}\label{p:GT}
For any $\alpha\in (0,1)$ and any $m\in \N$ there exists a constant $C (\alpha, m)$ with the following properties.
If $\phi, \psi: \T^3\to \R$ are the unique solutions of
\[
\left\{\begin{array}{l}
\Delta \phi = f\\ \\
\fint \phi =0
\end{array}\right.
\qquad\qquad 
\left\{\begin{array}{l}
\Delta \psi = {\rm div}\, F\\ \\
\fint \psi =0
\end{array}\right. \, ,
\]
then
\begin{equation}\label{e:GT_laplace}
\|\phi\|_{m+2+\alpha} \leq C (m, \alpha) \|f\|_{m, \alpha}
\quad\mbox{and}\quad \|\psi\|_{m+1+\alpha} \leq C (m, \alpha) \|F\|_{m, \alpha}\, .
\end{equation}
Moreover we have the estimates
\begin{eqnarray}
&&\|\mathcal{R} v\|_{m+1+\alpha} \leq C (m,\alpha) \|v\|_{m+\alpha}\label{e:Schauder_R}\\
&&\|\mathcal{R} ({\rm div}\, A)\|_{m+\alpha}\leq C(m,\alpha) \|A\|_{m+\alpha}\label{e:Schauder_Rdiv}
\end{eqnarray}
\end{proposition}

\section{Stationary phase lemma}\label{s:stationary}

We recall here the following simple facts. For completeness we include the proof given in \cite{DS3}.

\begin{proposition}\label{p:stat_phase}
(i) Let $k\in\Z^3\setminus\{0\}$ and $\lambda\geq 1$ be fixed. 
For any $a\in C^{\infty}(\T^3)$ and $m\in\N$ we have
\begin{equation}\label{e:average}
\left|\int_{\T^3}a(x)e^{i\lambda k\cdot x}\,dx\right|\leq \frac{[a]_m}{\lambda^m}.
\end{equation}

(ii) Let $k\in\Z^3\setminus\{0\}$ be fixed. For a smooth vector field $a\in C^{\infty}(\T^3;\R^3)$ let 
$F(x):=a(x)e^{i\lambda k\cdot x}$. Then we have
\begin{equation}\label{e:R(F)}
\|\RR(F)\|_{\alpha}\leq \frac{C}{\lambda^{1-\alpha}}\|a\|_0+\frac{C}{\lambda^{m-\alpha}}[a]_m+\frac{C}{\lambda^m}[a]_{m+\alpha},
\end{equation}
where $C=C(\alpha,m)$.
\end{proposition}

\section{One further commutator estimate}\label{s:commutator}

\begin{proposition}\label{p:commutator}
Let $k\in\Z^3\setminus\{0\}$ be fixed. For any smooth vector field $a\in C^\infty  (\T^3;\R^3)$ and any smooth function $b$, if we set $F(x):=a(x)e^{i\lambda k\cdot x}$, we then have
\begin{align}
&\|[b, \mathcal{R}] (F)\|_\alpha \leq  \bar{C} \lambda^{\alpha-2} \|b\|_1 \|a\|_0
+ C \lambda^{\alpha-m} \left(\|a\|_{m-1+\alpha} \|b\|_{1+\alpha} + \|a\|_\alpha
\|b\|_{m+\alpha}\right)\label{e:main_est_commutator}
\end{align}
where $C=C(\alpha,m)$ and $\bar{C}$ is a universal constant.
\end{proposition}

\bibliographystyle{acm}
\bibliography{eulerbib}

\end{document}